\documentclass[10pt]{article}

\usepackage{fullpage, amsmath}
\usepackage{amsthm,amsfonts,url,amssymb}
\usepackage{mathrsfs}
\usepackage{amscd}
\usepackage{color}
\usepackage{chemarrow}
\usepackage{pictexwd,dcpic}
\usepackage[textwidth=18cm,centering]{geometry}
\usepackage{extarrows}
\usepackage[bookmarks=false,linkcolor=red,anchorcolor=green,citecolor=blue]{hyperref}
\usepackage{enumitem}
\usepackage{lipsum}
\usepackage[all,cmtip]{xy}
\usepackage{leftidx}
\usepackage[version=3]{mhchem}
\usepackage{makeidx}
\usepackage{bm} 
\usepackage{stmaryrd}

\makeindex
\newtheorem{thm}{Theorem}[section]
\newtheorem{cor}[thm]{Corollary}
\newtheorem{lem}[thm]{Lemma}
\newtheorem{pro}[thm]{Proposition}
\newtheorem{dfn}[thm]{Definition}
\newtheorem{hypothesis}[thm]{Hypothesis}
\newtheorem{rmk}[thm]{Remark}

\newtheorem{exmp}[thm]{Example}

\newtheorem{expect}[thm]{Expectation}

\DeclareMathOperator{\Gal}{Gal}
\DeclareMathOperator{\GL}{GL}

\DeclareMathOperator{\Spf}{Spf}

\newcommand{\ccyc}{{\epsilon}}

\DeclareMathOperator{\coker}{coker}

\DeclareMathOperator{\diag}{diag}

\DeclareMathOperator{\im}{Im}

\DeclareMathOperator{\rank}{rank}

\newcommand{\rig}{{\rm rig}}
\DeclareMathOperator{\rk}{rk}

\newcommand{\gal}{{\rm Gal}}

\newcommand{\fil}{{\rm Fil}}

\newcommand{\ind}{{\rm Ind}}
\newcommand{\homo}{{\rm Hom}}
\newcommand{\EndO}{{\rm End}}
\newcommand{\ext}{{\rm Ext}}
\newcommand{\GLN}{{\rm GL}}

\newcommand{\ana}{{\rm an}}
\newcommand{\st}{{\rm St}}

\newcommand{\op}{{\overline{\mathbf{P}}}}

\newcommand{\ob}{{\overline{\mathbf{B}}}}

\newcommand{\unr}{{\rm unr}}

\newcommand{\hH}{{\mathrm{H}}}

\newcommand{\ra}{{\longrightarrow}}
\DeclareMathOperator{\val}{val}
\newcommand{\pdr}{{\rm pdR}}

\newcommand{\soc}{{\rm soc}}

\newcommand{\sm}{{\rm sm}}

\newcommand{\spf}{{\rm Spf}}

\newcommand{\df}{{\mathrm{DF}}}
\newcommand{\wdre}{{\textbf{r}}}

\newcommand{\rec}{{\rm rec}_L}
\newcommand{\lalg}{{\rm lalg}}

\newcommand{\Art}{{\rm Art}}

\newcommand{\wt}{{\rm wt}}

\newcommand{\pr}{{\rm pr}}
\newcommand{\dr}{{\rm dR}}

\newcommand{\gr}{{\rm gr}}
\newcommand{\hpi}{{{\mathbf{h}}}}
\newcommand{\Dpik}{\mathbf{D}}

\newcommand{\univ}{{\rm univ}}
\newcommand{\loc}{{\rm loc}}

		\newcommand{\BOne} {{\mathchoice{\hbox{\rm1\kern-2.7pt l\kern.9pt}}
				{\hbox{\rm1\kern-2.7pt l\kern.9pt}}
				{\hbox{\scriptsize\rm1\kern-2.3pt l\kern.4pt}}
				{\hbox{\scriptsize\rm1\kern-2.4pt l\kern.5pt}}}}
		
		\newcommand{\BA}{{\mathbb{A}}}

		\newcommand{\BE}{{\mathbb{E}}}
		
		\newcommand{\BG}{{\mathbb{G}}}

		\newcommand{\BP}{{\mathbb{P}}}
		\newcommand{\BQ}{{\mathbb{Q}}}

		\newcommand{\BU}{{\mathbb{U}}}
		
		\newcommand{\BW}{{\mathbb{W}}}

		\newcommand{\BZ}{{\mathbb{Z}}}
		
		\newcommand{\ba}{{\mathbf{a}}}
		\newcommand{\bA}{{\mathbf{A}}}
		\newcommand{\bB}{\mathbf{B}}
		\newcommand{\bC}{{\mathbf{C}}}

		\newcommand{\be}{{\mathbf{e}}}
		\newcommand{\bF}{{\mathbf{F}}}
		\newcommand{\bG}{{\mathbf{G}}}

		\newcommand{\bL}{\mathbf{L}}
		\newcommand{\bM}{{\mathbf{M}}}
		\newcommand{\bN}{\mathbf{N}}
		
		\newcommand{\bP}{\mathbf{P}}
		\newcommand{\bQ}{{\mathbf{Q}}}

		\newcommand{\bT}{{\mathbf{T}}}

		\newcommand{\bW}{{\mathbf{W}}}

		\newcommand{\bZ}{{\mathbf{Z}}}

		\newcommand{\bh}{{\mathbf{h}}}
		
		\newcommand{\bd}{{\mathbf{d}}}

		\newcommand{\cL}{\mathcal L}
		\newcommand{\co}{\mathcal O}

		\newcommand{\cR}{\mathcal R}
		\newcommand{\cH}{\mathcal H}
		\newcommand{\cC}{\mathcal C}
		\newcommand{\cS}{\mathcal S}
		
		\newcommand{\cI}{\mathcal I}
		
		\newcommand{\cT}{\mathcal T}
		\newcommand{\cM}{\mathcal M}
		\newcommand{\cF}{\mathcal F}
		
		\newcommand{\cE}{\mathcal E}
		\newcommand{\cG}{\mathcal G}
		
		\newcommand{\cO}{\mathcal O}

		\newcommand{\FX}{{\mathfrak{X}}}

		\newcommand{\fa}{{\mathfrak{a}}}
		\newcommand{\fb}{{\mathfrak{b}}}

		\newcommand{\fg}{{\mathfrak{g}}}

		\newcommand{\fl}{{\mathfrak{l}}}
		\newcommand{\fm}{{\mathfrak{m}}}
		\newcommand{\fn}{{\mathfrak{n}}}
		
		\newcommand{\fp}{{\mathfrak{p}}}

		\newcommand{\ft}{{\mathfrak{t}}}

		\newcommand{\fz}{{\mathfrak{z}}}

		\newcommand{\sW}{\mathscr W}
		\newcommand{\sL}{\mathscr L}

\begin{document}	
			
			\title{\textbf{Towards the $p$-adic Hodge parameters in semistable representations of $\GLN_n(\bQ_p)$}}

			\author{Yiqin He
				\thanks{Morningside Center of Mathematics, Chinese Academy of Sciences,\;No. 55, Zhongguancun East Road, Haidian District, Beijing 100190, P.R. China,\;E-mail address:\texttt{\;heyiqin@amss.ac.cn}
			}}
			
			\date{}
			\maketitle

			\begin{abstract}
				Let $\rho_p$ be an $n$-dimensional non-critical semistable $p$-adic Galois representation of the absolute Galois group of $\bQ_p$ with regular Hodge--Tate weights.\;Let $\Dpik$ be the associated $(\varphi,\Gamma)$-module over the Robba ring.\;By combining Ding's and Breuil--Ding's methods for the crystalline case with Qian's computation of higher extension groups of locally analytic generalized Steinberg representations,\;we capture the full information of the $p$-adic Hodge parameters of $\rho_p$ on the automorphic side by considering several Steinberg subquotients of $\Dpik$ and the ``crystalline'' Hodge parameters between them.\;These results also admit geometric and Lie-algebraic reformulations on flag varieties related to the moduli space of Hodge parameters.\;We then construct an explicit locally analytic representation $\pi_{1}(\rho_p)$ and explicitly describe which Hodge-parameters information of $\rho_p$ it determines.\;In particular,\;if the monodromy rank of $\rho_p$ is at most $1$,\;$\pi_{1}(\rho_p)$ determines $\rho_p$.\;When $\rho_p$ comes from a $p$-adic automorphic representation,\;we show that $\pi_{1}(\rho_p)$ is a subrepresentation of the $\GLN_n(\bQ_p)$-representation globally associated to $\rho_p$,\;under mild hypotheses.\;Although it is still difficult to construct an explicit representation $\pi_{1}(\rho_p)$ that determines $\rho_p$,\;our results provide new evidence for the $p$-adic Langlands program in general semistable cases and demonstrate the broad applicability of Ding's,\;Breuil--Ding's,\;and Qian's methods.\;

			\end{abstract}

			{\hypersetup{linkcolor=black}
				\tableofcontents}
			
			\numberwithin{equation}{section}
			
			\numberwithin{thm}{section}

			\setlength{\baselineskip}{13pt}
			
			\section{Introduction}
			
			This paper aims to extend the discussion of crystabelline (resp.,\;crystalline) $p$-adic Galois representations in \cite{ParaDing2024} (resp.,\;the recent work \cite{BDcritical25}) to the semistable case.\;
	
				Let $\rho_p:\gal_{\bQ_p}\rightarrow \GLN_n(E)$ be a de Rham $p$-adic Galois representation,\;where $\gal_{\bQ_p}$ is the absolute Galois group of $\bQ_p$ and $E$ is a finite extension of $\bQ_p$.\;By Fontaine's theory,\;we can attach an $n$-dimensional Weil--Deligne representation $\wdre(\rho_p)$ and thus an irreducible smooth representation $\pi_{\mathrm{sm}}(\rho_p)$ of $\GLN_n(\bQ_p)$ over $E$ (via the classical local Langlands correspondence).\;Assume that $\rho_p$ has regular Hodge--Tate weights $\bh=(\bh_1>\cdots>\bh_n)$.\;Then the locally algebraic representation $\pi_{\lalg}(\rho_p):=\pi_{\mathrm{sm}}(\rho_p)\otimes_EL(\bh-\theta)$ is expected to be the locally algebraic subrepresentation of the conjectural locally analytic representation $\pi_{\ana}(\rho_p)$ via the $p$-adic local Langlands correspondence,\;where $\theta=(0,-1,\cdots,1-n)$ and $L(\bh-\theta)$ is the algebraic representation of $\GLN_n(\bQ_p)$ with highest weight $\bh-\theta$.\;The representation $\pi_{\lalg}(\rho_p)$ only encodes the information in the $F$-semisimplification of $\wdre(\rho_p)$ and $\bh$,\;and therefore misses the Hodge-filtration information of $\rho_p$.\;A basic problem (and starting point) in the locally analytic $p$-adic local Langlands program is to recover the Hodge filtration from locally ${\BQ}_p$-analytic representations of $\GLN_n({\bQ}_p)$.\;

			In this paper,\;we study this question for a non-critical semistable Galois representation $\rho_p$.\;By Fontaine's theory,\;$\rho_p$ is determined by the associated filtered Deligne--Fontaine module $(\varphi,N,D,\fil_H^{\bullet}(D))$,\;where $D:=D_{\mathrm{st}}(\rho_p)\cong D_{\dr}(\rho_p)$ and $\fil^{\bullet}_{H}(D)$ is  the Hodge filtration.\;There exist integers $s,l_1,\cdots,l_s$ and scalars $\alpha_1,\cdots,\alpha_s\in E^{\times}$,\;such that the semisimplification of the underlying Deligne--Fontaine module $(\varphi,N,D)$ is
			$\oplus_{i=1}^s\oplus_{j=1}^{l_i}Ee_{i,j}$,\;where $\varphi(e_{i,j})=\alpha_ip^{j-1}$ for $1 \leq i\leq s$ and $N(e_{i,j})=e_{i,j-1}$ (resp.,\;$N(e_{i,j})=0$) when $1<j\leq l_i$ (resp.,\;$j=1$).\;Thus 
			\[\underline{\phi}:=(\alpha_1,\alpha_1p,\cdots,\alpha_1p^{l_1-1},\cdots,\alpha_s,\alpha_sp,\cdots,\alpha_sp^{l_s-1})=(\phi_i)_{1\leq i\leq n}\]
			are the $\varphi$-eigenvalues on $D$ (the ordering satisfies the condition that if $\alpha_j=\alpha_ip^{l_i}$,\;then $j=i+1$).\;We assume that $\phi_i\neq \phi_j$ for any $i\neq j$.\;We relabel the  basis $\{e_{1,j}\}_{1\leq j\leq l_1},\cdots,\{e_{s,j}\}_{1\leq j\leq l_s}$ as
			$e_1,\cdots,e_n$ (so that $\varphi(e_i)=\phi_ie_i$).\;Let $S_0\subseteq \Delta$ be the set of $i\in S_0$ such that $N(e_{i+1})=e_{i}$,\;which describes the monodromy type of $D$.\;
			
				Let $\bT$ (resp.,\;$\bB$) be the torus of diagonal matrices (resp.,\;the Borel subgroup of upper triangular matrices) in $\GLN_n$.\;Let $\Delta:=\{1,\cdots,n-1\}$ be the set of simple roots of $\GLN_n$.\;For $I\subseteq \Delta$,\;we attach the standard Levi subgroup (resp.,\;parabolic subgroup) $\bL_I\supseteq \bT$ (resp.,\;$\bP_I\supseteq \bL_I\bB$).\;Let $\bZ_{I}$ be the center of $\bL_{I}$,\;$\bN_I$ be the unipotent radical of $\bP_I$ (so that $\bP_I=\bL_I\bN_I$).\;Let $\sW_n\cong S_n$ be the Weyl group of $\GLN_n$.\;For $I\subset \Delta$,\;define $\sW_{I}$ to be the subgroup of $\sW_{n}$ generated by simple reflections $s_{i}$ with $i\in I$.\;For $I\subseteq \Delta$,\;let $\sW^{I}_n$ be the set of minimal length representatives in $\sW_{I}\backslash \sW_n$.\;For $A\in\{E,E[\epsilon]/\epsilon^2\}$,\;let $\cR_A$ be the Robba ring over $\bQ_p$ with coefficients in $A$.\;We write $\cR_{A,L}(\delta_A)$ for the rank-one $(\varphi,\Gamma)$-module over $\cR_{A,L}$ associated to a continuous character $\delta_A:L^\times\rightarrow A^\times$.\;

			\subsection{Galois side}

			Let $\Dpik:=D_{\rig}(\rho_p)$ be the associated $(\varphi,\Gamma)$-module over $\cR_E$ of rank $n$.\;Under the basis $e_{1},e_{2},\cdots,e_{n}$ of $D$, $\fil^{\bullet}_{H}(D)$ is parameterized by an element $\underline{\sL}(\Dpik)$ in $\bZ_{S_0}\backslash\GLN_{n}/\bB$, which we call the \textit{$p$-adic Hodge parameters} of $\Dpik$. The $(\varphi,N)$-stable complete flags in $D_{\mathrm{st}}(\rho_p)$ are
			$\cF_u:Ee_{u^{-1}(1)}\subseteq Ee_{u^{-1}(1)}\oplus Ee_{u^{-1}(2)}\subseteq \cdots\subseteq \oplus_{i=1}^nEe_{u^{-1}(n)}=D_{\mathrm{st}}(\rho_p)$ for $u\in \sW_n^{S_0}$. We say that $\rho_p$ is \textit{non-critical} if $\fil^{\bullet}_{H}(D)$ is in relative general position with respect to all $|\sW_n^{S_0}|$ $(\varphi,N)$-stable flags; this is equivalent to saying that (note that $w_{0}$ is the longest element in $\sW_n$)
			\[\underline{\sL}(\Dpik)\in \Phi_{\mathrm{nc},\Delta}(S_0):=\bigcap_{u\in \sW_n^{S_0}} \bZ_{S_0}\backslash u^{-1}\bB w_0\bB/\bB,\]
			$\Phi_{\mathrm{nc},\Delta}(S_0)$ is the moduli space of non-critical $p$-adic Hodge parameters with monodromy type $S_0$.\;Our goal is to detect $\Phi_{\mathrm{nc},\Delta}(S_0)$ on the automorphic side.\;There are two extreme cases in the semistable setting.
			
			\begin{itemize}
				\item[(1)] (\textit{Crystalline case:\;$S_0=\emptyset$}) Put $\Phi_{\mathrm{nc},n}^{\mathrm{{cr}}}:=\Phi_{\mathrm{nc},\Delta}(\emptyset)$. When $\rho_p$ is generic (i.e.,\;$\phi_i\phi_j^{-1}\neq p$ for $i\neq j$), this problem was first solved by Ding \cite{ParaDing2024}, and further developed in the recent work of Breuil--Ding \cite{BDcritical25}. Ding constructs an explicit locally analytic $\GLN_n(\bQ_p)$-representation $\pi_1(\Dpik)$ that determines $\underline{\sL}(\Dpik)\in\Phi_{\mathrm{nc},n}$.\;In \cite{HEparaforsemitable}, we extend this theory to study the non-critical and generic potentially crystalline case.
				\item[(2)] (\textit{Steinberg case:\;$S_0=\Delta$}) $D$ has maximal monodromy rank (i.e.,\;$N^{n-1}\neq 0$),\;and  $\Phi^{\mathrm{ST}}_{\mathrm{nc},n}:=\Phi_{\mathrm{nc},\Delta}(\Delta)= \bG_m\backslash \bB w_0\bB/\bB$.\;This case was first discussed by Breuil for $\GLN_2(\bQ_p)$, then by Ding \cite{2015Ding} for $\GLN_2(L)$, by \cite{schraen2011GL3}, \cite{HigherLinvariantsGL3(Qp)}, \cite{breuil2019ext1} for $\GLN_3(\bQ_p)$, and by \cite{2019DINGSimple} in general for $\GLN_{n}(\bQ_p)$.\;Qian's recent work \cite{wholeLINV} computes the higher extension groups of locally analytic generalized Steinberg representations in detail.\;He studied the cup-product structure and constructed a Coxeter filtration indexed by Coxeter elements in $\sW_n$ on these higher extension groups. He then defined the full version of the so-called Breuil--Schraen $\sL$-invariants, which capture $\Phi^{\mathrm{ST}}_{\mathrm{nc},n}$.
			\end{itemize}

			For general semistable case, we prove that $\Phi_{\mathrm{nc},\Delta}(S_0)$ can be captured on the automorphic side by combining and further developing the methods of Ding \cite{ParaDing2024}, Breuil-Ding \cite{BDcritical25} in $(1)$ and Qian \cite{wholeLINV} in $(2)$. More precisely, our key new observation is that \textit{we can recover $\underline{\sL}(\Dpik)$ through suitable choices of Steinberg subquotients of $\Dpik$ and ``crystalline'' Hodge parameters between different Steinberg subquotients or crystalline subquotients of $\Dpik$}.\;We refer to Section \ref{philoforParameter} for a typical example  for $\GLN_3(\bQ_p)$.\;See Theorem \ref{capturetype0intro}, Theorem \ref{capturetype1intro} and Theorem \ref{capturetype2intro} for the  precise statements.\;In the sequel,\;we use $\bullet-\bullet$ to denote an extension of two objects (for example,\;$(\varphi,\Gamma)$-modules or $G$-representations),\;where the first (resp.,\;second) object is the subobject (resp.,\;quotient).\;
			
			For $\alpha\in E^\times$,\;denote by $\unr(\alpha)$ the unramified character of $\bQ_p^\times$ sending $p$ to $\alpha$.\;For any $u\in \sW_n^{S_0}$,\;let $\cF_u:=[\cR_E(\unr(\phi_{u^{-1}(1)})z^{\bh_1})-\cdots-\cR_E(\unr(\phi_{u^{-1}(n)})z^{\bh_n})]$ be the triangulation of $\Dpik$ associated to the $(\varphi,N)$-stable complete flag $\cF_u$ on $D_{\mathrm{st}}(\rho_p)$.\;Let $S_0(u)\subseteq \Delta$ be the set of $i\in \Delta$ such that the rank-$2$ subquotient $[\cR_E(\phi_{u^{-1}(i)}z^{\bh_i})-\cR_E(\phi_{u^{-1}(i+1)}z^{\bh_{i+1}})]$ of $\Dpik$ is semistable non-crystalline.\;There is a $\bP_{S_0(u)}$-parabolic filtration by saturated $(\varphi,\Gamma)$-submodules of $\Dpik$ over $\cR_{E}$ ($f_u:=|\Delta\backslash S_0(u)|+1$):
			\[\cF_{S_0(u)}=\fil_{\bullet}^{\cF_{S_0(u)}}\Dpik: \ 0 =\fil_0^{\cF_{S_0(u)}}\Dpik \subsetneq \fil_1^{\cF_{S_0(u)}}\Dpik \subsetneq \cdots \subsetneq \fil_{f_u}^{\cF_{S_0(u)}} \Dpik=\Dpik\]
			 such that $E_{u,i}:=\gr_i^{\cF_{S_0(u)}}\Dpik$ are all Steinberg for $1\leq i\leq f_u$.\;Equivalently,\;we identity $\cF_{S_0(u)}$ with  the following successive  extension of $\{E_{u,i}\}$ for  $1\leq i\leq f_u$:
						\begin{equation}
				\Dpik=[E_{u,1}-E_{u,2}-\cdots-E_{u,f_u}],
			\end{equation}
			 In particular, if $u=1$, we write the $\bP_{S_0}$-parabolic filtration  with
			\begin{equation}
				\Dpik=[E_{1}-E_{2}-\cdots-E_{s}],\;E_i:=E_{1,i}.
			\end{equation}
			for simplicity.\;For $1\leq r<q\leq s$, let $E_r^q$ be the unique subquotient $[E_{r}-E_{r+1}-\cdots-E_{q}]$. 
			For $1\leq q\leq s$ and $1\leq t\leq l_q$,\;let $E_q^{(t)}$ be the unique rank-$t$ (Steinberg) $(\varphi,\Gamma)$-submodule of $E_q$,\;and $F^{(t)}_q=E_q/E_q^{(t)}$ be the unique (Steinberg) quotient of $E_q$ of rank-$(l_q-t)$.\;For $1\leq i\leq s$, put $t_0=0$ and $t_i=\sum_{j=1}^il_{j}$ (so that $t_s=n$).\;Note that $\rank(E_j)=l_j$ and $\rank(E_r^q)=t_q-t_{r-1}$.\;
			
			For $1\leq r<q\leq s$,\;$1\leq t\leq l_q$ and  $1\leq t'\leq l_r$,\;consider the following different parabolic filtrations of $\Dpik$ (for any subquotient $B$ of $\Dpik$, we use $B'$ to denote another $(\varphi,\Gamma)$-module such that $B[1/t]=B'[1/t]$.):
			\begin{equation}\label{lqErefinementIntro}
				\begin{aligned}
					&E_r^q=[E_r-E_{r+1}^{q-1}-E_{q}]=[(E_q^{(t)})'-E_r'-(E_{r+1}^{q-1})'-(E_q/E_q^{(t)})]=[E^{(t')}_r-(E_{r+1}^{q-1})'-E_q'-(F^{(t')}_r)'].
				\end{aligned}
			\end{equation}
			We denote $E_r'$  by $E_r^{[q,t]}$ (resp.\;$E_q'$  by $E_q^{[r,t']}$) to indicate the choice of $t$ (resp.\;$t'$). By the non-critical assumption, the Hodge--Tate weights of $E_r$ (resp., $E_r^{[q,t]}$) are $(\bh_{t_{r-1}+1},\cdots,\bh_{t_r})$ (resp., $(\bh_{t_{r-1}+1+t},\cdots,\bh_{t_r+t})$).\;The Hodge--Tate weights of $E_q$ (resp., $E_q^{[r,t']}$) are $(\bh_{t_{q-1}+1},\cdots,\bh_{t_q})$ (resp., $(\bh_{t_{q-1}+1-(l_r-t')},\cdots,\bh_{t_q-(l_r-t')})$).
			
			For $1\leq r\leq s$, put $\cC_{r}^r(\Dpik)=\{E_r\}$. Put
			\[\cC_{\mathrm{ST}}(\Dpik):=\bigcup_{1\leq r\leq q\leq s}\cC_{r}^q(\Dpik),\quad \cC_{r}^q(\Dpik):=\left\{E_r^{[q,t]},E_q^{[r,t']}\right\}_{1\leq t\leq l_q,1\leq t'\leq l_r}\]
			
			\begin{thm} [Theorem \ref{twoblocksPare1} and $\GLN_{3}(\bQ_p)$-case of Proposition \ref{twoblocksPare2}]\label{capturetype0intro} If $s=2$, $\cC_{\mathrm{ST}}(\Dpik)$ determines $\underline{\sL}(\Dpik)$. If $s=3$, $\cC_{\mathrm{ST}}(\Dpik)$ determines $\underline{\sL}(\Dpik)$ up to one Hodge parameter.
			\end{thm}
			In general, when $s>2$, besides $\cC_{\mathrm{ST}}(\Dpik)$, we need extra parameters coming from crystalline subquotients of $\Dpik$ or ``crystalline'' Hodge parameters between different Steinberg subquotients to recover  $\underline{\sL}(\Dpik)$.\;\textit{This phenomenon is parallel to the generic crystalline case in \cite{ParaDing2024} (a block version of \cite{ParaDing2024}).}\;In precise, for $1\leq r<q\leq s$,  the two parabolic filtrations $[E^{q-1}_{r}-E_q]=[E_q'-(E^{q-1}_{r})']$ of $E_r^q$ induce a map:
			\[\iota^{r,q}_{\Dpik}:E^{q-1}_{r}\hookrightarrow E_r^q\twoheadrightarrow (E^{q-1}_{r})',\]
Secondly, $\Dpik$ admits a unique maximal crystalline $(\varphi,\Gamma)$-quotient
			\[\Dpik^{\mathrm{cr}}:=[\cR_E(\unr(\alpha_1p^{l_1-1})z^{\bh_{n-s+1}})-\cR_E(\unr(\alpha_2p^{l_2-1})z^{\bh_{n-s+2}})-\cdots-\cR_E(\unr(\alpha_sp^{l_s-1})z^{\bh_{n}})]. \]
			Now we state our main theorems.
			
			\begin{thm}[Theorem \ref{capturetype1}]\label{capturetype1intro} $\Dpik$ is uniquely determined by $\cC_{\mathrm{ST}}(\Dpik)$ and $\Dpik^{\mathrm{cr}}$.
			\end{thm}
			\begin{thm}[Theorem \ref{capturetype1diamond}]\label{capturetype2intro} $\Dpik$ is uniquely determined by $\cC_{\mathrm{ST}}(\Dpik)$ and $\{\iota^{r,q}_{\Dpik}\}_{1\leq r<q\leq s}$.
			\end{thm}
			
			By quotienting $\bN_{S_0(u)}$, we get a natural map:
			\[p_{S_0(u)}:\bB w_0\bB/\bB\cong \bN_{\emptyset} \rightarrow (\bB\cap {\bL}_{S_0(u)}) w_{S_0(u),0}(\bB\cap {\bL}_{S_0(u)})/(\bB\cap {\bL}_{S_0(u)})\cong \bN_{\emptyset}\cap {\bL}_{S_0(u)},\]
			where $w_{S_0(u),0}$ is the longest element in $\sW_{S_0(u)}$. We thus obtain a well-defined morphism:
			\begin{equation*}
				\begin{aligned}
					p_{\mathrm{ST}}:\Phi_{\mathrm{nc},\Delta}(S_0)&\rightarrow \prod_{u\in \sW_n^{S_0}}\Phi_{\mathrm{nc},S_0(u)}(S_0(u)),\\
					[g\bB]&\mapsto ([p_{S_0(u)}(ug\bB)])_{u\in \sW_n^{S_0}},\;\Big(\text{equivalently},\;\underline{\sL}(\Dpik)\mapsto \prod_{u\in \sW_n^{S_0}}(\underline{\sL}(E_{u,1}),\cdots,\underline{\sL}(E_{u,f_u}))\Big).
				\end{aligned}
			\end{equation*}
			Note that $\Phi_{\mathrm{nc},S_0(u)}(S_0(u))=\bZ_{S_0(u)}\backslash (\bB\cap {\bL}_{S_0(u)}) w_{0,S_0(u)}(\bB\cap {\bL}_{S_0(u)})/(\bB\cap {\bL}_{S_0(u)})$.\;Extracting the $p$-adic Hodge parameters of $\Dpik^{\mathrm{cr}}$  defines another morphism $p_{\mathrm{cr}}:\Phi_{\mathrm{nc},\Delta}(S_0)\rightarrow \Phi_{\mathrm{nc},s}^{\mathrm{{cr}}}$.\;In the language of flag varieties,\;an equivalent description of Theorem \ref{capturetype1intro} is (see Remark \ref{rmkforlastthmintro}-$(3)$ for the reinterpretation of Theorem \ref{capturetype2intro} in terms of $E$-Lie algebras):
			
			\begin{thm} \label{parametersmaininjintro} The morphism $\Phi_{\mathrm{nc},\Delta}(S_0)\xrightarrow{(p_{\mathrm{ST}},p_{\mathrm{cr}})} \prod_{u\in \sW_n^{S_0}}\Phi_{\mathrm{nc},S_0(u)}(S_0(u)) \times \Phi_{\mathrm{nc},s}^{\mathrm{{cr}}}$ is injective.\;

			\end{thm}
			
			\subsection{Automorphic side}
			To the best of the author's knowledge, there is no general method to construct an explicit locally analytic representation $\pi_{1}(\Dpik)$ that uniquely determines $\Dpik$. Thus, we do not investigate that problem here and instead aim to capture $\underline{\sL}(\Dpik)$ through locally analytic methods. We organize into two parts.
			\begin{itemize}
				\item[(\textbf{I})] (Theorem \ref{decompforextgpsintro} and Proposition \ref{BSinv}) We compute higher extension groups of parabolic inductions of locally analytic generalized Steinberg representations, and use them to encode the information of $p_{\mathrm{ST}}(\underline{\sL}(\Dpik))$ or determine $\cC_{\mathrm{ST}}(\Dpik)$.
				\item[(\textbf{II})] (Theorem \ref{thmintor2}) We construct several explicit locally analytic representations $\pi_{1}^{\flat}(\Dpik)\hookrightarrow \pi_{1}^{\sharp}(\Dpik)\hookrightarrow \pi_{1}^{\Delta,-}(\Dpik)\hookrightarrow \pi_{1}^{\Delta}(\Dpik)\hookrightarrow \pi_{1}^{+}(\Dpik)$ (the last one is  conjectural,\;see Expectation \ref{constforfull}) and describe what information of $\sL(\Dpik)$ of equivalently the Hodge filtration $\fil^{\bullet}_{H}(D)$  they determine (Theorem \ref{deterHodgefilintro}).\;
			\end{itemize}
		
		\subsubsection{Breuil-Schraen $\sL$-invarians for the  monodromy type $S_0$}
Let $d$ be an integer,\;and let $\bT_d$ (resp.,\;$\bB_{d}$,\;resp.,\;$\bP_{d,I}$) be the standard diagonal torus (resp.,\;Borel subgroup of upper triangular matrices,\;resp.,\;parabolic subgroup associate to $I\subseteq \Delta_d:=\{1,\cdots,d-1\}$) of $\GLN_d$ and let $\overline{\bB}_{d}$ (resp.,\;$\op_{d,I}$) be the Borel subgroup opposite to $\bB_{d}$ (resp.,\;$\bP_{d,I}$).\;For $I\subseteq \Delta_d$, set
			\[v^{\ana}_{I,\Delta_{d}}:=i^{\ana}_{I,\Delta_{d}}/\sum_{I\subsetneq J}i^{\ana}_{J,\Delta_{d}},\quad i^{\ana}_{I,\Delta_{d}}:=\left(\ind^{\GLN_d(\bQ_p)}_{\op_{d,I}(\bQ_p)}1_{\op_{d,I}(\bQ_p)}\right)^{\ana}.\]
			$\{v^{\ana}_{I,\Delta_{d}}\}_{I\subseteq \Delta_d}$ are the so-called locally analytic generalized Steinberg representations of $\GLN_d(\bQ_p)$.
			
				Let $J\subseteq \Delta$, and assume that $\bL_J\cong \prod_{i=1}^m\GLN_{n_i}:=\prod_{i=1}^m\bL_{J,i}$. Let $\Delta_{J,i}\subseteq \Delta$ be the simple roots of $\bL_{J,i}$ for $1\leq i\leq m$ (thus $J=\sqcup_{i=1}^m \Delta_{J,i}$).\;Fix $\ba:=(a_1,a_2,\cdots,a_m)\in \BZ^{m}$. For $I\subseteq J$, consider the locally analytic parabolic induction ($\overline{\bP}_{J}$ is the parabolic subgroup opposite to $\bP_{J}$):
			\begin{equation}
				v^{\ana}_{I,J}(\ba)=\left(\ind^G_{\overline{\bP}_{J}(\bQ_p)}\big(\widehat{\boxtimes}_{j=1}^m|\mathrm{det}_{\bL_{J,i}}|^{a_j}v^{\ana}_{\Delta_{J,j}\cap I,\Delta_{J,j}}\big)\right)^{\ana}.
			\end{equation}
			Consider the following higher extension group of admissible locally analytic representations:
			\[\BE_{I_1,I_2}(\ba):=\ext_{G,Z}
			^{|I_2\backslash I_1|}\big(v^{\ana}_{I_1, J}(\ba),v^{\ana}_{I_2, J}(\ba)\big).\]
			Based on the results of \cite{wholeLINV}, we show that
			\begin{thm}[Theorem \ref{decompforextgps}]\label{decompforextgpsintro} $\BE_{I_1,I_2}(\ba)\cong \bigotimes_{j=1}^m\ext_{\bL_{J,j}(\bQ_p),Z}
				^{|I_2'\backslash I_1'|}\big(v^{\ana}_{I_1\cap \Delta_{J,j},\Delta_{J,j}},v^{\ana}_{I_2\cap \Delta_{J,j},\Delta_{J,j}}\big)$.\;If $|I_1\backslash I_2|=1$, we have a canonical isomorphism
				$\BE_{I_1,I_2}(\ba)\cong \homo(\bQ_p^{\times},E)$.
				
			\end{thm}
			This theorem reduces $\BE_{I_1,I_2}(\ba)$ to the higher extension groups of its Steinberg Levi blocks, thus inhert many properties and structures of the higher extension groups in \cite{wholeLINV}. 
			
			Put $\BE_{J}(\ba):={\BE}_{J,\emptyset}(\ba)$. A Breuil-Schraen $\sL$-invariant is a codimension-$1$ subspace $\BW(\ba)\subseteq \BE_J(\ba)$ that satisfies many transversality conditions (see Definition \ref{BDlinvgeneral}). Let $\mathcal{BS}_J(\ba)$ be the moduli space of Breuil-Schraen $\sL$-invariants inside $\BE_J(\ba)$. Let $\Phi_J^+$ be the set of positive roots of $\bL_J$, then for $\alpha\in \Psi_{J}^+$, we can attach a unique $\sL_{\alpha}(\ba)\in E$, which induces an isomorphism
			\begin{equation}
				\begin{aligned}
					\mathcal{BS}_J(\ba)\xrightarrow{\sim} \fl_{J}\cap \fn_{\emptyset}, \BW(\ba)\mapsto (\sL_{\alpha}(\ba))_{\alpha\in \Psi_{J}^+},
				\end{aligned}
			\end{equation}
			where $\fl_{I}$ (resp., $\fn_{I}$) is the $E$-Lie algebra of $\bL_{I}$ (resp., the unipotent radical of $\bP_{I}$). The symbol $(\sL_{\alpha}(\ba))_{\alpha\in \Psi_{J}^+}$ indicates that there exists a way to match Breuil-Schraen $\sL$-invariants inside $\BE_J(\ba)$ with the $p$-adic Hodge parameters or Fontaine-Mazur $\sL$-invariants of Steinberg $(\varphi,\Gamma)$-modules.
			
				\begin{pro}\label{BSinv}For any $u\in \sW_n^{S_0}$, $\{\underline{\sL}(E_{u,i})\}_{1\leq i\leq f_u}$ are encoded by a Breuil-Schraen $\sL$-invariants $\BW_{\Dpik}(u)\subseteq \BE_{S_0(u)}(\ba(u))$ for certain $\ba(u)$. In particular, $\{\BW_{\Dpik}(u)\}_{u\in \sW_n^{S_0}}$ determines $p_{\mathrm{ST}}(\underline{\sL}(\Dpik))$.
			\end{pro}
			
			
			\subsubsection{Locally analytic representations and $p$-adic Hodge parameters }
 For $u\in \sW_n^{S_0}$, consider the locally analytic (resp.,\;locally algebraic) principal series:
			\[\mathrm{PS}_{u}(\underline{\phi},\bh):=\left(\ind^G_{\ob(\bQ_p)}\unr(\underline{\phi}^u)\eta\chi_{\lambda}\right)^{\ana},\quad\text{resp.,\;}\mathrm{PS}_{u}^{\lalg}(\underline{\phi},\bh):=\left(\ind^G_{\ob(\bQ_p)}\unr(\underline{\phi}^u)\eta\right)^{\infty}\otimes_EL(\lambda)\]
			where $\chi_{\lambda}$ (resp.,\;$L(\lambda)$) is the character of $\bT(\bQ_p)$ (resp.,\;algebraic representation of $G$) with weight (resp.,\; highest weight) ${\lambda}=\bh-\theta$  and $\eta$ is the square root of the modulus character of $\ob$.\;
			
			Let $\pi_{\sharp}^{\lalg}(\underline{\phi},\bh)=\mathrm{PS}^{\lalg}_{1}(\underline{\phi},\bh)$.\;The unique generic irreducible constituent $\mathrm{ST}^{\lalg}(\underline{\phi},\bh)$ of $\pi_{\sharp}^{\lalg}(\underline{\phi},\bh)$ is its unique cosocle.\;Let $\mathrm{ST}_{u}(\underline{\phi},\bh)$ be the unique maximal quotient of $\mathrm{PS}_{u}(\underline{\phi},\bh)$ with socle $\mathrm{ST}^{\lalg}(\underline{\phi},\bh)$ and $\mathrm{ST}^{\lalg}_u(\underline{\phi},\bh)$ be the locally algebraic vectors in $\mathrm{ST}_u(\underline{\phi},\bh)$.\;Let $\mathrm{ST}_{u,1}(\underline{\phi},\bh)$ be the ``first two layers" of the socle filtration of
			$\mathrm{ST}_u(\underline{\phi},\bh)$,\;i.e.,\;the unique subrepresentation of $\mathrm{ST}_u(\underline{\phi},\bh)$ satisfying the form:
			\[\Big[\mathrm{ST}^{\lalg}_u(\underline{\phi},\bh)-\soc_G\mathrm{ST}_u(\underline{\phi},\bh)/\mathrm{ST}^{\lalg}_u(\underline{\phi},\bh)\Big].\]
			For $j\in \Delta$ and $u\in \sW_n^{S_0}$, let $\widehat{j}:=\Delta\backslash \{j\}$  and  $[u]_j$ be the minimal length representative in $\sW_{\widehat{j}}u$.\;We put
			\[\cI_j^{\sharp}:=\{[u]_j: u\in \sW_n^{S_0}\},\quad \cI^{\flat}_{j}:=\{u\in \cI_j^{\sharp}:j\notin S_0(u) \}.\]
			For $\star\in\{\flat,\sharp\}$, there exists a locally analytic representation $\pi^{\star}_{1}(\underline{\phi},\bh)$ which lies in the short exact sequence:
			\begin{equation}\label{firstwholeintro}
				0\rightarrow \pi^{\lalg}_{1}(\underline{\phi},\bh) \rightarrow \pi^{\star}_{1}(\underline{\phi},\bh)\rightarrow \bigoplus_{u\in \cI^{\star}_{j},j\in \Delta}C_{j,u}\rightarrow 0,
			\end{equation}
			where $C_{j,u}$ are explicit locally analytic representations given by the Orlik-Strauch functor (see \cite[The main theorem]{orlik2015jordan}).\;$\pi^{\sharp}_{1}(\underline{\phi},\bh)$ is the unique quotient of  certain "amalgamated sum" of $\{\mathrm{ST}_{u,1}(\underline{\phi},\bh)\}_{u\in \sW^{\emptyset,S_0}_{n}}$ with socle $\mathrm{ST}^{\lalg}(\underline{\phi},\bh)$.\;

			\begin{pro}\label{thmintor2}	(Proposition \ref{construext1}) Suppose $\star\in\{\flat,\sharp\}$.\;Let $\homo_{\fil}(D,D)$ be the subspace of endomorphisms of $D$ that respect $\fil^{H}_{\bullet}(D)$.\;There is a map (which depends only on the choice of $\log_p(p)\in E$):
					\begin{equation}\label{reconsfortDpikcirc}
						{t}^{\star}_{D}:\ext^1_{G}\big(\pi_{\sharp}^{\lalg}(\underline{\phi},\bh),\pi^{\star}_{1}(\underline{\phi},\bh)\big)\twoheadrightarrow
						\ext^{1,\flat}_{\gr}(\Dpik,\Dpik)\oplus\homo^{\star}_{\fil}(D,D),
					\end{equation}
					where $\ext^{1,\flat}_{\gr}(\Dpik,\Dpik)\cong \homo(\bZ_{S_0}(\bQ_p),E)$ and $\homo^{\star}_{\fil}(D,D)\subseteq \homo_{\fil}(D,D)$ is  given  in Remark \ref{rmkforlastthmintro} $(2)$ explicitly.\;

			\end{pro}

			For $i\in\Delta$,\;define
			\[D^{\star}_{(i)}:=E\langle e_{j}:j\in \{1,\cdots,n\}\backslash \{u^{-1}(1),\cdots,u^{-1}(i)\},u\in \cI^{\star}_i\rangle\subseteq D.\]
			Put $\bd^{\star}_{i}=n-\dim_ED^{\star}_{(i)}$.\;By definition,\;$D^{\sharp}_{(n-1)}\subseteq D^{\sharp}_{(n-2)}\subseteq\cdots \subseteq D^{\sharp}_{(1)}=D$ and $0=\bd^{\sharp}_{1}\leq \bd^{\sharp}_{2}\leq \cdots\leq \bd^{\sharp}_{n}$.\;
			
				\begin{thm}\label{deterHodgefilintro} (Theorem \ref{Hodgeparadeter} $\&$ Corollary  \ref{Hodgeparadetercor}) For $1\leq t\leq n-1$,\;$\ker(t^{\star}_{D})$ determines
				\begin{equation}
					\begin{aligned}
						\Big\{\fil_H^{-\bh_{{\max\{j,\bd^{\star}_{t}\}}+1}}(D)\oplus D/\fil_H^{-\bh_{\bd^{\star}_{t}}}(D)\Big\}_{1\leq j\leq t}
					\end{aligned}
				\end{equation}	
				In particular,\;if $\bd^{\star}_{t}=0$,\;then $\ker(t^{\star}_{D})$ determines $\{\fil_H^{-\bh_{j+1}}(D)\}_{1\leq j\leq t}$.\;
			\end{thm}
			
			Let $\pi^{\star}_{\min}(\Dpik)$ be the unique quotient of the tautological extension of $\ker(t^{\star}_{D})\otimes_E\pi_{\natural}^{\lalg}(\underline{\phi},\bh)$ by $\pi^{\star}_{1}(\underline{\phi},\bh)$ with socle $\mathrm{ST}^{\lalg}(\underline{\phi},\bh)$.\;Moreover,\;similar to the constructions of locally analytic representations in \cite{2019DINGSimple},\;we construct a locally analytic representation $\pi^{\Delta,-}_{1}(\underline{\phi},\bh)\supseteq\pi^{\star}_{1}(\underline{\phi},\bh)$
			that carries exactly the information of $\{\sL^{\Delta}(E_{u,i})\}_{u\in \sW_n^{S_0},1\leq i\leq f_u}$ (equivallently,\;$p^{\Delta}_{\mathrm{ST}}(\underline{\sL}(\Dpik))$ in Remark \ref{rmkforlastthmintro}-(4)).\;Put
			\[\pi^{\star,\Delta,-}_{\min}(\Dpik):=\pi^{\star}_{\min}(\Dpik)\oplus_{\pi^{\star}_{1}(\underline{\phi},\bh)}\pi^{\Delta,-}_{1}(\underline{\phi},\bh).\]
			
			\begin{thm}\label{thmintor3}(Theorem \ref{mainthmglobal} $\&$ Corollary \ref{thmforamosttwoLGC}: local and global compatibility for the patched setting) Let $\Pi_{\infty}$ be the patched $R_{\infty}$-admissible unitary representation of $G$, where $R_{\infty}$ is the patched Galois deformation ring.\;If $\rho_p$ comes from a maximal ideal $\fm_{\rho}$ of $R_{\infty}[1/p]$,\;then $\pi^{\star,\Delta,-}_{\min}(\Dpik)\hookrightarrow \Pi_{\infty}[\fm_{\rho}]$.\;In particular,\;when $ \max_{1\leq l\leq s}l_i\leq 2$, $\pi^{\star,\Delta,-}_{\min}(\Dpik)$ determines $\Dpik$.
			\end{thm}
				\begin{rmk}\label{rmkforlastthmintro}
				\begin{itemize}
					\item[(1)] $\ker(t^{\star}_{D})$ captures ``crystalline'' Hodge parameters between the Steinberg subquotients.\;
					\item[(2)] For $I\subseteq \Delta$,\;let $\fn_I$ and $\fz_{I}$ be the $E$-Lie algebra of $\bN_I$ and $\bZ_{I}$ respectively.\;Put $\tau_{I}=\fz_{I}\ltimes\fn_{I}$.\;Under the basis $\{e_i\}_{1\leq i\leq n}$,\;we identify the space $\homo_{E}(D,D)$ with $\fg$, so that $\homo_{\fil}(D,D)\cong \mathrm{Ad}_g(\fb)$.\;Then 
					\begin{equation*}
						\begin{aligned}
							&\homo^{\flat}_{\fil}(D,D)\cong\sum_{u \in \sW_n^{S_0}}\mathrm{Ad}_{u^{-1}}(\tau_{S_0(u)})\cap \mathrm{Ad}_{g}(\fb)\hookrightarrow\homo^{\sharp}_{\fil}(D,D)\cong\sum_{u \in \sW_n^{S_0}}\mathrm{Ad}_{u^{-1}}(\fb)\cap \mathrm{Ad}_{g}(\fb).
						\end{aligned}
					\end{equation*}
					\item[(3)] Consider the morphism $g_{S_0}^{\circ}:\bigoplus_{u\in \sW_n^{S_0}}\mathrm{Ad}_{u^{-1}}(\tau_{S_0(u)})\cap \mathrm{Ad}_{g}(\fb)\rightarrow \mathrm{Ad}_{g}(\fb)$.\;Theorem \ref{capturetype2intro} is equivalent to the following statements on $E$-Lie algebras: the element $[g\bB]\in \Phi_{\mathrm{nc},\Delta}(S_0)$ is determined by $\ker(g_{S_0}^{\circ})$ and $\big\{(\fl_{S_0(u)}\cap \mathrm{Ad}_{ug}(\fb))_{_{/\mathrm{Ad}(\bZ_{S_0(u)})}}\big\}_{u\in \sW_n^{S_0}}$.
					\item[(4)]Denoted by $p^{\Delta}_{\mathrm{ST}}(\Phi_{\mathrm{nc},\Delta}(S_0))$ the ``simple'' part of $p_{\mathrm{ST}}(\Phi_{\mathrm{nc},\Delta}(S_0))$,\;i.e.,\;$p^{\Delta}_{\mathrm{ST}}(\underline{\sL}(\Dpik))$ encodes the information $\{\sL^{\Delta}(E_{u,i})\}_{u\in \sW_n^{S_0},1\leq i\leq f_u}$,\;where $\sL^{\Delta}(E_{u,i})$ denotes the Fontaine-Mazur simple $\sL$-invariants of Steinberg $E_{u,i}$ (see \cite[Page 7994]{2019DINGSimple}).\;In particular,\;when $\max_{1\leq l\leq s}l_i\leq 2$,\;$p^{\Delta}_{\mathrm{ST}}(\Phi_{\mathrm{nc},\Delta}(S_0))=p_{\mathrm{ST}}(\Phi_{\mathrm{nc},\Delta}(S_0))$.\;We use it to prove the last assertion in Theorem \ref{thmintor3}.
					\item[(5)] 	See Appendix \ref{exampleGL23} for explicit examples of such locally analytic representations for the cases $n=2,3$.
				\end{itemize}
			\end{rmk}


			\begin{expect} \label{constforfull}(Expectation \ref{descibeimagetDGL2})
				For each $u\in \sW_n^{S_0}$ and $1\leq i\leq f_u$, assume that we can associate a locally analytic representation $\pi_{\min}(E_{u,i})$ that determines $E_{u,i}$.\;Consider the locally analytic parabolic induction
				\begin{equation*}
					\mathrm{PS}_{\cF_{u}
					}(\underline{\phi},\bh):=\left(\ind^G_{\op_{S_0(u)}(\bQ_p)}\pi_{\ana}(\underline{\phi}^u)\eta_{S_0(u)}\right)^{\ana}, \pi_{\ana}(\underline{\phi}^u):=\boxtimes_{j=1}^{f_u}\pi_{\min}(E_{u,j}).
				\end{equation*}
				Let $\mathrm{ST}_{\cF_{u}}(\underline{\phi},\bh)$ be the unique maximal quotient of $\mathrm{PS}_{\cF_{u}}(\underline{\phi},\bh)$ with socle $\mathrm{ST}^{\lalg}(\underline{\phi},\bh)$. We have injections of $G$-representations $\mathrm{PS}_{S_0(u)}(\underline{\phi},\bh)\hookrightarrow\mathrm{PS}_{\cF_{u}}(\underline{\phi},\bh)$ and $\mathrm{ST}_{S_0(u)}(\underline{\phi},\bh)\hookrightarrow\mathrm{ST}_{\cF_{u}}(\underline{\phi},\bh)$.\;Taking  ``amalgamated sum" of $\{\mathrm{ST}_{\cF_{u}}(\underline{\phi},\bh)\}_{u\in \sW_n^{S_0}}$ leads  to a locally analytic representation $\pi_{1}^+(\underline{\phi},\bh)$ with socle  $\mathrm{ST}^{\lalg}(\underline{\phi},\bh)$.\;It is possible to extend (\ref{reconsfortDpikcirc})  to a surjection (such map is predicted in \cite[(6)]{BDcritical25} for generl de Rham $\Dpik$)
				\[t^+_{D}:\ext^1_{G}\left(\pi_{\natural}^{\lalg}(\underline{\phi},\bh),\pi^+_{1}(\underline{\phi},\bh)\right)\twoheadrightarrow\ext^{1,\flat}_{\gr}(\Dpik,\Dpik)\oplus\homo_{\fil}(D,D).\]
			\end{expect}
			Moreover,\;Theorem \ref{capturetype2intro} thus  indicates that 
			\begin{expect}
				Then $\ker(t^+_{D})$ determines $\Dpik$.\;
			\end{expect}
			Let $\pi_{\min}(\Dpik)$ be the unique quotient of the tautological extension of $\ker(t^{+}_{D})\otimes_E\pi_{\natural}^{\lalg}(\underline{\phi},\bh)$ by $\pi^{+}_{1}(\underline{\phi},\bh)$ with socle $\mathrm{ST}^{\lalg}(\underline{\phi},\bh)$.\;It gives a possible way to construct a locally analytic representation which determines $\Dpik$.

           We  relate Theorem \ref{thmintor2} to trianguline deformations of $\Dpik$.\;For $u\in \sW_n^{S_0}$, let $\ext^1_u(\Dpik,\Dpik)$ 
			be the subspace of $ \ext^1(\Dpik,\Dpik)$ of trianguline deformations that respect to $\cF_u$,\;i.e.,\;$\widetilde{D}\in \ext^{1}_{u}(\Dpik,\Dpik)$ if and only if $\widetilde{D}=[\widetilde{R}_{1}-\widetilde{R}_{2}-\cdots-\widetilde{R}_{n}]$, where $\widetilde{R}_l\cong \cR_{E[\epsilon]/\epsilon^2}(\unr(\phi_{u^{-1}(l)})z^{\bh_l}(1+\psi_l\epsilon))$ for $\psi_l\in \homo(\bQ_p^{\times},E)$.\;We have natural maps:
			\begin{equation}
				\begin{aligned}
					\kappa_{u}&:\ext^{1}_{u}(\Dpik,\Dpik)\rightarrow \homo(\bT(\bQ_p),E),\widetilde{D}\mapsto (\psi_i)_{1\leq i\leq n}.\;
				\end{aligned}
			\end{equation}
			See Section \ref{generalforparedefor} for a study of the kernel and image of $\kappa_{u}$ (the images are described by the Fontaine-Mazur simple $\sL$-invariants).\;Consider a natural morphism:
			\begin{equation}
				\begin{aligned}
					\homo(\bT(\bQ_p),E)&\rightarrow\ext^1_{G}(\mathrm{PS}_{u}(\underline{\phi},\bh),\mathrm{PS}_{u}(\underline{\phi},\bh)), \\
					\psi&\mapsto \left(\ind^G_{\ob(\bQ_p)}\unr(\underline{\phi}^u)\eta\chi_{\lambda}\otimes_E(1+\psi\epsilon)\right)^{\ana}.\;
				\end{aligned}
			\end{equation}
			It induces a map (push forward via $\mathrm{PS}_{u}(\underline{\phi},\bh)\twoheadrightarrow \mathrm{ST}_{u}(\underline{\phi},\bh)\hookrightarrow \pi^{\sharp}_{1}(\underline{\phi},\bh)$ and pullback via a natural map $\pi_{\sharp}^{\lalg}(\underline{\phi},\bh)\rightarrow \mathrm{PS}_{u}(\underline{\phi},\bh)$)
			\[\zeta_{u}:\homo(\bT(\bQ_p),E)\rightarrow \ext^1_{G}\left(\pi_{\sharp}^{\lalg}(\underline{\phi},\bh),\pi^{\sharp}_{1}(\underline{\phi},\bh)\right).\]
			Let $\ext^{1}_{0}(\Dpik,\Dpik)=\ker \kappa_{1}$.\;For any subspace $V\subseteq \ext^{1}(\Dpik,\Dpik)$,\;put $\overline{V}:=V/V\cap \ext^{1}_0(\Dpik,\Dpik)$.\;Thus,\;the composition $\oplus_{u\in\sW_n^{S_0}}(\zeta_u\circ \kappa_u)$ gives a map
			\begin{equation}\label{gammDpikintro}
				\begin{aligned}
					\gamma_{\Dpik}:\bigoplus_{u\in\sW_n^{S_0}}\overline{\ext}^{1}_{u}(\Dpik,\Dpik)\rightarrow \ext^1_{G}\left(\pi_{\sharp}^{\lalg}(\underline{\phi},\bh),\pi_{1}^{\sharp}(\underline{\phi},\bh)\right).
				\end{aligned}
			\end{equation}
			Let $\iota_u:\overline{\ext}^{1}_{u}(\Dpik,\Dpik)\hookrightarrow \overline{\ext}^{1}(\Dpik,\Dpik)$ be the natural inclusion.\;Put $g_{\Dpik}:\bigoplus_{u\in \sW_n^{S_0}}\overline{\ext}^{1}_{u}(\Dpik,\Dpik)\rightarrow \overline{\ext}^{1}(\Dpik,\Dpik).$

			\begin{pro}\label{compareintrodeformations}
				\begin{itemize}
					\item[(1)](Proposition \ref{splitingforext1gps}) Using the theory of almost de Rham, there exists a natural morphism $\nu:{\ext}^{1}(\Dpik,\Dpik)\rightarrow \homo_{\fil}(D,D)$.\;We have a splitting (which depends only on the choice of $\log_p(p)\in E$):
					\begin{equation}\label{splittingintro}
						f_{\Dpik}:\overline{\ext}^{1}(\Dpik,\Dpik)\xrightarrow{\sim}\ext^{1,\flat}_{\gr}(\Dpik,\Dpik)\oplus\mathrm{Im}(\nu).
					\end{equation}
					\item[(2)]	(Theorem \ref{descibeimagetD})
					Under the splitting (\ref{splittingintro}),\;
					$t^{\sharp}_{D}\circ\gamma_{\Dpik}|_{\overline{\ext}^{1}_{u}(\Dpik,\Dpik)}=f_{\Dpik}\circ\iota_u$.\;In particular,\;$f_{\Dpik}\circ g_{\Dpik}=t^{\sharp}_{D}\circ\gamma_{\Dpik}$.
				\end{itemize}
			\end{pro}			
			
			\begin{rmk}
				For $u\in \sW_n^{S_0}$, let $\ext^{1}_{S_0(u)}(\Dpik,\Dpik)$
				be the $\bP_{S_0(u)}$-parabolic deformations of $\Dpik$,\;i.e.,\;$\widetilde{D}\in \ext^{1}_{S_0(u)}(\Dpik,\Dpik)$ if and only if 
				$\widetilde{D}=[\widetilde{E}_{u,1}-\widetilde{E}_{u,2}-\cdots-\widetilde{E}_{u,f_u}]$,\;where $\widetilde{E}_{u,i}$ is a deformation of $E_{u,i}$ over $E[\epsilon]/\epsilon^2$.\;Then (\ref{gammDpikintro}) extends to $	\gamma^+_{\Dpik}:\bigoplus_{u\in\sW_n^{S_0}}\overline{\ext}^{1}_{S_0(u)}(\Dpik,\Dpik)\rightarrow\ext^1_{G}\left(\pi_{\natural}^{\lalg}(\underline{\phi},\bh),\pi^+_{1}(\underline{\phi},\bh)\right)$.\;Similar to Proposition \ref{compareintrodeformations} $(2)$,\;we expect that  $f_{\Dpik}\circ \overline{g}^{+}_{\Dpik}=t^+_{D}\circ\gamma^+_{\Dpik}$ (see Expectation \ref{constforfull}  for $t^+_{D}$).\;
			\end{rmk}

			\section*{Acknowledgment}
			The author thanks Yiwen Ding and Zicheng Qian for discussions and for answering questions.\;The author is especially grateful to Zicheng Qian.\;
			
			\section{Preliminaries}
			
			\subsection{General notation}

			\noindent Fix a prime $p$.\;Let $E$ be a finite extension of $\bQ_p$ with ring of integers $\co_E$ and uniformizer $\varpi_E$.\;Let $\gal_{\bQ_p}:=\gal(\overline{\bQ}_p/\bQ_p)$ be the absolute Galois group of $\bQ_p$.\;
			
			Let $\GL_n$ be the general linear group over $\bQ_p$.\;Let $\Delta:=\Delta_n$ be the set of simple roots of $\GL_n$ (with respect to the Borel subgroup $\bB$ of upper triangular matrices), and identify $\Delta$ with $\{1,\cdots,n-1\}$ so that $i\in\{1,\cdots,n-1\}$ corresponds to the simple root $\alpha_i:\;(x_1,\cdots,x_n)\in\ft\mapsto x_i-x_{i+1}$, where $\ft$ is the Lie algebra of the (standard) diagonal torus $\bT$.\;Each $I\subseteq\Delta$ gives rise to the standard Levi subgroup $\bL_I\subseteq\GL_n$ with $\bT\subseteq\bL_I$, and $I$ is exactly the set of simple roots of $\bL_I$.\;Let $\bP_I:=\bL_I\bB$ be the standard parabolic subgroup of $\GL_n$ containing $\bB$.\;In particular,\;$\bP_{\Delta}=\GL_n$ and $\bP_{\emptyset}=\bB$.\;Let $\overline{\bP}_{I}$ be the parabolic subgroup opposite to $\bP_I$.\;Let $\bN_{I}$ (resp.\;$\overline{\bN}_{I}$) be the unipotent radical of $\bP_{I}$ (resp.\;$\overline{\bP}_{I}$).\;We have Levi decompositions $\bP_I=\bL_I\bN_I$ and $\overline{\bP}_I=\bL_I\overline{\bN}_I$, and $\Delta\backslash I$ is precisely the set of simple roots of $\bN_I$.\;In particular,\;if $I=\{i\}$ for $i\in\Delta$,\;we write the subscript $i$ instead of $\{i\}$.\;Let $\bZ$ (resp.,\;$\bZ_I$) be the center of $\GL_n$ (resp.,\;$\bL_I$).\;Let $\fg$,\;$\fb$,\;$\fp_I$,\;$\fl_I$,\;$\fn_I$, and $\fz_I$ be the Lie algebras of $\GL_n$,\;$\bB$,\;$\bP_I$,\;$\bL_I$,\;$\bN_I$, and $\bZ_I$ over $E$, respectively.\;We put $G:=\GL_n(\bQ_p)$.\;

			Let $m\in\BZ_{\geq 1}$, and let $\pi$ be an irreducible smooth admissible representation of $\GLN_m(\bQ_p)$.\;Let $\rec(\pi)$ be the $m$-dimensional absolutely irreducible $F$-semisimple Weil--Deligne representation of the Weil group $W_{\bQ_p}$ via the normalized classical local Langlands correspondence (normalized as in \cite{scholze2013local}).\;We normalize the reciprocity isomorphism $\rec:\bQ_p^\times\rightarrow W_{\bQ_p}^{\mathrm{ab}}$ from local class field theory so that the uniformizer $p$ is sent to geometric Frobenius,\;where $W_{\bQ_p}^{\mathrm{ab}}$ is the abelianization of the Weil group $W_{\bQ_p}\subset \gal_{\bQ_p}$.\;
			
			Let ${\ccyc}:\gal_{\bQ_p}\rightarrow \bZ_p^\times$ be the $p$-adic cyclotomic character.\;Then, by local class field theory,\;$\ccyc\circ\mathrm{rec}_{\bQ_p}=x|x|$.\;For $k\in \BZ$,\;we denote by $z^{k}$ the character sending $z$ to $z^k$.\;Let $A$ be an affinoid $E$-algebra.\;For $\alpha\in A^\times$,\;denote by $\unr(\alpha)$ the unramified character of $\bQ_p^\times$ sending $p$ to $\alpha$.\;A locally $\bQ_p$-analytic character $\delta:\bQ_p^\times\rightarrow A^\times$ induces a $\bQ_p$-linear map $\bQ_p\rightarrow A$,\;$x\mapsto \frac{d}{dt}\delta(\exp(tx))|_{t=0}$, and hence an $E$-linear map $E\rightarrow A$.\;There exists $\wt(\delta)$ (which we call the weight of $\delta$) such that this map is given by $a\mapsto a\cdot\wt(\delta)$.\;

			Let $\Art_E$ be the category of local Artinian $E$-algebras with residue field isomorphic to $E$.\;Let $\cR_{\bQ_p}:=B_{\rig}^{\dagger}$ be the Robba ring.\;For $A\in \Art_E$,\;let $\cR_{A}:=\cR_{\bQ_p}\widehat{\otimes}_{\bQ_p}A$,\;the Robba ring over $\bQ_p$ with coefficients in $A$.\;We write $\cR_{A,L}(\delta_A)$ for the rank-one $(\varphi,\Gamma)$-module over $\cR_{A,L}$ associated to a continuous character $\delta_A:L^\times\rightarrow A^\times$.\;Let $M$ be a $(\varphi,\Gamma)$-module over $\cR_{A,L}$;\;for simplicity we write $M(\delta_A):=M\otimes_{\cR_{A,L}}\cR_{A,L}(\delta_A)$.\;For $A=E[\epsilon]/\epsilon^2$ and a $(\varphi,\Gamma)$-module $M$ over $\cR_E$,\;we identify elements of $\ext^1_{(\varphi,\Gamma)}(M,M)$ with deformations of $M$ over $\cR_{E[\epsilon]/\epsilon^2}$.\;We define the Hodge--Tate weights of a de Rham representation as the opposites of the jumps in the filtration on the covariant de Rham functor, so that the Hodge--Tate weight of ${\ccyc}$ is $1$.\;

			Let ${\lambda}:=(\lambda_{1},\cdots,\lambda_{n})$ be a weight of $\ft$.\;For $I\subseteq\Delta$,\;we say that ${\lambda}$ is $I$-dominant if $\lambda_{i}\geq\lambda_{i+1}$ for all $i\in I$.\;Denote by $X_{I}^+$ the set of $I$-dominant integral weights of $\ft$.\;For ${\lambda}\in X_{I}^+$, there exists a unique irreducible algebraic representation $L({\lambda})_{I}$ of $\bL_{I}$ with highest weight ${\lambda}$,\;so that $\overline{L}(-{\lambda})_{I}:=(L({\lambda})_{I})^\vee$ is the irreducible algebraic representation of $\bL_{I}$ with highest weight $-{\lambda}$.\;Denote $\chi_{{\lambda}}:=L({\lambda})_{\emptyset}$.\;If ${\lambda}\in X_{\Delta}^+$, let $L({\lambda}):=L({\lambda})_{\Delta}$.\;For an integral weight ${\lambda}\in X_I^+$, denote by $M_I({\lambda}):=\text{U}(\fg)\otimes_{\text{U}(\fp_{I})}L({\lambda})_{I}$ (resp.\;$\overline{M}_I({\lambda}):=\text{U}(\fg)\otimes_{\text{U}(\overline{\fp}_{I})}L({\lambda})_{I}$) the corresponding Verma module with respect to $\fp_{I}$ (resp.\;$\overline{\fp}_{I}$).\;Let $L({\lambda})$ (resp.\;$\overline{L}({\lambda})$) be the unique simple quotient of $M({\lambda}):=M_{\emptyset}({\lambda})$ (resp.\;of $\overline{M}_{\emptyset}({\lambda})$).\;

			Denote by $\sW_n$ ($\cong S_n$) the Weyl group of $\GL_n$, and by $s_{i}$ the simple reflection corresponding to $i\in\Delta$.\;For any $I\subset\Delta$,\;define $\sW_{I}$ to be the subgroup of $\sW_{n}$ generated by the simple reflections $s_{i}$ with $i\in I$ (so that $\sW_I$ is the Weyl group of $\bL_I$).\;For $w\in\sW_n$,\;we identify $w$ with the corresponding permutation matrix.\;Recall that $\sW_{I}\backslash\sW_n$ has a canonical set of representatives, denoted by $\sW^{I}_n$, obtained by taking the minimal-length element in each coset.\;Let $w_0:=w_{0,n}$ (resp.,\;$w_{I,0}$) be the longest element in $\sW_n$ (resp.,\;$\sW_I$).\;
			
			If $V$ is a continuous representation of $G$ over $E$,\;we denote by $V^{\ana}$ its locally $\bQ_p$-analytic vectors.\;If $V$ is a locally $\bQ_p$-analytic representation of $G$,\;we denote by $V^{\mathrm{lalg}}$ the locally $\bQ_p$-algebraic subrepresentation of $V$ consisting of its locally $\bQ_p$-algebraic vectors (see \cite{Emerton2007summary} for details).\;Let $P$ be a parabolic subgroup of $G$,\;and $\pi_P$ be a locally $\bQ_p$-analytic representation of $P$ on a locally convex $E$-vector space of compact type (resp.,\;a smooth representation of $P$ over $E$);\;we denote by
			\begin{equation}\label{smoothadj2}
				\begin{aligned}
					&(\mathrm{Ind}_{P}^{G}\pi_P)^{\bQ_p-\ana}:=\{f:G\rightarrow \pi_P \text{\;locally $\bQ_p$-analytic},\;f(pg)=pf(g)\},\\
					&\text{resp.,\;}(\mathrm{Ind}_{P}^{G}\pi_P)^{\infty}=\{f:G\rightarrow \pi_P \text{\;smooth},\;f(pg)=pf(g)\}
				\end{aligned}
			\end{equation}
			the locally $\bQ_p$-analytic parabolic induction (resp.,\;the un-normalized smooth parabolic induction) of $G$.\;It is a locally $\bQ_p$-analytic representation (resp.,\;a smooth representation) of $G$ over $E$ on a locally convex $E$-vector space of compact type, where the left action of $G$ is given by right translation on functions: $(gf)(g')=f(g'g)$.\;
			
			Throughout the paper,\;we use $\bullet-\bullet$ to denote an extension of two objects (for example,\;$(\varphi,\Gamma)$-modules or $G$-representations),\;where the first (resp.,\;second) object is the subobject (resp.,\;quotient).\;

			\subsection{\texorpdfstring{Semistable $(\varphi,\Gamma)$-module over $\cR_{E}$}{Lg}}\label{Omegafil}
			
			Let $\df_{k,\alpha}$ be the following absolutely indecomposable Deligne-Fontaine module over $E$:
			\begin{equation}\label{dfnindecompDFmod}
				(\varphi,N,\df_{k})=\bigoplus_{i=1}^k\;(p^{i-1}\unr(\alpha),0,Ee_i)
			\end{equation}
			such that the monodromy operator $N$ sends $(p^{i}\unr(\alpha),0,Ee_{i+1})$ to $(p^{i-1}\unr(\alpha),0,Ee_i)$ for $i\geq 1$, and sends $(\unr(\alpha),0,Ee_1)$ to zero.\;
			
			Let $\Dpik$ be a semistable $(\varphi,\Gamma)$-module over $\cR_{E}$ of rank $n$,\;and let $\df_{\Dpik}=(\varphi,N,D_{\mathrm{st}}(\Dpik))$ be the associated Deligne--Fontaine module,\;where $D_{\mathrm{st}}(\Dpik):=(\cR_E[\log(X),1/t]\otimes_{\cR_E}\Dpik)^{\Gamma}$ is a finite free $E$-module of rank $n$, and the $(\varphi,N)$-action on $D_{\mathrm{st}}(\Dpik)$ is induced from the $(\varphi,N)$-action on $\cR_E[\log(X),1/t]$.\;There exist integers $\{l_i\}_{1\leq i\leq s}$ and $\{\alpha_i\}_{1\leq i\leq s}$ with $\alpha_i\in E^{\times}$ such that the semisimplification $(\df_{\Dpik})^{\mathrm{ss}}$ of $\df_{\Dpik}$ is $\oplus_{i=1}^s\df_{l_i,\alpha_i}$.\;Thus
			\[\underline{\phi}:=\underline{\phi}_{\Dpik}=(\alpha_1,\alpha_1p,\cdots,\alpha_1p^{l_1-1},\cdots,\alpha_s,\alpha_sp,\cdots,\alpha_sp^{l_s-1}):=(\phi_1,\cdots,\phi_n),\]
			are the $\varphi$-eigenvalues on $D_{\mathrm{st}}(\Dpik)$.\;In the sequel,\;we assume that $\phi_i\neq \phi_j$ for $i\neq j$, and that the ordering satisfies the condition: if $\alpha_j=\alpha_ip^{l_i}$,\;then $j=i+1$.\;

			For $1\leq i\leq s$,\;we write $\df_{l_i,\alpha_i}=\oplus_{j=1}^{l_i}(p^{j-1}\unr(\alpha_i),N=0,Ee_{i,j})$, so that $N(e_{i,j})=e_{i,j-1}$ (resp., $N(e_{i,1})=0$) when $j>1$ (resp.,\;$j=1$).\;We relabel the basis $\{e_{1,j}\}_{1\leq j\leq l_1},\cdots,\{e_{s,j}\}_{1\leq j\leq l_s}$ as $e_1,\cdots,e_n$ (so that $\varphi(e_i)=\phi_ie_i$).\;We define two subsets $S_0:=S_0(\Dpik)\subseteq I_0:=I_0(\Dpik)\subseteq \Delta$ as follows:\;$i\in S_0$ if and only if $N(e_{i+1})=e_{i}$, and $i'\in I_0(\Dpik)$ if and only if $\phi_{i'+1}=\phi_{i'}p$.\;By definition,\;$S_0$ describes the monodromy type of $\df_{\Dpik}$ and $I_0$ describes the non-generic relations among $\underline{\phi}$.\;We say that $\Dpik$ is \textit{generic} if $S_0=I_0$.\;In particular,\;$\Dpik$ is crystalline (resp.,\;has maximal monodromy rank) if and only if $S_0=\emptyset$ (resp.,\;$S_0=\Delta$).\;Then all $(\varphi,N)$-stable complete
			flags in $\df_{\Dpik}$ are $\cF_u:Ee_{u^{-1}(1)}\subseteq Ee_{u^{-1}(1)}\oplus Ee_{u^{-1}(2)}\subseteq \cdots\subseteq \oplus_{i=1}^nEe_{u^{-1}(n)}=\df_{\Dpik}$ for $u\in \sW_n^{S_0}$.\;By Fontaine's theory,\;the $(\varphi,N)$-stable complete
			flag $\cF_u$ on $\df_{\Dpik}$ corresponds to a triangulation (or say $u$-refinement) on $\cM_\Dpik:=\Dpik[1/t]$ (resp.,\;$\Dpik$):
			\begin{equation}
				\begin{aligned}
					\fil_{\bullet}^{\cF_u}\cM_\Dpik: &\ 0 =\fil_0^{\cF_u}\cM_\Dpik \subsetneq \fil_1^{\cF_u}\cM_\Dpik \subsetneq \cdots \subsetneq \fil_{n}^{\cF_u} \cM_\Dpik=\cM_\Dpik,\\
					\text{resp.,\;}\fil_{\bullet}^{\cF_u}\Dpik: &\ 0 =\fil_0^{\cF_u}\Dpik \subsetneq \fil_1^{\cF_u}\Dpik \subsetneq \cdots \subsetneq \fil_{n}^{\cF_u} \Dpik=\Dpik,\;\fil_i^{\cF_u}\Dpik=(\fil_{i}^{\cF_u}\cM_\Dpik)\cap \Dpik
				\end{aligned}
			\end{equation}
			given by saturated $(\varphi,\Gamma)$-submodules of $\cM_\Dpik$ (resp.,\;$\Dpik$) over $\cR_{E}[1/t]$ (resp.,\;$\cR_{E}$).\;
			
		Since $\Dpik$ is semistable,\;it is de Rham.\;Hence we have $D:=D_{\mathrm{st}}(\Dpik)\cong D_{\dr}(\Dpik)$.\;The triangulation $\cF_u$ on $D$
			induces a flag $\fil_{\bullet}^{\cF_u}(D)$ on $D$.\;Moreover,\;the module $D$ is equipped with a natural Hodge filtration $\fil^{\bullet}_{H}(D)$.\;We assume that $\Dpik$ has regular Hodge--Tate weights $\bh:=\bh_{\Dpik}=(\hpi_{1}>\hpi_{2}>\cdots>\hpi_{n})$,\;so $\fil^{\bullet}_{H}(D)$ can be expressed by the following complete flag:
			\[\fil^{\bullet}_{H}(D): \ 0 \subsetneq \fil^{-\hpi_{n}}_{H}(D) \subsetneq \fil^{-\hpi_{n-1}}_{H}(D) \subsetneq \cdots \subsetneq \fil^{{-\hpi_{1}}}_{H}(D)=D.\]
			In this paper,\;we assume that $\Dpik$ is \textit{non-critical}  for all $u\in \sW_n^{S_0}$,\;i.e.,\;the Hodge--Tate weights of $\fil_{i}^{\cF_u}\Dpik$ (resp.,\;$\gr_{i}^{\cF_u} \Dpik$) are $\{\bh_{1},\cdots,\bh_{i}\}$ (resp.,\;$\bh_{i}$).\;For $i=1,\cdots,n$,\;we have 
			\[R_{u,i}:=\gr_{i}^{\cF_u} \Dpik \cong \cR_{E}(\phi_{u^{-1}(i)}{z^{\bh_{i}}}),\;\quad\gr_{i}^{\cF_u} \cM_\Dpik \cong \cR_{E}(\phi_{u^{-1}(i)})[1/t].\]
		We usually write the triangulation $\cF_{u}$ as the following successive (non-split)  extension of $\{R_{u,i}\}$ :
			\[\Dpik=[R_{u,1}-R_{u,2}-\cdots-R_{u,n}],\]
			where $R_{u,i}:=\cR_{E}(\phi_{u^{-1}(i)}z^{\bh_i})$ for $1\leq i\leq n$.\;

			Let $\cT$ be the character space of $\bT(\bQ_p)$ over $E$,\;i.e.,\;the rigid space over $E$ parameterizing continuous characters of $\bT(\bQ_p)$.\;For $u\in \sW_n^{S_0}$,\;put
			\[\unr(\underline{\phi}^u):=\boxtimes_{j=1}^n\unr(\phi_{u^{-1}(j)})\in \cT.\]
			Then $\delta_{\underline{\phi},u}:=\unr(\underline{\phi}^u)z^{\bh}:=(\delta_{\underline{\phi},u,i}:=\unr(\phi_{u^{-1}(i)}){z^{\bh_{i}}})_{1\leq i\leq n}\in \cT$ the parameter of $\cF_u$.\;
			
			\begin{dfn}
				We say that $\Dpik$ is Steinberg if $S_0=\Delta$ (i.e.,\;$\df_{\Dpik}$ has maximal monodromy).\;Thus,\;by definition,\;any semistable $(\varphi,\Gamma)$-module is a successive extension of Steinberg $(\varphi,\Gamma)$-modules.\;
			\end{dfn}

			Let $S_0(u)^c$ be the set of $i\in \Delta$ such that the rank-$2$ subquotient $\cR_{E}(\phi_{u^{-1}(i)}z^{\bh_i})-\cR_{E}(\phi_{u^{-1}(i+1)}z^{\bh_{i+1}})$ of $\Dpik$ is crystalline.\;Let $S_0(u)=\Delta\backslash S_0(u)^c$.\;Let $S_0(u)^c=\{l_1,\cdots,l_{f_u-1}\}$ and put $l_0:=0$, $l_{f_u}:=n$, where $f_u:=|S_0(u)^c|+1$.\;For $1\leq i\leq f_u$,\;put
			\[E_{u,i}:=[R_{u,l_{i-1}+1}-\cdots-R_{u,l_{i}}],\]
			which is a subquotient of $\Dpik$ and a Steinberg $(\varphi,\Gamma)$-module of rank $r_{u,i}:=l_i-l_{i-1}$.\;Therefore,\;
			\begin{equation}\label{parafils0u}
				\Dpik=[E_{u,1}-E_{u,2}-\cdots-E_{u,f_u}],
			\end{equation}
			which induces an increasing $\bP_{S_0(u)}$-parabolic filtration $\cF_{S_0(u)}$ on $\Dpik$ by saturated $(\varphi,\Gamma)$-submodules of $\Dpik$ over $\cR_{E}$:
			\[\fil_{\bullet}^{\cF_{S_0(u)}}\Dpik: \ 0 =\fil_0^{\cF_{S_0(u)}}\Dpik \subsetneq \fil_1^{\cF_{S_0(u)}}\Dpik \subsetneq \cdots \subsetneq \fil_{f_u}^{\cF_{S_0(u)}} \Dpik=\Dpik,\]
			such that $\gr_{i}^{\cF_{S_0(u)}}\Dpik=E_{u,i}$ for $1\leq i\leq f_u$.\;In particular,\;if $u=1$,\;we write $\cF_{S_0}:=\cF_{S_0(1)}$  with the  form:
			\begin{equation}
				\Dpik:=[E_{1}-E_{2}-\cdots-E_{s}],\;E_{s}:=E_{1,s}.\;
			\end{equation}

%
			
			\begin{rmk}\label{dfnforRii1}
				For $i\in \Delta$ and $u\in\sW_n^{S_0}$,\;consider the $\bL_{i}$-parabolic filtration $\cF_{u,[i]}$ associated to $\cF_u$:
				\[\Dpik=[R_{u,1}-\cdots-R_{u,i-1}-R_{u,i}^{i+1}-R_{u,i+2}-\cdots-R_{u,n}],\]
where $R_{u,i}^{i+1}:=[R_{u,i}-R_{u,i+1}]$ is the rank-$2$ subquotient of $\Dpik$.\;
			\end{rmk}

		\section{Reinterpretation of $p$-adic Hodge parameters}

		\subsection{Notations and precise statements for general philosophy}\label{philoforParameter}

		A typical example is given by $n=3$ and $S_0=\{1\}$.\;Let $\underline{\phi}_{\Dpik}=(\alpha,\alpha p,\beta)=(\phi_1,\phi_2,\phi_3)$ for some $\alpha,\beta\in E^{\times}$, and let $(e_1,e_2,e_3)$ be the corresponding $\varphi$-eigenvectors respectively.\;Note that $\sW_{3}^{\{1\}}=\{1,s_2,s_2s_1\}$.\;Let $\Dpik=[\cR_E(\unr(\phi_1)z^{h_1})-\cR_E(\unr(\phi_2)z^{h_2})-\cR_E(\unr(\phi_3)z^{h_3})]$.\;The triangulation $\cF_{s_1s_2}$ has the form $\Dpik=[\cR_E(\unr(\phi_3)z^{h_1})-\cR_E(\unr(\phi_1)z^{h_2})-\cR_E(\unr(\phi_2)z^{h_3})]$.\;Put
		\begin{equation}
			\begin{aligned}
				&D_1:=[\cR_E(\unr(\phi_1)z^{h_1})-\cR_E(\unr(\phi_2)z^{h_2})]\hookrightarrow \Dpik,\\
				&\Dpik\twoheadrightarrow C_1:=[\cR_E(\unr(\phi_1)z^{h_2})-\cR_E(\unr(\phi_2)z^{h_3})].
			\end{aligned}
		\end{equation}
		We obtain an injection $\iota_{\Dpik}:D_1\hookrightarrow \Dpik\twoheadrightarrow C_1$ (note that $\homo(D_1,C_1)$ is $1$-dimensional and generated by this $\iota_{\Dpik}$).\;We list the Hodge filtration $\fil^{\bullet}_H(D)$ on $D$ and the two $p$-adic Hodge parameters $\cL_{01}$ and $\cL_{02}$ of $\Dpik$:
		\begin{equation}\label{fil1}
			\fil^{i}_H(D)=\left\{
			\begin{array}{ll}
				D,\;&i\leq -h_1\\
				E(e_3+e_2+\cL_{02}e_1)\oplus E(e_2+\cL_{01}e_1),\;&-h_1<i\leq -h_2\\
				E(e_3+e_2+\cL_{02}e_1),\;& -h_2<i\leq -h_3,\\
				0,\;& i>-h_3.
			\end{array}
			\right.
		\end{equation}
		In the language of \cite{2019DINGSimple},\;the Hodge parameter $\cL_{01}$ is a simple $\cL$-invariant of the Steinberg block $D_1$,\;and $\cL_{02}$ is a Hodge parameter between the blocks $D_1$ and $\cR_E(\unr(\phi_2)z^{h_3})$ (the so-called higher $\cL$-invariant in \cite{HigherLinvariantsGL3(Qp)}).\;Quotienting by $e_3$ in (\ref{fil1}),\;we see that $\cL_{02}$ becomes the simple $\cL$-invariant of the Steinberg block $C_1$.\;In conclusion,\;$\Dpik$ is uniquely determined by $D_1$ and $C_1$.\;This example indicates that we can "translate" the higher Hodge parameters between different Steinberg blocks  into Hodge parameters inside other Steinberg subquotients of $\Dpik$.\;We give precise statements as follows.\;
		
		Fix a non-critical semistable $(\varphi,\Gamma)$-module $\Dpik$ and keep the notation of Section \ref{Omegafil}.\;Recall the two subsets $S_0\subseteq I_0$ of $\Delta$.\;Under the basis $e_{1},e_{2},\cdots,e_{n}$ of $D$,\;the Hodge filtration $\fil^{\bullet}_{H}(D)$ is parameterized by an element $\underline{\sL}(\Dpik)$ in $\GLN_{n}/\bB$, and by its class $[\underline{\sL}(\Dpik)]\in \bZ_{S_0}\backslash\GLN_{n}/\bB$,\;which we call the $p$-adic parameter of $\Dpik$.\;For any $u\in \sW_n^{S_0}$,\;the non-critical assumption implies that
		\[(\underline{\sL}(\Dpik)\bB(E),u^{-1}\bB(E))\in\GLN_n(E)(1,w_0)(\bB\times \bB)(E)\subset (\GLN_{n}/\bB\times\GLN_{n}/\bB)(E).\]
		Therefore,\;$\underline{\sL}(\Dpik)\in u\bB w_0\bB/\bB$ for all $u\in \sW_n^{S_0}$,\;so
		\[[\underline{\sL}(\Dpik)]\in \bigcap_{u\in \sW_n^{S_0}} \bZ_{S_0}\backslash u^{-1}\bB w_0\bB/\bB:=\Phi_{\mathrm{nc},\Delta}(S_0),\]
		where $\Phi_{\mathrm{nc},\Delta}(S_0)$ is the moduli space of non-critical $p$-adic Hodge parameters with monodromy type $S_0$.\;More precisely,\;the Hodge filtration on $D$ is
		\begin{equation}\label{parameterfilwrite}
			\begin{aligned}
				\fil^{j}_{H}(D)=\left\{
				\begin{array}{ll}
					D=\fil^{-\bh_{2}}_{H}(D)\oplus E(e_1),\;&j\in (-\infty,-\bh_{1}],\\
					\fil^{-\bh_{3}}_{H}(D)\oplus E(e_2+\cL_{12}e_1),\;&j\in (-\bh_{1},-\bh_{2}],\\
					\cdots&\cdots\\
					\fil^{-\bh_{n}}_{H}(D)\oplus E(e_{n-1}+\sum_{l=1}^{n-2}\cL_{l,n-1}e) ,\;& j\in (-\bh_{n-2},-\bh_{n-1}]\\
					E(e_n+\sum_{l=1}^{n-1}\cL_{l,n}e),\;&j\in (-\bh_{n-1},-\bh_n]\\
					0,\;& j\in (-\bh_n,+\infty).
				\end{array}
				\right.
			\end{aligned}
		\end{equation}
		Thus,\;$\underline{\sL}(\Dpik)=(\cL_{ij})_{i<j}\in \bN_{\emptyset}(E)$ in terms of the coordinates in $\bN_{\emptyset}\cong \bB w_0\bB/\bB$.\;Recall that $\Dpik=[E_{1}-E_{2}-\cdots-E_{s}]$ with each $E_{i}$ Steinberg,\;then
		\[\underline{\sL}(\Dpik)=\left(\begin{array}{ccccc}
			\underline{\sL}(\Dpik)_1^1	& \underline{\sL}(\Dpik)_1^2 & \cdots & \underline{\sL}(\Dpik)_1^{s-1} & \underline{\sL}(\Dpik)_1^{s} \\
			0	& \underline{\sL}(\Dpik)_2^2  & \cdots & \underline{\sL}(\Dpik)_2^{s-1} &  \underline{\sL}(\Dpik)_2^{s} \\
			\vdots	& \vdots  & \ddots & \vdots & \vdots \\
			\vdots	& \ddots & 0  & \underline{\sL}(\Dpik)_{s-1}^{s-1}  & \underline{\sL}(\Dpik)_{s-1}^{s} \\
			0	& \cdots & \cdots & 0 & \underline{\sL}(\Dpik)_s^s
		\end{array}\right),\]
		where $\underline{\sL}(\Dpik)_r^q$ denotes the Hodge parameters between the $r$-th block $E_r$ and the $q$-th block $E_q$ for $1\leq r<q\leq s$, and $\underline{\sL}(\Dpik)_r^r$ denotes the Hodge parameters of $E_r$ for $1\leq r\leq s$.\;In general,\;we normalize $[\underline{\sL}(\Dpik)]$ so that $\cL_{t_l,t_l+1}=1$ for $1\leq l\leq s$.\;Thus,\;our goal is to capture the information of $\{\underline{\sL}(\Dpik)_r^q\}_{1\leq r\leq q\leq s}$.\;
		
		For $u\in\sW_n^{S_0}$,\;recall $\Dpik=[E_{u,1}-E_{u,2}-\cdots-E_{u,f_u}]$ such that $\{E_{u,i}\}_{1\leq i\leq f_u}$ are all Steinberg.\;Then $D_{\mathrm{st}}(E_{u,i})$ inherits an induced Hodge filtration $\fil_H^{\bullet}D_{\mathrm{st}}(E_{u,i})$ from $\Dpik$.\;We may ask how much information in $\{\underline{\sL}(\Dpik)_r^q\}_{1<r<q<s}$ can be captured or recovered through $\{\fil_H^{\bullet}D_{\mathrm{st}}(E_{u,i})\}_{u\in \sW_n^{S_0}}$.\;After answering such a question,\;the results in \cite{wholeLINV} show that the $p$-adic Hodge parameters of $\{E_{u,i}\}_{u\in \sW_n^{S_0}}$ can be captured through higher extension groups among locally analytic generalized Steinberg representations.\;
		
		We have the following reformulations in terms of the language of flag varieties.\;For $J\subseteq \Delta$,\;we write $\bL^{\bB}_{J}:=\bL_{J}\cap \bB$.\;For any $u\in \sW_n^{S_0}$,\;we see that $ug\bB\in \bB w_0\bB$.\;We have a natural map:
		\[p_{\mathrm{ST},u}:\bB w_0\bB\cong \bN_{\emptyset}\twoheadrightarrow \bN_{\emptyset}/\bN_{S_0(u)}\cong \bL^{\bB}_{S_0(u)}w_{S_0(u),0}\bL^{\bB}_{S_0(u)}/\bL^{\bB}_{S_0(u)}\]
		which induces a natural map
		\[\bZ_{S_0}\backslash \bB w_0\bB\rightarrow \Phi_{\mathrm{nc},S_0(u)}(S_0(u))= \bZ_{S_0(u)}\backslash\bL^{\bB}_{S_0(u)}w_{S_0(u),0}\bL^{\bB}_{S_0(u)}/\bL^{\bB}_{S_0(u)}.\]
		By definition,\;we have a natural morphism 
		\begin{equation}\label{dfnofpreftoLevi}
			\begin{aligned}
				p_{\mathrm{ST}}:\Phi_{\mathrm{nc},\Delta}(S_0)&\rightarrow \prod_{u\in \sW_n^{S_0}}\Phi_{\mathrm{nc},S_0(u)}(S_0(u)),\\
				[g\bB]&\mapsto \Big([p_{\mathrm{ST},u}(ug\bB)]\Big)_{u\in \sW_n^{S_0}},\;\\\Big(\text{equivalently},\;[\underline{\sL}(\Dpik)]&\mapsto\prod_{u\in \sW_n^{S_0}}\big([\underline{\sL}(E_{u,1})],\cdots,[\underline{\sL}(E_{u,f_u})]\big)\Big).
			\end{aligned}
		\end{equation}
		This map is injective when $s=2$,\;i.e.,\;$\{\fil_H^{\bullet}D_{\mathrm{st}}(E_{u,i})\}_{u\in \sW_n^{S_0}}$ determines $\underline{\sL}(\Dpik)$,\;see Theorem \ref{twoblocksPare1}. 
		
		However,\;when $s>2$,\;besides
		$\{\fil_H^{\bullet}D_{\mathrm{st}}(E_{u,i})\}_{u\in \sW_n^{S_0}}$,\;there are extra ``crystalline'' Hodge parameters between different Steinberg blocks.\;Such phenomena are parallel to the generic crystalline case in \cite{ParaDing2024} and the potentially crystalline case in \cite{HEparaforsemitable}.\;See Theorem \ref{capturetype1},\;Theorem \ref{capturetype1diamond},\;Theorem \ref{capturetype1flag} and Theorem   \ref{capturetype1Liealg} for precise statements of the main theorems.\;
		
		Indeed,\;we only need to consider some special elements in $\sW_n^{S_0}$ (for instance,\;the ``cycle elements" in $\sW_n^{S_0}$).\;The following discussion generalizes the argument in \cite[Section 2.2]{ParaDing2024} to the non-critical semistable case.\;Recall $\Dpik=[E_{1}-E_{2}-\cdots-E_{s}]$ with each $E_{i}$ Steinberg.\;For $1\leq r<q\leq s$,\;let $E_r^q$ be the unique subquotient $[E_{r}-E_{r+1}-\cdots-E_{q}]$ of $\Dpik$ of rank $t_r^q$,\;where $t_r^q:=l_r+\cdots+l_q$.\;For $1\leq i\leq s$, put $t_0=0$ and $t_i=\sum_{j=1}^il_{j}$ (so that $t_s=n$).

		For $1\leq q\leq s$ and $1\leq t\leq l_q$,\;let
		\[E_q^{(t)}=[\cR_E(\unr(\alpha_q)z^{\bh_{t_{q-1}+1}})-\cdots-\cR_E(\unr(\alpha_q p^{t-1})z^{\bh_{t_{q-1}+t}})]\]
		be the unique Steinberg rank-$t$ $(\varphi,\Gamma)$-submodule of $E_q$,\;and
		\[F^{(t)}_q=E_q/E_q^{(t)}=[\cR_E(\unr(\alpha_rp^{t})z^{\bh_{t_{q-1}+t+1}})-\cdots-\cR_E(\unr(\alpha_r p^{l_q-1})z^{\bh_{t_{q-1}+l_q}})]\]
		be the unique Steinberg quotient of $E_q$ of rank-$(l_q-t)$.\;Consider the following parabolic filtrations of $\Dpik$ (recall	for any subquotient $B$ of $\Dpik$, we use $B'$ to denote another $(\varphi,\Gamma)$-module such that $B[1/t]=B'[1/t]$):
		\begin{equation}\label{lqErefinement}
			E_r^q=[E_r-E_{r+1}^{q-1}-E_{q}]=[(E_q^{(t)})'-E_r'-(E_{r+1}^{q-1})'-(E_q/E_q^{(t)})]=[E^{(t')}_r-(E_{r+1}^{q-1})'-E_q'-(F^{(t')}_r)'].\;
		\end{equation}
		We rewrite $E_r'$ as $E_r^{[q,t]}$ (resp.,\;$E_q'$ as $E_q^{[r,t']}$ ) to indicate the choice of $q$ and $1\leq t\leq l_q$ (resp.,\;$r$ and $1\leq t'\leq l_r$).\;By the non-critical assumption,\;the Hodge-Tate weights of $E_r$ (resp.,\;$E_r^{[q,t]}$) are
		\begin{equation}
			\bh_{E_r}:=(\bh_{t_{r-1}+1},\cdots,\bh_{t_r})\; (\text{resp.,\;}\bh_{E_r^{[q,t]}}:=(\bh_{t_{r-1}+1+t},\cdots,\bh_{t_r+t}))
		\end{equation} 
and the Hodge-Tate weights of $E_q$ (resp.,\;$E_q^{[r,t']}$) are
		\begin{equation}
			\bh_{E_q}:=(\bh_{t_{q-1}+1},\cdots,\bh_{t_q})\; (\text{resp.,\;}\bh_{E_q^{[r,t']}}:=(\bh_{t_{q-1}+1-(l_r-t'+1)},\cdots,\bh_{t_q-(l_r-t'+1)}))
		\end{equation} 
		\begin{rmk}By construction,\;we get a map $\iota^{r,q,-}_{\Dpik}(t):E_r\rightarrow E_r^{[q,t]}$.\;We have $\dim_E\homo_{(\varphi,\Gamma)}(E_r,E_r^{[q,t]})=1$ so that
			$\iota^{r,q,-}_{\Dpik}(t)$ generates $\homo_{(\varphi,\Gamma)}(E_r,E_r^{[q,t]})$. The map $\iota^{r,q,-}_{\Dpik}(t)$ induces an injection of filtered structures $\fil_H^{\bullet}D_{\mathrm{st}}(E_r)\rightarrow \fil_H^{\bullet}D_{\mathrm{st}}(E_r^{[q,t]})$.\;The same holds for $E_q^{[r,t']}$.\;
		\end{rmk}
		
		For $1\leq r< q\leq s$,\;define
		\[\cC_{r}^q(\Dpik):=\left\{E_r^{[q,t]},E_q^{[r,t']}\right\}_{1\leq t\leq l_q,1\leq t'\leq l_r}\]
		and $\cC_{r}^r(\Dpik):=\{E_r\}$.\;Put
		\[\cC_{\mathrm{ST}}(\Dpik):=\bigcup_{1\leq r\leq q\leq s}\cC_{r}^q(\Dpik),\]
	Such data can be captured completely through the so-called Breuil-Schraen $\sL$-invariants (which are defined through higher extensions of certain locally analytic representations) in Section \ref{BSINVtype1}.\;
		
		Fix $1\leq r<q\leq s$.\;We next explain explicitly how $\cC_{r}^q(\Dpik)$ captures \& recover the information $\underline{\sL}(\Dpik)^q_{r}$.\;We divide into two cases.\;
		\begin{itemize}
			\item[(1)] When $q=r+1$,\;this is Theorem \ref{twoblocksPare1}.\;We show that $\cC_{r}^{r+1}$ determines $\underline{\sL}(\Dpik)^{r+1}_{r}$.\;
			\item[(2)] Suppose that $q>r+1$,\;this is Proposition \ref{twoblocksPare2}.\;We recover the information in $\underline{\sL}(\Dpik)^q_{r}$ inductively.\;If $\underline{\sL}(\Dpik)^{q'}_{r'}$ is already determined for all $r\leq r'<q'\leq q$ and $(r',q')\neq (r,q)$,\;then $\cC_{r}^q(\Dpik)$ determines $\underline{\sL}(\Dpik)^q_{r}$ up to one parameter.\;The remaining parameters can be viewed as the so-called ``crystalline'' Hodge parameters between Steinberg blocks;\;see Section \ref{cryparablocks} and the proof of Theorem \ref{capturetype1diamond} for more details.\;
		\end{itemize}

		\begin{thm}\label{twoblocksPare1} For $1\leq r<s$,\;$E^{r+1}_r$ (equivalently,\;$\underline{\sL}(\Dpik)^{r+1}_{r}$) is uniquely determined by $\cC_{r}^{r+1}$.
		\end{thm}
		\begin{proof}To simplify notation,\;assume $s=2$.\;Write $\Dpik=E_1-E_2$.\;Suppose that $E_1$ has rank $l$.\;For $l+1\leq q\leq n$, recall that
			$E_0:=E_2^{(q-l)}$ is the unique $(\varphi,\Gamma)$-submodule of $E_2$ of rank $q-l$.\;Thus
			\[\Dpik=[\Dpik_1^{q}-(E_2/E_0)],\quad \Dpik_1^{q}:=[E_0'-E_1'].\]
			Quotienting by the basis $e_{l+1},\cdots,e_{q}$ of $D_{\mathrm{st}}(E_0)$ in the induced Hodge filtration on $D_{\mathrm{st}}(\Dpik_1^{q})$ (see (\ref{parameterfilwrite})) yields an induced Hodge filtration on $D_{\mathrm{st}}(E_1')$:
			\begin{equation}
				\begin{aligned}
					\fil^{j}_{H}D_{\mathrm{st}}(E_{1}')=\left\{
					\begin{array}{ll}
						D_{\mathrm{st}}(E_{1}'),\;&j\in (-\infty,-\bh_{q-l+1}],\\
						\fil^{-\bh_{q-l+3}}_{H}D_{\mathrm{st}}(E_{1}')\oplus E(\sum_{j=1}^{l}\cL_{j,q-l+1}e_{j}),\;&j\in (-\bh_{q-l+1},-\bh_{q-l+2}],\\
						\cdots&\cdots\\
						\fil^{-\bh_{q}}_{H}D_{\mathrm{st}}(E_{1}')\oplus E(\sum_{j=1}^{l}\cL_{j,q-1}e_{j}) ,\;& j\in (-\bh_{q-2},-\bh_{q-1}]\\
						E(\sum_{j=1}^{l}\cL_{j,q}e_{j}),\;&j\in (-\bh_{q-1},-\bh_q]\\
						0,\;& j\in (-\bh_q,+\infty).
					\end{array}
					\right.
				\end{aligned}
			\end{equation}
			In what follows,\;we use $\underline{\sL}(\Dpik)^{[l'_1,l'_2]}_{[l_1,l_2]}$ to denote the minor of $\underline{\sL}(\Dpik)$ formed by the $l_1$-th to $l_2$-th rows and the $l_1'$-th to $l_2'$-th columns.\;
			
			We now describe $\fil^{\bullet}_{H}D_{\mathrm{st}}(E_{1}')$ explicitly in terms of the matrix $\underline{\sL}(\Dpik)$.\;From $\underline{\sL}(\Dpik)^{[q-l+1,q]}_{[1,l]}$, via elementary column operations, we obtain an upper triangular matrix $\underline{\sL}(\Dpik)^{\urcorner}_{q}$ whose diagonal entries are all equal to $1$.\;Indeed,\;we first use the $q$-th column to make the $l$-th row equal to $(0,0,\cdots,0,1)$,\;then use the $(q-1)$-th column to make the $(l-1)$-th row equal to $(0,0,\cdots,1,\ast)$, and so on.\;This procedure is possible by the non-critical assumption.\;Then $(\underline{\sL}(\Dpik)^{\urcorner}_{q})w_{0,l}\in \bZ_l\backslash\bB_lw_{0,l}\bB_l/\bB_l$ determines the complete flag $\fil^{\bullet}_{H}D_{\mathrm{st}}(E_{1}')$ of $D_{\mathrm{st}}(E_{1}')$.\;

			On the other hand,\;for $1\leq t\leq l$,\;consider the unique quotient $E_1\twoheadrightarrow F_0=F_1^{(t)}=E_1/E_1^{(t)}$.\;Thus
			\[\Dpik=[E_1^{(t)}-\Dpik_{t+1}^{n}],\quad\Dpik_{t+1}^{n}=[E_2'-F_0'].\]
			Quotienting by the basis $e_{1},\cdots,e_{t}$ of $D_{\mathrm{st}}(E_1^{(t)})$,\;we get from (\ref{parameterfilwrite}) the induced Hodge filtration on $D_{\mathrm{st}}(\Dpik_{t+1}^{n})$,\;which restricts to the Hodge filtration on $D_{\mathrm{st}}(E_2')$:\;
			\begin{equation*}
				\begin{aligned}
					\fil^{j}_{H}D_{\mathrm{st}}(E_{2}')=\left\{
					\begin{array}{ll}
						D_{\mathrm{st}}(E_{2}'),\;&j\in (-\infty,-\bh_{t+1}],\\
						\fil^{-\bh_{t+3}}_{H}D_{\mathrm{st}}(E_{2}')\oplus E(\sum_{j=1}^{l}\cL_{j,t+1}e_{j}),\;&j\in (-\bh_{t+1},-\bh_{t+2}],\\
						\cdots&\cdots\\
						\fil^{-\bh_{t+(n-l)}}_{H}D_{\mathrm{st}}(E_{2}')\oplus E(\sum_{j=1}^{l}\cL_{j,t+(n-l)-1}e_{j}) ,\;& j\in (-\bh_{t+(n-l)-2},-\bh_{t+(n-l)-1}]\\
						E(\sum_{j=1}^{l}\cL_{j,t+(n-l)}e_{j}),\;&j\in (-\bh_{t+(n-l)-1},-\bh_{t+(n-l)}]\\
						0,\;& j\in (-\bh_{t+(n-l)},+\infty).
					\end{array}
					\right.
				\end{aligned}
			\end{equation*}
			Consider the $(n-l)$-order minors $\underline{\sL}(\Dpik)^{[t,t+(n-l)]}_{[t+1,t+(n-l)]}$ and $\underline{\sL}(\Dpik)^{[t,t+(n-l)]}_{[t+1,n]}$.\;Using the $[t+(n-l)+1,n]$-columns of $\underline{\sL}(\Dpik)^{[t,t+(n-l)]}_{[t+1,n]}$ and elementary column operations,\;we obtain a lower triangular matrix $\underline{\sL}(\Dpik)^{\llcorner}_{t}$ of $\underline{\sL}(\Dpik)^{[t+1,t+(n-l)]}_{[t+1,t+(n-l)]}$ whose diagonal entries are all $1$.\;Then $(\underline{\sL}(\Dpik)^{\llcorner}_{t})w_{0,n-l}\in \bZ_{n-l}\backslash\bB_{n-l}w_{0,n-l}\bB_{n-l}/\bB_{n-l}$ determines the complete flag $\fil^{\bullet}_{H}D_{\mathrm{st}}(E_{2}')$ of $D_{\mathrm{st}}(E_{2}')$.\;

			It remains to show that $\{\underline{\sL}(\Dpik)^{\urcorner}_{q}\}_{l+1\leq q\leq n}$ and 
			$\{\underline{\sL}(\Dpik)^{\llcorner}_{t}\}_{1\leq t\leq l}$ determine $\underline{\sL}(\Dpik)$.\;Indeed,\;for $l+1\leq q\leq n$,\;using the last column of each matrix $\underline{\sL}(\Dpik)^{\urcorner}_{q}$,\;we obtain the information of $\{\cL_{j,q}/\cL_{l,q}\}_{1\leq j<l,l+1\leq q\leq n}$.\;We then apply above discussion to  $(\Dpik^{\vee},F_0^{\vee})$,\;and obtain an explicit expression for $\underline{\sL}(\Dpik)^{\llcorner}_{t}$.\;Indeed, with respect to the dual basis of $D_{\mathrm{st}}(\Dpik^{\vee})\cong D_{\mathrm{st}}(\Dpik)^{\vee}$,\;the Hodge filtration on $D_{\mathrm{st}}(\Dpik^{\vee})$ is given by the matrix $w_0(\underline{\sL}(\Dpik)^{-1})^Tw_0$.\;Let $\underline{\sL}(\Dpik)^{-1}=(\cL'_{ij})$,\;then $w_0(\underline{\sL}(\Dpik)^{-1})^Tw_0=(\cL'_{n+1-j,n+1-i})_{i<j}$.\;In particular,\;$\{\underline{\sL}(\Dpik)^{\llcorner}_{t}\}_{1\leq t\leq l}$ encodes the information of $\{\cL'_{n-q,j}/\cL'_{n-q,n-l}\}_{\substack{1\leq j<n-l\\n-l+1\leq q\leq n}}$.\;Write 
			\[\underline{\sL}(\Dpik)=\left(\begin{array}{cc}
				A  & \underline{\sL}'(\Dpik)  \\
				0  & B
			\end{array}\right),\]
			where $A$ (resp.,\;$B$) is a matrix of size $l\times l$ (resp.,\;$(n-l)\times (n-l)$).\;Therefore,\;\[\underline{\sL}(\Dpik)^{-1}=\left(\begin{array}{cc}
				A^{-1}  & -A^{-1}\underline{\sL}'(\Dpik)B^{-1}  \\
				0  & B^{-1}
			\end{array}\right).\]
			Therefore,\;$\{\underline{\sL}(\Dpik)^{\llcorner}_{t}\}_{1\leq t\leq l}$ gives the matrix $\diag\{\cL'_{1,n-l},\cdots,\cL'_{l,n-l}\}^{-1}A^{-1}\underline{\sL}'(\Dpik)B^{-1}$, and $\{\underline{\sL}(\Dpik)^{\urcorner}_{q}\}_{l+1\leq q\leq n}$ gives the matrix $\underline{\sL}'(\Dpik)\diag\{\cL'_{l,l+1},\cdots,\cL'_{l,n}\}^{-1}$.\;The result follows.\;
		\end{proof}

		\begin{pro}\label{twoblocksPare2}
			Assume that $q>r+1$.\;Suppose that $\underline{\sL}(\Dpik)^{q'}_{r'}$ is already determined for all $r\leq r'<q'\leq q$ with $(r',q')\neq (r,q)$.\;Then $\cC_{r}^q(\Dpik)$ together with the lower-left entry $\cL_{t_{r-1}+l_r,t_{q-1}+1}$ of $\underline{\sL}(\Dpik)^q_{r}$ determines $\underline{\sL}(\Dpik)^q_{r}$ (equivalently,\;$\cC_{r}^q(\Dpik)$ determines $\underline{\sL}(\Dpik)^q_{r}$ up to one entry).\;
		\end{pro}
		\begin{proof}
			
			For each $1\leq t\leq l_q$,\;recall that $E_q^{(t)}\subseteq E_q$ is the unique $(\varphi,\Gamma)$-submodule of rank $t$.\;Then
			\[\Dpik=[(E_q^{(t)})'-(E_r^{q-1})'-F_q^{(t)}].\]
			By quotienting the basis of $D_{\mathrm{st}}((E_q^{(t)})')$ in (\ref{parameterfilwrite}),\;we obtain an induced Hodge filtration on $D_{\mathrm{st}}((E_r^{q-1})')$, and hence on $D_{\mathrm{st}}(E_r')$.\;Note that $\rk E_r^{q-1}=t_r^{q-1}$.\;Consider the $t_r^{q-1}$-order minor $\underline{\sL}(\Dpik)^{[t_{r-1}+1+t,t_{q-1}+t]}_{[t_{r-1}+1,t_{q-1}]}$.\;Using elementary column operations,\;we obtain an upper triangular matrix $\underline{\sL}(\Dpik)^{\urcorner}_{t}$ with all diagonal entries equal to $1$, supported on the $(t_{r-1}+1)$-th to $t_{r}$-th rows and the $(t_{r-1}+t+1)$-th to $(t_{r}+t)$-th columns.\;This determines the complete flag $\fil^{\bullet}_{H}D_{\mathrm{st}}(E_r')$ on $E_r'$.\;
			
			On the other hand,\;for each $1\leq t\leq l_r$,\;consider the quotient map $E_r\twoheadrightarrow F_r^{(t)}$ of $(\varphi,\Gamma)$-modules.\;Then
			\[\Dpik=[E_r^{(t)}-(E_{r+1}^q)'-(E_r^{(t)})'].\]
			By quotienting the basis of $D_{\mathrm{st}}(E_r^{(t)})$ in (\ref{parameterfilwrite}),\;we obtain an induced Hodge filtration $\fil^{\bullet}_{H}D_{\mathrm{st}}((E_{r+1}^q)')$ on $D_{\mathrm{st}}((E_{r+1}^q)')$.\;Consider the $t_{r}^q\times(t_{r}^q-t)$-order minor $\underline{\sL}(\Dpik)^{[t_{r-1}+1,t_q]}_{[t_{r-1}+t+1,t_{q}]}$.\;Using the $[t_q-(l_r-t)+1,t_q]$-columns of $\underline{\sL}(\Dpik)^{[t_{r-1}+1,t_q]}_{[t_{r-1}+t+1,t_{q}]}$ and elementary column operations,\;we transform $\underline{\sL}(\Dpik)^{[t_q-(l_r-t)+1,t_q]}_{[t_{r-1}+t,t_{r}]}$ (resp.,\;$\underline{\sL}(\Dpik)^{[t_{r-1}+1,t_q-(l_r-t)]}_{[t_{r-1}+t,t_{q}]}$) into an upper triangular matrix $\underline{\sL}(\Dpik)^{\llcorner}_{t}$ with unit diagonal entries (resp.,\;the zero matrix).\;Hence the induced Hodge filtration on $D_{\mathrm{st}}(E_q')$ is encoded by $\underline{\sL}(\Dpik)^{\llcorner}_{t}$.\;
			
			It remains to check that $\{\underline{\sL}(\Dpik)^{\urcorner}_{t}\}_{1\leq t\leq l_q}$ and 
			$\{\underline{\sL}(\Dpik)^{\llcorner}_{t}\}_{1\leq t\leq l_r}$ encode enough information.\;Let $\BQ_{\pm 1}^{(r,q)}$ be the space of rational functions in the matrix entries of $\{\underline{\sL}(\Dpik)_{r'}^{q'}\}_{r\leq r'<q'\leq q,(r',q')\neq (r,q)}$ with coefficients in $\{\pm1\}$.\;Then, for $1\leq t\leq l_q$,\;a careful inspection of the elementary column operations shows that the last column of $\underline{\sL}(\Dpik)^{\urcorner}_{t}$ equals $(H_{1,t}/H_{l_r,t},\cdots,H_{l_r,t}/H_{l_r,t})^T$,\;where 
			\[H_{j,t}:=H_{j,t}\big(\{\cL_{t_{r-1}+j,t_{q-1}+l}\}_{1\leq l\leq t}\big)\]
			are certain linear forms in the variables $\{\cL_{t_{r-1}+j,t_{q-1}+l}\}_{1\leq l\leq t}$ with coefficients in $\BQ_{\pm 1}^{(r,q)}$,\;for each $1\leq j<l_r$.\;In particular,\;if we regard $\{\cL_{t_{r-1}+l_r,t_{q-1}+t}\}_{1\leq t\leq l_q}$ as free variables,\;then each $\cL_{t_{r-1}+j,t_{q-1}+t}$ can be expressed as a linear combination of these free variables with known coefficients.\;
			
		With respect to the dual basis of $D_{\mathrm{st}}(\Dpik^{\vee})\cong D_{\mathrm{st}}(\Dpik)^{\vee}$,\;the Hodge filtration on $D_{\mathrm{st}}(\Dpik^{\vee})$ is given by the matrix $w_0(\underline{\sL}(\Dpik)^{-1})^Tw_0$,\;i.e.,\;if $\underline{\sL}(\Dpik)^{-1}=(\cL'_{ij})$,\;then $w_0(\underline{\sL}(\Dpik)^{-1})^Tw_0=(\cL''_{ij})_{i<j}$ with $\cL''_{ij}:=\cL'_{n+1-j,n+1-i}$.\;If we write
			\[\underline{\sL}(\Dpik)=\left(\begin{array}{cc}
				A  & \underline{\sL}''(\Dpik)  \\
				0  & B
			\end{array}\right),\]
			where $A$ (resp.,\;$B$) is a matrix of order $t_r^{q-1}\times t_r^{q-1}$ (resp.,\;$l_q\times l_q$).\;Then we see that
			\begin{equation}\label{inverseofLD}
				\begin{aligned}
					\underline{\sL}(\Dpik)^{-1}&=\left(\begin{array}{cc}
						A^{-1}  & -A^{-1}\underline{\sL}''(\Dpik)B^{-1}  \\
						0  & B^{-1}
					\end{array}\right).
				\end{aligned}
			\end{equation}
			Thus the last column of $\underline{\sL}(\Dpik)^{\llcorner}_{t}$ equals $(H'_{1,t}/H'_{l_q,t},\cdots,H'_{l_q,t}/H'_{l_q,t})^T$,\;where
			\[H'_{j,t}:=H'_{j,t}\big(\{\cL''_{t_{q+1}^s+j,t^s_{r-1}+l}\}_{1\leq l\leq l_r-t+1}\big)=H'_{j,t}\big(\{\cL'_{t_r+1-l,t_q+1-j}\}_{1\leq l\leq l_r-t+1}\big)\]
			are certain linear forms in the variables $\{\cL'_{t_r+1-l,t_q+1-j}\}_{1\leq l\leq l_r-t}$ with coefficients in $\BQ_{\pm 1}^{(r,q)}$ (using (\ref{inverseofLD})),\;for each $1\leq l\leq l_r$.\;Furthermore,\;using (\ref{inverseofLD}) again,\;we see that
			\[H'_{j,t}\big(\{\cL'_{t_r+1-l,t_q+1-j}\}_{1\leq l\leq l_r-t+1}\big)=H''_{j,t}\big(\{\cL_{t_r+1-l,t_q+1-j}\}_{1\leq l\leq l_r-t+1}\big)\]
			for another family of linear forms $H''_{j,t}$ in the variables $\{\cL_{t_r+1-l,t_q+1-j}\}_{1\leq l\leq l_r-t}$ with coefficients in $\BQ_{\pm 1}^{(r,q)}$.\;In particular,\;taking $t=l_r$,\;the collection
			\[\Big\{\frac{H''_{j,t}(\cL_{t_{r-1}+l_r,t_q+1-j})}{H''_{l_r,q}(\cL_{t_{r-1}+l_r,t_q+1-l_q})}\Big\}_{1\leq j<l_q}\]
			is captured by $\underline{\sL}(\Dpik)^{\llcorner}_{l_r}$,\;so we can express $\cL_{t_{r-1}+l_r,t_q+1-j}$ in terms of $\cL_{t_{r-1}+l_r,t_q+1-l_q}$ and $\underline{\sL}(\Dpik)^{\llcorner}_{l_r}$.\;Since $\cL_{t_{r-1}+j,t_{q-1}+t}$ is a linear combination of $\{\cL_{t_{r-1}+l_r,t_{q-1}+t}\}_{1\leq t\leq l_q}$ with known coefficients,\;we obtain
			\[\cL_{t_{r-1}+j,t_{q-1}+t}=A_{j,t}\cL_{t_{r-1}+l_r,t_q+1-l_q}+B_{j,t},\]
			where $A_{j,t}$ and $B_{j,t}$ are known coefficients.\;This completes the proof.\;
		\end{proof}

			\subsection{``Crystalline'' Hodge parameters inside $\Dpik$}\label{cryparablocks}

			This section follows the strategy in \cite[Sections 2.2-2.4]{ParaDing2024} and develops a block version.\;Our goal is to extract the Hodge parameters of $\Dpik$ that arise from crystalline subquotients of $\Dpik$ or the ``crystalline'' Hodge parameters between Steinberg subquotient of $\Dpik$ .\;

			Let $1\leq r<q\leq s$.\;Consider $E_r^q=[E_{r}^{q-1}-E_{q}]=[E_q'-(E_{r}^{q-1})']$.\;Then the composition $E_r^{r+1}\hookrightarrow E_r^q\twoheadrightarrow (E_{r}^{q-1})']$ factors through $(E_r^{r+1})'$,\;we obtain a map 
			\[\iota^{r,q}_{\Dpik}:
			E_r^{r+1}\rightarrow (E_r^{r+1})'.\]
			Similar to \cite[Proposition 2.1]{ParaDing2024},\;if $\dim_E\homo_{(\varphi,\Gamma)}(E_r^{r+1},(E_r^{r+1})')=2$ (hence $q>r+1$),\;then $\iota^{r,q}_{\Dpik}$ encodes one additional parameter.\;

			On the other hand,\;any maximal crystalline subquotient $\Dpik_0$ of $\Dpik$ has the form:
			\begin{equation}\label{formforcrysubquo}
				\begin{aligned}
					\Dpik_0=[\cR_E(\unr(\alpha_1p^{j_1})z^{\bh_{b+1}})-\cR_E(\unr(\alpha_2p^{j_2})z^{\bh_{b+2}})-\cdots-\cR_E(\unr(\alpha_sp^{j_s})z^{\bh_{b+s}})].\;
				\end{aligned}
			\end{equation}
			for integers $0\leq j_i\leq l_i-1$ ($1\leq i\leq s$) and some $0\leq b\leq n-s$.\;In particular,\;$\Dpik$ admits a unique (maximal) crystalline $(\varphi,\Gamma)$-submodule (resp.,\;quotient) $\Dpik_{\mathrm{cr}}$ (resp.,\;$\Dpik^{\mathrm{cr}}$).\;More precisely,\;
			\begin{equation}
				\begin{aligned}
					\Dpik_{\mathrm{cr}}=&\;[\cR_E(\unr(\alpha_1)z^{\bh_{1}})-\cR_E(\unr(\alpha_2)z^{\bh_{2}})-\cdots-\cR_E(\unr(\alpha_s)z^{\bh_{s}})],\\
					\Dpik^{\mathrm{cr}}=&\;[\cR_E(\unr(\alpha_1p^{l_1-1})z^{\bh_{n-s+1}})-\cR_E(\unr(\alpha_2p^{l_2-1})z^{\bh_{n-s+2}})-\cdots-\cR_E(\unr(\alpha_sp^{l_s-1})z^{\bh_{n}})].\;
				\end{aligned}
			\end{equation}
			
			\begin{rmk}
				We explain why we consider crystalline subquotients.\;First,\;the Hodge parameters in $\cC_{\mathrm{ST}}(\Dpik)$ and the ``crystalline'' Hodge parameters are tightly intertwined, and crystalline subquotients provide a way to isolate the purely crystalline part.\;Second,\;the Hodge parameters of crystalline subquotients can be captured by explicit locally analytic representations, which suggests an internal structure for the conjectural locally analytic representation $\pi(\Dpik)$.\;
			\end{rmk}
			
			We apply the above discussion to a crystalline subquotient $\Dpik_0$ of $\Dpik$ of the form (\ref{formforcrysubquo}) (i.e.,\;$\rk(E_i)=1$ for all $1\leq i\leq s$).\;For $1\leq r<q\leq s$,\;recall
			\[(\Dpik_0)_r^q:=[\cR_E(\unr(\alpha_rp^{j_r})z^{\bh_{b+r}})-\cdots-\cR_E(\unr(\alpha_qp^{j_q})z^{\bh_{b+q}})].\]
			If  $q>r+1$,\;the following refinements of $\Dpik_0$
			\begin{equation}\label{Dcrlqpermutation}
				\begin{aligned}
					(\Dpik_0)_r^q=[(\Dpik_0)_r^{q-1}-\cR_E(\unr(\alpha_qp^{j_q})z^{\bh_{b+q}})]=[\cR_E(\unr(\alpha_qp^{j_q})z^{\bh_{b+r}})-((\Dpik_0)_r^{q-1})']\end{aligned}
			\end{equation}
		induces a natural map
			\[\iota^{r,q}_{\Dpik_0}:(\Dpik_0)_r^{r+1}\rightarrow (\Dpik_0)^{r+1}_{r,[q]}:=((\Dpik_0)_r^{r+1})'.\]
			By \cite[Remark 2.3]{ParaDing2024},\;we have 
			\begin{lem}
				We have $\dim_E\homo_{(\varphi,\Gamma)}((\Dpik_0)_r^{r+1},((\Dpik_0)_r^{r+1})')=2$.\;
			\end{lem}
			
			\begin{thm}\label{higherdetercry}
				The Hodge parameters of $\Dpik_0$ are determined by the data $\{\iota^{r,q}_{\Dpik_0}\}_{1\leq r<r+1<q\leq s}$.\;
			\end{thm}
			\begin{proof}The case $s=2$ is clear.\;Assume the theorem holds for every crystalline $(\varphi,\Gamma)$-module of rank $s-1$.\;We prove it for rank $s$.\;Note that $\Dpik_0=[(\Dpik_0)_{1}^{s-1}-\cR_E(\unr(\alpha_sp^{j_s})z^{\bh_{b+s}})]=[\cR_E(\unr(\alpha_1p^{j_1})z^{\bh_{b+1}})-(\Dpik_0)_{2}^{s}]$.\;By the induction hypothesis,\;the $p$-adic Hodge parameters of $(\Dpik_0)_{1}^{s-1}$ are determined by $\{\iota^{r,q}_{\Dpik_0}\}_{1\leq r<r+1<q\leq s-1}$.\;Hence it remains to determine the line $e_{b+s}+\cL_{b+s-1,b+s}e_{b+s-1}+\cdots+\cL_{b+1,b+s}e_{b+1}$ in $\fil^{-\bh_{b+s}}_{H}D_{\mathrm{cr}}(\Dpik_0)$.\;Applying the induction hypothesis to $(\Dpik_0)_{2}^{s}$,\;we know that $\{\iota^{r,q}_{\Dpik_0}\}_{2\leq r<r+1<q\leq s}$ determines the vector $e_{b+s}+\cL_{b+s-1,b+s}e_{b+s-1}+\cdots+\cL_{b+2,b+s}e_{b+2}$.\;Therefore it suffices to determine the remaining parameter $\cL_{b+1,b+s}$.\;Write $(D_0)_1^{s-1}:=D_{\mathrm{cr}}(((\Dpik_0)_1^{s-1})')$.\;The induced Hodge filtration on $(D_0)_1^{s-1}$ is:\;
				\begin{equation}
					\begin{aligned}
						\fil^{j}_{H}((D_0)_1^{s-1})=\left\{
						\begin{array}{ll}
							(D_0)_1^{s-1},\;&j\in (-\infty,-\bh_{b+1}],\\
							\fil^{-\bh_{b+3}}_{H}((D_0)_1^{s-1})\oplus E(e_{b+2}+\cL_{b+1,b+2}e_{b+1}),\;&j\in (-\bh_{b+1},-\bh_{b+2}],\\
							\cdots&\cdots\\
							\fil^{-\bh_{b+s}}_{H}((D_0)_1^{s-1})\oplus E(e_{b+s-1}+\sum_{l=1}^{s-2}\cL_{b+l,b+s-1}e_{b+l}) ,\;& j\in (-\bh_{b+s-2},-\bh_{b+s-1}]\\
							E(e_{b+s-1}+\sum_{l=1}^{s-2}\cL_{b+l,b+s}e_{b+l}),\;&j\in (-\bh_{b+s-1},-\bh_{b+s}]\\
							0,\;& j\in (-\bh_{b+s},+\infty).
						\end{array}
						\right.
					\end{aligned}
				\end{equation}
				Using $\fil^{-\bh_{b+s}}_{H}((D_0)_1^{s-1})$,\;we can translate the line $E(e_{b+s-1}+\sum_{l=1}^{s-2}\cL_{b+l,b+s-1}e_{b+l})$ in $\fil^{-\bh_{b+s-1}}_{H}((D_0)_1^{s-1})$ into the $E$-line
				\[E\Big(e_{b+s-2}+\sum_{l=1}^{s-3}\frac{\cL_{b+l,b+s-1}-\cL_{b+l,b+s}}{\cL_{b+s-2,b+s-1}-\cL_{b+s-2,b+s}}e_{b+l}\Big).\]
				Using this new $E$-line,\;we can translate the line $E(e_{b+s-2}+\sum_{l=1}^{s-3}\cL_{b+l,b+s-2}e_{b+l})$ in $\fil_{\bh_{b+s-1}}^{H}((D_0)_1^{s-1})$ into an $E$-line $E(e_{b+s-3}+\sum_{l=1}^{s-3}\cL'_{b+l,b+s-3}e_{b+l})$.\;Note that $\cL_{b+1,b+s}$ appears in these coefficients $\cL'_{b+l,b+s-3}$.\;Iterating this procedure,\;we eventually obtain an $E$-line $E(e_{b+2}+\cL'_{b+1,b+2}e_{b+1})$ in which $\cL_{b+1,b+s}$ appears in $\cL'_{b+1,b+2}$.\;Since $\cL'_{b+1,b+2}$ is captured by the morphism $\iota^{1,s}_{\Dpik_0}$,\;the parameter $\cL_{b+1,b+s}$ is determined by $\iota^{1,s}_{\Dpik_0}$.\;This completes the induction step.\;
			\end{proof}
			Theorem \ref{higherdetercry} is equivalent to the following statement.\;The moduli space of non-critical crystalline $p$-adic Hodge parameters is:
			\[\Phi_{\mathrm{nc},s}^{\mathrm{{cr}}}:=\Phi_{\mathrm{nc},s}(\emptyset)=
			\bigcap_{u\in \sW_{s}}\bT_{s}\backslash u^{-1}\bB_{s} w_{0,s}\bB_{s}/\bB_{s} \subseteq \bT_{s}\backslash \bB_{s} w_{0,s}\bB_{s}/\bB_{s}.\] 
			
			\begin{cor}
				The following morphism is injective:
				\begin{equation}
					\begin{aligned}
						\Phi_{\mathrm{nc},s}^{\mathrm{{cr}}}&\rightarrow\prod_{1\leq r<r+1<q\leq s}\bT_2\backslash\left(\GLN_{2}/\bB_2\times \GLN_{2}/\bB_2\right),\\
						[\underline{\sL}(\Dpik_0)]&\mapsto \Big[\Big(\fil^{\bullet}_{H}D_{\mathrm{cris}}((\Dpik_0)_r^{r+1}),\fil^{\bullet}_{H}D_{\mathrm{cris}}((\Dpik_0)^{r+1}_{r,[q]})\Big)\Big],
					\end{aligned}
				\end{equation}
				where $\GLN_{2}/\bB_2\times \GLN_{2}/\bB_2$ carries the diagonal action of $\bT_2$.\;
			\end{cor}

			\subsection{Capture all the Hodge parameters via different refinements}\label{mainthmforcapturePara}
			
			In this section,\;we prove the following theorems.\;
			\begin{thm}\label{capturetype1} $\Dpik$ is uniquely determined by the data $\cC_{\mathrm{ST}}(\Dpik)$ and $\{\iota^{r',q'}_{\Dpik^{\mathrm{cr}}}\}_{1\leq r'<q'\leq s}$ (equivalently,\;$\Dpik^{\mathrm{cr}}$).\;
			\end{thm}
			
			\begin{thm}\label{capturetype1diamond} $\Dpik$ is uniquely determined by the data $\cC_{\mathrm{ST}}(\Dpik)$ and $\{\iota^{r,q}_{\Dpik}\}_{1\leq r\leq r+1<q\leq s}$.\;
			\end{thm}
				\begin{proof}[Proof of Theorem \ref{capturetype1}] We prove Theorem \ref{capturetype1} by induction on $s$.\;There is nothing to prove for $s=1$, and the case $s=2$ is Theorem \ref{twoblocksPare1}.\;Assume the statement holds for every semistable $(\varphi,\Gamma)$-module of rank $n-1$, and prove it for rank $n$.\;Write
				$\Dpik=[\Dpik_{1}^{n-1}-\cR_E(\unr(\phi_n)z^{\bh_n})]=[\cR_E(\unr(\phi_1)z^{\bh_1})-\Dpik_{2}^{n}]$.\;
				In particular,\;there is a surjection $E_s\twoheadrightarrow \cR_E(\unr(\phi_n)z^{\bh_n})$ (resp.,\;an injection $\cR_E(\unr(\phi_1)z^{\bh_1})\hookrightarrow E_1$) of $(\varphi,\Gamma)$-modules.\;By the induction hypothesis,\;$\Dpik_{1}^{n-1}$ (resp.,\;$\Dpik_{2}^{n}$) is determined by the corresponding data.\;Using $\Dpik_1^{n-1}$,\;it remains to determine the line
				\[\fil^{-\bh_{n}}_{H}D_{\mathrm{st}}(\Dpik)=E(e_n+\cL_{n-1,n}e_{n-1}+\cdots+\cL_{1,n}e_1).\]
				On the other hand,\;the vector $e_n+\cL_{n-1,n}e_{n-1}+\cdots+\cL_{2,n}e_2$ is already captured by $D_{2}^{n}$, so only the parameter $\cL_{1,n}$ remains.\;If $\Dpik$ is crystalline, this is exactly Theorem \ref{higherdetercry}.\;Hence we may assume that $\Dpik$ is semistable non-crystalline and that $l_1>1$.\;Write $\Dpik=[(E_2^s)'-E_1']$ and consider the following induced Hodge filtration on $D_{\mathrm{st}}(E_1')$ (obtained by quotienting the basis $e_{l_1+1},\cdots,e_{n}$ of $D_{\mathrm{st}}((E_2^s)')$ in (\ref{parameterfilwrite})):\;$\fil^{j}_{H}D_{\mathrm{st}}(E_{1}')=$
				\begin{equation}
					\begin{aligned}
						\left\{
						\begin{array}{ll}
							D_{\mathrm{st}}(E_{1}'),\;&j\in (-\infty,-\bh_{n-l_1+1}],\\
							\fil^{-\bh_{n-l_{1}+3}}_{H}D_{\mathrm{st}}(E_{1}')\oplus E(\sum_{l=1}^{l_1}\cL_{l,n-l_{1}+2}e),\;&j\in (-\bh_{n-l_{1}+1},-\bh_{n-l_{1}+2}],\\
							\cdots&\cdots\\
							\fil^{-\bh_{n}}_{H}D_{\mathrm{st}}(E_{1}')\oplus E(\sum_{l=1}^{l_{1}}\cL_{l,n-1}e) ,\;& j\in (-\bh_{n-2},-\bh_{n-1}]\\
							E(\sum_{l=1}^{l_{1}}\cL_{l,n}e),\;&j\in (-\bh_{n-1},-\bh_n]\\
							0,\;& j\in (-\bh_n,+\infty).
						\end{array}
						\right.
					\end{aligned}
				\end{equation}
				In particular,\;the line $E(\sum_{l=1}^{l_{1}}\cL_{l,n}e)$ is captured by $E_1'\in \cC_{\mathrm{ST}}(\Dpik)$.\;Since $\{\cL_{l,n}\}_{2\leq l\leq l_1}$ is already determined by the induction hypothesis,\;$\cL_{1,n}$ is also determined.\;This completes the induction step.\;
			\end{proof}
			
			\begin{proof}[Proof of Theorem \ref{capturetype1diamond}]
				We prove the theorem by induction on $s$.\;The statement is trivial for $s=1$, and the case $s=2$ is Theorem \ref{twoblocksPare1}.\;Assume it holds for $s-1$.\;Applying the induction hypothesis to $E_{1}^{s-1}$ and $E_{2}^s$,\;it remains to determine $\underline{\sL}(\Dpik)_1^s$.\;By Proposition \ref{twoblocksPare2},\;it is enough to determine the lower-left entry $\cL_{l_1,t_{s-1}+1}$ of $\underline{\sL}(\Dpik)^s_{1}$.\;We show that $\cL_{l_1,t_{s-1}+1}$ is determined by $\iota^{1,s}_{\Dpik}$.\;Write
				\[\Dpik=[E_1^{s-1}-E_s]=[E_s'-(E_1^{s-1})'].\]
				By quotienting the basis of $D_{\mathrm{st}}(E_q)$ in (\ref{parameterfilwrite}),\;we obtain an induced Hodge filtration $\fil^{\bullet}_{H}D_{\mathrm{st}}((E_r^{q-1})')$ on $D_{\mathrm{st}}((E_r^{q-1})')$.\;Consider the $t_{s-1}$-order minor $\underline{\sL}(\Dpik)^{[1+l_s,n]}_{[1,t_{s-1}]}$.\;Using elementary column operations,\;we transform $\underline{\sL}(\Dpik)^{[1+l_s,n]}_{[1,t_{s-1}]}$ into an upper triangular matrix $\underline{\sL}(\Dpik)^{\urcorner}$ with unit diagonal entries.\;Then $(\underline{\sL}(\Dpik)^{\urcorner})^{[l_s+1,l_s+l_1+l_2]}_{[1,l_1+l_2]}$ determines the complete flag $\fil_H^{\bullet}D_{\mathrm{st}}((E_1^2)')$.\;Hence the map $\iota^{1,s}_{\Dpik}:E_1^{2}\rightarrow (E_1^{2})'$ induces an injection of filtered $(\varphi,N)$-modules
				\[\iota^{1,s}_{\Dpik}:\fil_H^{\bullet}D_{\mathrm{st}}(E_1^2)\rightarrow\fil_H^{\bullet}D_{\mathrm{st}}((E_1^2)'),\]
				whose underlying map of $(\varphi,N)$-modules is the identity on $D_{\mathrm{st}}(E_1^2)\rightarrow D_{\mathrm{st}}((E_1^2)')$.\;In particular,\;$\iota^{1,s}_{\Dpik}$ compares the two matrices $(\underline{\sL}(\Dpik))^{[1,l_1+l_2]}_{[1,l_1+l_2]}$ and $(\underline{\sL}(\Dpik)^{\urcorner})^{[l_s+1,l_s+l_1+l_2]}_{[1,l_1+l_2]}$.\;This comparison determines $\cL_{l_1,t_{s-1}+1}$,\;since $\cL_{l_1,t_{s-1}+1}$ appears in $\fil_H^{\bh_{l_s+l_1}}D_{\mathrm{st}}((E_1^2)')$ by tracing the elementary column operations.\;
			\end{proof}
			
			Let $u_{0}\in \sW_n^{S_0}$  be an element satisfying $\Dpik^{\mathrm{cr}}=[R_{u_0,n-s+1}-\cdots-R_{u_0,n}]$.\;Note that ${\bL}_{\widehat{n-s-1}}\cong \GLN_{n-s}\times \GLN_s$.\;Let $\bT_s\subseteq \bB_s\subseteq \GLN_s$ be the corresponding torus and standard Borel subgroup.\;We have the following natural maps:
			\begin{equation}
				\begin{aligned}
					p_{\widehat{n-s-1}}:\bB w_0\bB\cong \bN_{\emptyset}\twoheadrightarrow \bN_{\emptyset}/\bN_{\widehat{n-s-1}}\cong \bL^{\bB}_{\widehat{n-s-1}}w_{\widehat{n-s-1},0}\bL^{\bB}_{\widehat{n-s-1}}/\bL^{\bB}_{\widehat{n-s-1}}
					\twoheadrightarrow \bB_s w_{0,s}\bB_s/\bB_s,\;
				\end{aligned}
			\end{equation}
			where $w_{0,s}$ is the longest element in $\sW_s$.\;Therefore,\;we obtain a map $\Phi_{\mathrm{nc},\Delta}(S_0) \rightarrow \bT_{s}\backslash \bB_s w_{0,s}\bB_s/\bB_s$ that sends$[g\bB]$ to $[p_{\widehat{n-s-1}}(u_0^{-1}g\bB)]$.\;By the non-critical assumption,\;this map factors through
			\begin{equation}
				p_{\mathrm{cr}}:\Phi_{\mathrm{nc},\Delta}(S_0)\rightarrow \Phi_{\mathrm{nc},s}^{\mathrm{{cr}}}.\;
			\end{equation}
			Recall the natural morphism $p_{\mathrm{ST}}:\Phi_{\mathrm{nc},\Delta}(S_0)\rightarrow \prod_{u\in \sW_n^{S_0}}\Phi_{\mathrm{nc},S_0(u)}(S_0(u))$ in (\ref{dfnofpreftoLevi}).\;Then Theorem \ref{capturetype1} is equivalent to the following theorem.\;
			\begin{thm}\label{capturetype1flag} The following morphism is injective:
				\begin{equation}\label{parametersmaininj}
					\Phi_{\mathrm{nc},\Delta}(S_0)\xrightarrow{(p_{\mathrm{ST}},p_{\mathrm{cr}})}  \prod_{u\in \sW_n^{S_0}}\Phi_{\mathrm{nc},S_0(u)}(S_0(u)) \times \Phi_{\mathrm{nc},s}^{\mathrm{{cr}}}.\;
				\end{equation}
			\end{thm}
			
				We reinterpret (\ref{dfnofpreftoLevi}) in terms of $E$-Lie algebras.\;For each $u\in \sW_n^{S_0}$,\;we deduce from the non-critical assumption that $g\bB=u^{-1}b_uw_{0}\bB$ for some $b_u\in \bB(E)$ (of course $b_u$ depends on $u$).\;Note that $\overline{\fb}=\mathrm{Ad}_{w_0}(\fb)$ coincides with the Borel algebra of lower triangular matrices.\;For $u\in \sW_n^{S_0}$,\;note that
				\[\mathrm{Ad}_{u^{-1}}(\fl_{S_0(u)})\cap \mathrm{Ad}_{g}(\fb)= \mathrm{Ad}_{u^{-1}}(\fl_{S_0(u)}\cap \mathrm{Ad}_{p_{\mathrm{ST},u}(b_u)}(\fl_{S_0(u)}\cap \overline{\fb})).\]
				Then $p_{\mathrm{ST}}$ is equivalent to the following map:
				\[\mathrm{Ad}_{g}(\fb)_{/\mathrm{Ad}(\bZ_{S_0})}\mapsto \Big((\fl_{S_0(u)}\cap \mathrm{Ad}_{u^{-1}g}(\fb))_{/\mathrm{Ad}(\bZ_{S_0(u)})}\Big)_{u\in \sW_n^{S_0}}.\]
					Moreover,\;consider the natural morphism:
			\begin{equation}
				g_{S_0}^{\circ}:\bigoplus_{u\in \sW_n^{S_0}}\mathrm{Ad}_{u^{-1}}(\tau_{S_0(u)})\cap \mathrm{Ad}_{g}(\fb)\rightarrow \mathrm{Ad}_{g}(\fb).
			\end{equation}
			Denote the image of $g_{S_0}^{\circ}$ by
			\[\mathrm{Ad}_{g}(\fb)^{\circ}_{S_0}:=\sum_{u\in \sW_n^{S_0}}\mathrm{Ad}_{u^{-1}}(\tau_{S_0(u)})\cap \mathrm{Ad}_{g}(\fb).\]
			Then we have an exact sequence:
			\begin{equation}
				0\rightarrow \ker(g_{S_0}^{\circ})\rightarrow \bigoplus_{u\in \sW_n^{S_0}}\mathrm{Ad}_{u^{-1}}(\tau_{S_0(u)})\cap \mathrm{Ad}_{g}(\fb)\rightarrow \mathrm{Ad}_{g}(\fb)^{\circ}_{S_0}\rightarrow 0
			\end{equation}
			Theorem \ref{capturetype1diamond} can be reformulated as the following Lie-algebras statement:
			\begin{thm}\label{capturetype1Liealg}
				$\ker(g_{S_0}^{\circ})$ and $\big\{(\fl_{S_0(u)}\cap \mathrm{Ad}_{u^{-1}g}(\fb))_{_{/\mathrm{Ad}(\bZ_{S_0(u)})}}\big\}_{u\in \sW_n^{S_0}}$ determine $[g\bB]\in \Phi_{\mathrm{nc},\Delta}(S_0).\;$
			\end{thm}

			Moreover,\;consider the following natural surjection:
			\begin{equation}
				\begin{aligned}
					\Phi_{\mathrm{nc},S_0(u)}(S_0(u))=\;\bZ_{S_0(u)}\backslash (\bL^{\bB}_{S_0(u)}) &w_{S_0(u),0}(\bL^{\bB}_{S_0(u)})/(\bL^{\bB}_{S_0(u)})\\
					&\;	\twoheadrightarrow \prod_{j\in S_0(u)}\bG_m\backslash(\bL^{\bB}_{j}) w_{\{j\},0}(\bL^{\bB}_{j})/(\bL^{\bB}_{j}),
				\end{aligned}
			\end{equation}
			which extracts the Hodge parameters of each $2\times 2$ block in $\Phi_{\mathrm{nc},S_0(u)}(S_0(u))$ (in the language of \cite{2019DINGSimple},\;these are called simple Hodge parameters).\;For $u\in \sW_n^{S_0}$,\;we obtain maps:
			\begin{equation}\label{parametersmaininjtosimple}
				\begin{aligned}
					p^{\Delta}_{\mathrm{ST},u}&:	\Phi_{\mathrm{nc},\Delta}(S_0)\rightarrow  \prod_{j\in S_0(u)}\bG_m\backslash(\bL^{\bB}_{j}) w_{\{j\},0}(\bL^{\bB}_{j})/(\bL^{\bB}_{j}),\\
					p^{\Delta}_{\mathrm{ST}}:=\prod_{u\in \sW_n^{S_0}}p^{\Delta}_{\mathrm{ST},u}&:\Phi_{\mathrm{nc},\Delta}(S_0)\rightarrow  \prod_{u\in \sW_n^{S_0}}\prod_{j\in S_0(u)}\bG_m\backslash(\bL^{\bB}_{j}) w_{\{j\},0}(\bL^{\bB}_{j})/(\bL^{\bB}_{j}).
				\end{aligned}
			\end{equation}
			\begin{cor}\label{captureparameters2} 
				If $\max_{1\leq i\leq s}l_i\leq 2$,\;then $p^{\Delta}_{\mathrm{ST}}=p_{\mathrm{ST}}$.\;
			\end{cor}

			\section{Deformations of semistable $(\varphi,\Gamma)$-modules}

			\subsection{Deformations of $\cF_{u}$ and $\cF_{S_0(u)}$ for $u\in \sW_n^{S_0}$}\label{generalforparedefor}
			
			For $u\in \sW_n^{S_0}$,\;let $\ext^1_u(\Dpik,\Dpik)$ (resp.,\;$\ext^{1}_{S_0(u)}(\Dpik,\Dpik)$) be the subspace of  $ \ext^1(\Dpik,\Dpik)$ of trianguline deformations (resp.,\;$\bP_{S_0(u)}$-parabolic deformations) of $\Dpik$ with respect to $\cF_u$ (resp.,\;$\cF_{S_0(u)}$),\;i.e.,\;$\widetilde{D}\in \ext^{1}_{u}(\Dpik,\Dpik)$ (resp.\;$\ext^{1}_{S_0(u)}(\Dpik,\Dpik)$) if and only if
			\[\widetilde{\Dpik}=[\widetilde{R}_{u,1}-\widetilde{R}_{u,2}-\cdots-\widetilde{R}_{u,n}],\;\text{resp.,\;}\widetilde{\Dpik}=[\widetilde{E}_{u,1}-\widetilde{E}_{u,2}-\cdots-\widetilde{E}_{u,f_u}]\]
			where $\widetilde{R}_{u,l}\cong {R}_{u,l}\otimes_{\cR_{E}}\cR_{E[\epsilon]/\epsilon^2}(1+\psi_l\epsilon)$ for $\psi_l\in \homo(\bQ_p^{\times},E)$ (resp.,\;$\widetilde{E}_{u,i}$ is a deformation of ${E}_{u,i}$ over $\cR_{E[\epsilon]/\epsilon^2}$).\;We further define a subspace $\ext^{1,\flat}_{S_0(u)}(\Dpik,\Dpik)$ of $\ext^{1}_{S_0(u)}(\Dpik,\Dpik)$, i.e.,\;$\widetilde{D}=[\widetilde{E}_{u,1}-\widetilde{E}_{u,2}-\cdots-\widetilde{E}_{u,f_u}]\in \ext^{1}_{S_0(u)}(\Dpik,\Dpik)$ belongs to $\ext^{1,\flat}_{S_0(u)}(\Dpik,\Dpik)$ if and only if $\widetilde{E}_{u,i}\cong {E}_{u,i}\otimes_{\cR_{E}}\cR_{E[\epsilon]/\epsilon^2}(1+\psi_i'\epsilon)$ for $\psi'_i\in \homo(\bQ_p^{\times},E)$.\;We have natural maps:
			\begin{equation}\label{dfnforkappau}
				\begin{aligned}
					\kappa_{u}&:\ext^{1}_{u}(\Dpik,\Dpik)\rightarrow \homo(\bT(\bQ_p),E),\widetilde{D}\mapsto (\psi_i)_{1\leq i\leq n}\\
					\kappa_{S_0(u)}&:\ext^{1,\flat}_{S_0(u)}(\Dpik,\Dpik)\rightarrow \homo(\bZ_{S_0(u)}(\bQ_p),E),\;\widetilde{D}\mapsto (\psi'_i)_{1\leq i\leq f_u}.\\
				\end{aligned}
			\end{equation}
			Note that $\ext^{1,\flat}_{S_0(u)}(\Dpik,\Dpik)\subseteq \ext^{1}_{u}(\Dpik,\Dpik)\subseteq \ext^{1}_{S_0(u)}(\Dpik,\Dpik)$.\;
			
			For $u\in \sW_n^{S_0}$ and $i\in \Delta$,\;we further define a subspace $\ext^{1}_{u,[i]}(\Dpik,\Dpik)$ of $\ext^{1}(\Dpik,\Dpik)$,\;i.e.,\;$\widetilde{D}\in \ext^{1}(\Dpik,\Dpik)$ belongs to $\ext^{1}_{u,[i]}(\Dpik,\Dpik)$ if $\widetilde{D}=[\widetilde{R}_{u,1}-\cdots-\widetilde{R}_{u,i-1}-\widetilde{R_{u,i}^{i+1}}-\widetilde{R}_{u,i+2}-\cdots-\widetilde{R}_{u,n}]$,\;where $\widetilde{R}_{u,l}\cong {R}_{u,l}\otimes_{\cR_{E}}\cR_{E[\epsilon]/\epsilon^2}(1+\psi_l\epsilon)$ for $\psi_l\in \homo(\bQ_p^{\times},E)$ and $l\neq i,i+1$,\;and $\widetilde{R_{u,i
				}^{i+1}}$ is a deformation of $R_{u,i}^{i+1}$ over $E[\epsilon]/\epsilon^2$ (see Remark \ref{dfnforRii1}).\;Note that $\ext^{1}_{u}(\Dpik,\Dpik)\subseteq \ext^{1}_{u,[i]}(\Dpik,\Dpik)$.\;

			Let $R^+_u:=\{(i,j):1\leq i<j\leq n,u^{-1}(j)=u^{-1}(i)+1,u^{-1}(i)\in I_0\}$.\;Put
			\begin{equation}\label{dfnforimagehomou}
				\homo_u(\bT(\bQ_p),E):=\left\{(\psi_i)\in \homo(\bT(\bQ_p),E)\;|\;
				\psi_j-\psi_i\in \sL^u_{i,j},\forall\;(i,j)\in R^+_u \right\}.
			\end{equation}
			where $\sL^u_{i,j}$ is the so-called simple $\sL$-invariant in \cite{2019DINGSimple}; it is a codimension-$1$ subspace of $\homo(\bQ_p^{\times},E)$ and depends on $u,i,j$.\;Moreover,\;$E\val_p\neq\sL^u_{i,j}$ for $u^{-1}(i)\in S_0$, while $E\val_p=\sL^u_{i,j}$ for $u^{-1}(i)\in I_0\backslash S_0$.\;See Proposition \ref{imageofkappaproof} in Appendix-$\mathrm{II}$ for its inductive definition.\;In what follows,\;we do not need the precise description.\;
			
			
			\begin{pro}\label{profordimofdefor}
				$\dim_E\ext^{1}_{u}(\Dpik,\Dpik)=1+\frac{n(n+1)}{2}$ and $\kappa_{u}$ is a surjection onto $\homo_u(\bT(\bQ_p),E)$.\;
				
			\end{pro}
			\begin{proof}
				We first compute $\dim_E\ext^{1}_{u}(\Dpik,\Dpik)$.\;We follow the route in the proof of
				\cite[Proposition 3.6]{2019DINGSimple}.\;Let $\EndO_{\cF_u}(\Dpik)$ be the saturated $(\varphi,\Gamma)$-submodule of $\EndO(\Dpik)$ that preserves the triangulation $\cF_u$.\;Then we have a natural exact sequence of 
				$(\varphi,\Gamma)$-modules over $\cR_E$:
				\[0\rightarrow \homo_{\cR_E}(R_{u,n},\Dpik
				)\rightarrow \EndO_{\cF_u}(\Dpik)\rightarrow \EndO_{\cF_u|_{\Dpik_1^{n-1}}}(\Dpik_1^{n-1})\rightarrow 0.\]
				Note that $\hH^0_{(\varphi,\Gamma)}(\homo_{\cR_E}(R_{u,n},\Dpik))=0$ by non-split assumption.\;By induction on $n$,\;we show that $\hH^0_{(\varphi,\Gamma)}(\EndO_{\cF_u}(\Dpik))\xrightarrow{\sim}\hH^0_{(\varphi,\Gamma)}(\EndO_{\cF_u|_{\Dpik_1^{n-1}}}(\Dpik_1^{n-1}))\cong E$.\;It suffices to show that $\hH^2_{(\varphi,\Gamma)}(\EndO_{\cF_u}(\Dpik))=0$.\;By induction on $n$,\;we only need to show that $\hH^2_{\cR_E}(\homo_{\cR_E}(R_{u,n},\Dpik
				))=0$.\;By Tate duality,\;we have isomorphism $\hH^2_{(\varphi,\Gamma)}(\homo_{\cR_E}(R_{u,n},\Dpik
				))\cong \hH^0_{(\varphi,\Gamma)}(\homo_{\cR_E}(\Dpik,
				R_{u,n}))$.\;If $i-1\not\in I_0$,\;the result is obvious.\;If $i-1\in I_0$,\;we still return to the case in \cite[Proposition 3.6]{2019DINGSimple}.\;See Proposition \ref{imageofkappaproof} for the proof  the second assertion.\;
			\end{proof}
			
			
			For $u\in \sW_n^{S_0}$,\;let $\ext^{1}_{g,u}(\Dpik,\Dpik)\subseteq\ext^{1}_{u}(\Dpik,\Dpik)$ be the preimage of $\homo_{\sm,u}(\bT(\bQ_p),E)$  under $\kappa_u$,\;where
			\[\homo_{\sm,u}(\bT(\bQ_p),E):=\homo_{\sm}(\bT(\bQ_p),E)\cap \homo_{u}(\bT(\bQ_p),E).\]

			\begin{pro}\label{studyforkerkappau}Let $\ext^1_g(\Dpik,\Dpik)$ be the subspace of de Rham deformations of $\Dpik$.\;
				\begin{itemize}
					\item[(1)] $\ext^1_g(\Dpik,\Dpik)=\ext^{1}_{g,1}(\Dpik,\Dpik)$ and 
					$\dim_E{\ext}^{1}_g(\Dpik,\Dpik)=1+\frac{n(n-1)}{2}+|I_0\backslash S_0|$.
					\item[(2)] For $u\in \sW_n^{S_0}$,\;$\ext^{1}_{g,u}(\Dpik,\Dpik)=\ext^{1}_{u}(\Dpik,\Dpik)\cap \ext^1_g(\Dpik,\Dpik)$.\;
				\end{itemize}
			\end{pro}
			\begin{proof}
				By \cite[Appendix,\;Proposition A.3]{Dingsocle},\;for any $i<j$,\;we deduce  that  $\dim_E\hH^1_g(\gal_{\bQ_p},W(R_{1,i}^{\vee}\otimes_{\cR_{E}}{R_{1,j}}))=0$,\;where $W(R_{1,i}^{\vee}\otimes_{\cR_{E}}{R_{1,j}})$ means the $E$-$B$ pair associated to $R_{1,i}^{\vee}\otimes_{\cR_{E}}{R_{1,j}}$.\;The d\'{e}vissage argument shows the inclusion $\ext^1_g(\Dpik,\Dpik)\subseteq \ext^{1}_{1}(\Dpik,\Dpik)$.\;Therefore,\;$\ext^1_g(\Dpik,\Dpik)=\ext^{1}_{g,1}(\Dpik,\Dpik)$ by definition.\;Since $\homo_{\sm,1}(\bT(\bQ_p),E)$ has codimension $2n-|I_0|-s=n-|I_0|-|S_0|$ in $\homo_{1}(\bT(\bQ_p),E)$,\;we get that $\dim_E{\ext}^{1}_g(\Dpik,\Dpik)=1+\frac{n(n-1)}{2}+|I_0\backslash S_0|$.\;It remains to prove $(2)$.\;Since $\cR_{E[\epsilon]/\epsilon^2}(1+\psi_i\epsilon)$ is de Rham if and only if $\psi_i\in \homo_{\sm}(\bQ^{\times}_p,E)$,\;we deduce that $\ext^{1}_{u}(\Dpik,\Dpik)\cap \ext^1_g(\Dpik,\Dpik)\subseteq \ext^{1}_{g,u}(\Dpik,\Dpik)$.\;On the other hand,\;for $(\psi_i)\in \homo_{\sm}(\bT(\bQ_p),E)$,\;we will show that \[\widetilde{\Dpik}=[\widetilde{R}_{u,1}-\widetilde{R}_{u,2}-\cdots-\widetilde{R}_{u,n}],\;\]
				with $\widetilde{R}_{u,i}\cong R_{u,i}\otimes_{\cR_{E}}\cR_{E[\epsilon]/\epsilon^2}(1+\psi_i\epsilon)$ and $\psi\in \homo_{\sm}(\bQ^{\times}_p,E)$ de Rham.\;It suffices to prove, by induction on $j$, that $\widetilde{\Dpik}_1^j:=\widetilde{R}_{1}-\widetilde{R}_{2}-\cdots-\widetilde{R}_{j}$ is de Rham.\;The case $j=1$ is obvious.\;Assume $\widetilde{\Dpik}_1^j$ is de Rham, and prove that $\widetilde{\Dpik}_1^{j+1}$ is de Rham.\;It suffices to show that each element in
				\[\hH^1\left(\gal_{\bQ_p},W(\widetilde{\Dpik}')\right),\;\widetilde{\Dpik}':=\widetilde{\Dpik}_1^j\otimes_{\cR_{E}}R_{u,j+1}^{\vee}\otimes_{\cR_{E}}\cR_{E[\epsilon]/\epsilon^2,L}(1-\psi_{j+1}\epsilon)\]
				is de Rham.\;Since all the Hodge-Tate weights of $\widetilde{\Dpik}'$ are positive,\;we deduce from \cite[Proposition A.3]{Dingsocle}  that $\hH_g^1(\gal_{\bQ_p},W(\widetilde{\Dpik}'))=\hH^1(\gal_{\bQ_p},W(\widetilde{\Dpik}'))$ (note that $\widetilde{\hH}^2(\gal_{\bQ_p},W(\widetilde{\Dpik}'))={\hH}^2(\gal_{\bQ_p},W(\widetilde{\Dpik}'))$).\;We thus obtain $\ext^{1}_{g,u}(\Dpik,\Dpik)\subseteq\ext^{1}_{u}(\Dpik,\Dpik)\cap \ext^1_g(\Dpik,\Dpik)$.\;
			\end{proof}
			The following Lemma is an analogue of \cite[Lemma 2.11]{ParaDing2024}.\;
			\begin{lem}\label{smhomoindepent}For any $u_1,u_2\in \sW_n^{S_0}$,\;put $\ext^{1}_{g,\{u_1,u_2\}}(\Dpik,\Dpik):=\ext^{1}_{g,u_1}(\Dpik,\Dpik)\cap \ext^{1}_{g,u_2}(\Dpik,\Dpik)$.\;Then the following diagram commutes:
				\begin{equation}
					\xymatrix{ \ext^{1}_{g,\{u_1,u_2\}}(\Dpik,\Dpik)\ar@{=}[d] \ar[r]_{\kappa_{u_1}\hspace{15pt}} & \homo_{\sm,\{u_1,u_2\}}(\bT(\bQ_p),E) \ar[d]^{\sim}_{u_2u_1^{-1}}\\
						\ext^{1}_{g,\{u_1,u_2\}}(\Dpik,\Dpik) \ar[r]_{\kappa_{u_2}\hspace{15pt}} & \homo_{\sm,\{u_1,u_2\}}(\bT(\bQ_p),E) },
				\end{equation}
				where   $\homo_{\sm,\{u_1,u_2\}}(\bT(\bQ_p),E):=\homo_{\sm,u_1}(\bT(\bQ_p),E)\cap \homo_{\sm,u_2}(\bT(\bQ_p),E)$.\;
			\end{lem}

			For any $u\in\sW_n^{S_0}$,\;let $\ext^{1}_{0,u}(\Dpik,\Dpik):=\ker\kappa_{u}\subseteq\ext^{1}_{g,u}(\Dpik,\Dpik)$ and $\ext^{1}_0(\Dpik,\Dpik):=\ker\kappa_{1}$.\;Then we have $\ext^{1}_0(\Dpik,\Dpik)\subseteq\ext^{1}_g(\Dpik,\Dpik)$.\;As a corollary,\;we obtain
			\begin{cor}\label{dimforext1g}
				For any $u\in\sW_n^{S_0}$,\;we have $\ext^{1}_{0,u}(\Dpik,\Dpik)=\ext^{1}_{u}(\Dpik,\Dpik)\cap \ext^{1}_0(\Dpik,\Dpik)$.\;Moreover,\;we have $\dim_E\ext^{1}_{0}(\Dpik,\Dpik)=1+\frac{n(n-1)}{2}+|I_0\backslash S_0|-n$.\;
			\end{cor}
			For any subspace $V\subseteq \ext^1(\Dpik,\Dpik)$, we put $\overline{V}:=V/V\cap \ext^{1}_{0}(\Dpik,\Dpik)$.\;

			\subsection{Reinterpretation and supplements for deformations}\label{reinterfor33}
			
			The discussion in this section is similar to \cite[Section 2.4]{BDcritical25} and \cite[Section 3.3]{HEparaforsemitable}.
			
			We briefly recall Fontaine's theory of almost de Rham representations.\;Let $B_{\pdr}^+:=B_{\dr}^+[\log t]$ and $B_{\pdr}:=B_{\pdr}^+[1/t]$.\;The $\gal_{\bQ_p}$-action on $B_{\dr}$ extends uniquely to an action on $B_{\pdr}$ with $g(\log t)=\log t+\log(\ccyc(g))$.\;Let $v_{\pdr}$ denote the unique $B_{\dr}$-linear derivation of $B_{\pdr}$ such that $v_{\pdr}(\log t)=-1$.\;Note that $v_{\pdr}$ commutes with $\gal_{\bQ_p}$ and that both preserve $B_{\pdr}^+$.\;Let $W$ (resp.,\;$W^+$) be a $B_{\dr}$-representation (resp.,\;$B^+_{\dr}$-representation) of $\gal_{\bQ_p}$ that is free of finite rank.\;Let $D_{\pdr}(W):=(B_{\pdr}\otimes_{B_{\dr}}W)^{\gal_{\bQ_p}}$, which is a finite-dimensional $L$-vector space of dimension at most $\dim_{B_{\dr}}W$.\;The $B_{\dr}$-representation $W$ is called \textit{almost de Rham} if $\dim_LD_{\pdr}(W)=\dim_{B_{\dr}}W$.\;$W^+$ is called \textit{almost de Rham} if $W^+[1/t]$ is almost de Rham.\;

			Keep the notation of Section \ref{Omegafil}.\;In this section,\;we use the language of $E$-$B$-pairs (see \cite{nakamura2009classification}).\;Recall that the $(\varphi,\Gamma)$-module $\cM_\Dpik:=\Dpik[1/t]$ over $\cR_{E}[1/t]$ admits a triangulation $\cF_u$ with graded pieces $R_{u,i}[1/t]=\cR_E(\unr(\phi_{u^{-1}(i)}))[1/t]$ for $1\leq i\leq n$.\;Let $\bW_{\Dpik}=W_{\dr}(\cM_\Dpik)$ (resp.,\;$\bW^+_{\Dpik}:=W_{\dr}^+(\Dpik)$) be the $\bB_\dr\otimes_{\bQ_p}E$-representation (resp.,\;$\bB_\dr^+\otimes_{\bQ_p}E$-representation) of $\gal_{\bQ_p}$ associated to $\cM_\Dpik$ (resp.\;$\Dpik$).\;Note that $D\cong D_{\pdr}(\bW_\Dpik)$.\;
			
			For $u\in \sW_n^{S_0}$,\;the filtration $\cF_u$ on $\cM_{\Dpik}$ induces a complete flag:
			\[\bF_{u}=\fil_{\bullet}^{\bF_{u}}\bW_\Dpik: \ 0 =\fil_0^{\bF_{u}}\bW_\Dpik \subsetneq \fil_1^{\bF_{u}}\bW_\Dpik \subsetneq \cdots \subsetneq \fil_{n}^{\bF_{u}}\bW_\Dpik=\bW_\Dpik\]
			on $\bW_\Dpik$ by $(\bB_\dr\otimes_{\bQ_p}E)$-subrepresentations of $\bW_{\Dpik}$.\;Applying the functor $D_{\pdr}(-)$,\;we obtain the corresponding complete flag $\fil_{\bullet}^{\bF_{u}}(D)$ on $D$:
			\begin{equation}
				\begin{aligned}
					\fil_{\bullet}^{\bF_{u}}(D)&: \ 0 =\fil_0^{\bF_{u}}(D) \subsetneq \fil_1^{\bF_{u}}(D) \subsetneq \cdots \subsetneq \fil_{n}^{\bF_{u}}(D)=D\\
				\end{aligned}
			\end{equation}
			Moreover,\;the $B_{\dr}^+$-lattice $\bW^+_\Dpik$ induces the Hodge filtration $\fil^{\bullet}_{H}(D)$:
			\begin{equation}
				\begin{aligned}
					\fil^{\bullet}_{H}(D)&: \ 0 \subsetneq \fil^{-\hpi_{n}}_{H}(D) \subsetneq \fil^{-\hpi_{n-1}}_{H}(D)  \subsetneq \cdots \subsetneq \fil^{{-\hpi_{1}}}_{H}(D)=D
				\end{aligned}
			\end{equation}
			with $\fil^{-\bh_{n+1-i}}_{H}(D):=\big(t^{-\bh_{n+1-i}}D_{\pdr}(\bW^+_\Dpik)\big)^{\gal_{\bQ_p}}$.\;

			For $u\in \sW_n^{S_0}$,\;let $\homo_{\fil}(D,D)$ (resp.,\;$\homo_{\bF_u}(D,D)$) be the subspace of endomorphisms of the $E$-vector space $D$ that preserve the filtration $\fil_{H}^{\bullet}(D)$ (resp.,\;$\fil_{\bF_{u}}^{\bullet}(D)$).\;Put
			\[\homo_{\fil,\bF_u}(D,D):=\homo_{\fil}(D,D)\cap \homo_{\bF_u}(D,D).\]
			By the same argument as in the discussion before \cite[(112)]{BDcritical25},\;we obtain a canonical $E$-linear morphism:
			\begin{equation}\label{maptofilDsigma}
				\begin{aligned}
					\nu:&\;\ext^{1}(\Dpik,\Dpik)\rightarrow \homo_{\fil}(D,D).
				\end{aligned}
			\end{equation}
			By \cite[Lemma 2.4.1]{BDcritical25},\;$\nu$ has kernel $\ext^{1}_g(\Dpik,\Dpik)$.\;Moreover,\;for $u\in\sW_n^{S_0}$,\;the map (\ref{maptofilDsigma}) induces
			\begin{equation}\label{sendtofilD}
				\begin{aligned}
					\nu_{u}:&\;\ext^{1}_u(\Dpik,\Dpik)/\ext^{1}_{g,u}(\Dpik,\Dpik)\rightarrow\homo_{\fil,\bF_u}(D,D).
				\end{aligned}
			\end{equation}
			For $1\leq i\leq s$,\;put $D^i:=D_{\pdr}(\bW_{E_i})$.\;Similar to \cite[Lemma 3.13]{HEparaforsemitable},\;we obtain
			\begin{lem}
				\begin{itemize}
					\item[(1)]
					For $1\leq i\leq s$,\;we have the following commutative diagram:
					\begin{equation}
						\xymatrix{
							0   \ar[r]  & \ext^{1,\flat}_g(E_i,E_i) \ar@{^(->}[d]\ar[r]  & \ext^{1,\flat}_1(E_i,E_i) \ar@{^(->}[d]\ar[r]  &  E \ar@{^(->}[d] \ar[r] & 0  \\
							0   \ar[r]  & \ext^{1}_g(E_i,E_i) \ar@{=}[d]\ar[r]  & \ext^{1}_1(E_i,E_i) \ar@{^(->}[d]\ar[r]  &  \homo_{\fil,\bF_1}(D^i,D^i) \ar@{^(->}[d] \ar[r] & 0  \\
							0    \ar[r] & \ext^1_g(E_i,E_i)  \ar[r] & \ext^{1}(E_i,E_i)\ar[r]   & \homo_{\fil}(D^i,D^i) \ar[r] & 0 }  
					\end{equation}	
					where $\dim_E\ext^1_g(E_i,E_i)=1+\frac{r_i(r_i-1)}{2}$ and $\homo_{\fil,\bF_1}(D^i,D^i)\cong E^{\oplus r_i}$.\;
					\item[(2)] For $1\leq i\leq s$,\;$\ext^{1}_g(E_i[1/t],E_i[1/t])\cong \homo_{\sm}(\bQ_p^{\times},E)$,\;and we have a natural isomorphism (obtained by inverting $t$) $\ext^{1,\flat}_g(E_i,E_i)\xrightarrow{\sim}\ext^{1}_g(E_i[1/t],E_i[1/t])$.\;
				\end{itemize}
			\end{lem}
			In the sequel,\;define 
			\begin{equation}\label{extgrmap}
				\begin{aligned}
				\ext^{1,\flat}_{\gr}({\Dpik},{\Dpik})&:=\prod_{i=1}^s\ext^{1,\flat}_g(E_i,E_i)\xrightarrow{\sim}\prod_{i=1}^s\ext^{1}_g(E_i[1/t],E_i[1/t]),\\
			\ext^{1}_{\gr}({\Dpik},{\Dpik})&:=\prod_{i=1}^s\ext^{1}_g(E_i,E_i).\;
				\end{aligned}
			\end{equation}
			\begin{lem}\label{lemforexactforextgropus}
				\begin{itemize}
					\item[(1)]	There is an isomorphism
					\begin{equation}\label{expianextphi}
						{\ext}^{1}_g(\Dpik,\Dpik)/{\ext}^{1}_0(\Dpik,\Dpik)\xrightarrow{\sim}\ext^{1,\flat}_{\gr}({\Dpik},{\Dpik}).
					\end{equation}
					\item[(2)] For $u\in \sW_n^{S_0}$,\;the following diagram commutes:
					\begin{equation}\label{commutativefortriandwhole}
						\xymatrix{
							0   \ar[r]  & \ext^{1,\flat}_{\gr}({\Dpik},{\Dpik}) \ar@{=}[d]\ar[r]  &\overline{\ext}^{1}_{u}(\Dpik,\Dpik) \ar@{^(->}[d]\ar[r]  &  \mathrm{Im}(\nu_{u}) \ar@{^(->}[d] \ar[r] & 0  \\
							0    \ar[r] & \ext^{1,\flat}_{\gr}({\Dpik},{\Dpik}) \ar[r] & \overline{\ext}^{1}(\Dpik,\Dpik)\ar[r]   &  \mathrm{Im}(\nu) \ar[r] & 0  }
					\end{equation}
					Moreover,\;$\dim_E\mathrm{Im}(\nu)=\frac{n(n+1)}{2}-|I_0\backslash S_0|$.\;
				\end{itemize}
			\end{lem}
			\begin{proof}Recall the dimensions of ${\ext}^{1,\flat}_0(\Dpik,\Dpik)$ and ${\ext}^{1}_{g,1}(\Dpik,\Dpik)={\ext}^{1}_g(\Dpik,\Dpik)$ from Proposition (\ref{studyforkerkappau}) and Corollary \ref{dimforext1g}.\;By Proposition 		\ref{profordimofdefor},\;these subspaces yield a natural morphism
				\[{\ext}^{1}_g(\Dpik,\Dpik)\rightarrow \ext^{1,\flat}_{\gr}({\Dpik},{\Dpik}).\]
				Since $\dim_E{\ext}^{1}_g(\Dpik,\Dpik)/{\ext}^{1,\flat}_0(\Dpik,\Dpik)=s$, which equals the $E$-dimension of the right-hand side of (\ref{expianextphi}), we obtain the isomorphism in $(1)$.\;Moreover,\;we deduce that $\dim_E\mathrm{Im}(\nu)=\frac{n(n+1)}{2}-|I_0\backslash S_0|$.\;$(2)$ follows from Corollary \ref{dimforext1g}.\;
			\end{proof}
			
			\begin{lem}
				We have $\cM_{\Dpik}\cong \oplus_{i=1}^sE_i[1/t]$.\;
			\end{lem}
			\begin{proof} Note that $\hH^1_{e}(\cR_E(\unr(\alpha))[1/t])=\hH^1_{g}(\cR_E(\unr(\alpha))[1/t])=0$ if $\alpha\neq 1$,\;and $\hH^1_{g}(\cR_E(\unr(\alpha))[1/t])\cong E$ if $\alpha=p$,\;by \cite[Lemmas 4.2 \& 4.3]{nakamura2009classification}.\;
			\end{proof}
			
			Similar to \cite[(121)]{BDcritical25} and the commutative diagram below \cite[(123)]{BDcritical25},\;we have the natural maps
			\[\ext^{1}(\Dpik,\Dpik)\rightarrow \ext^{1}(\cM_{\Dpik},\cM_{\Dpik})\rightarrow \homo_{E}(D,D).\]
			Similar to the argument in \cite[(122)-(123)]{BDcritical25},\;we have a commutative diagram of short exact sequences: 
			\begin{equation}\label{diagDtoML'}
				\begin{aligned}
					\xymatrix{
					0   \ar[r]  &  \ext^{1,\flat}_{\gr}({\Dpik},{\Dpik}) \ar[d]^{(\ref{extgrmap})}\ar[r]  & \overline{\ext}^{1}(\Dpik,\Dpik) \ar[d]\ar[r]  &  \mathrm{Im}(\nu) \ar@{^(->}[d] \ar[r] & 0  \\
					0    \ar[r] &  \ext^{1,\flat}_{\gr}({\Dpik},{\Dpik}) \ar[r] & \ext^{1}(\cM_{\Dpik},\cM_{\Dpik})\ar[r]   & \homo_{E}(D,D) \ar[r]  & 0 }
				\end{aligned}
			\end{equation}
			As in \cite[Proposition 2.4.4]{BDcritical25} and \cite[Proposition 3.16]{HEparaforsemitable},\;we have:
			\begin{pro}\label{splitingforext1gps}
				There is a splitting of the bottom exact sequence in (\ref{diagDtoML'}) which only depends on the choice of $\log_p(p)\in E$.\;Therefore,\;for $\ast\in \{\emptyset,u\}$,\;we have 
				\begin{equation}\label{splittingforextgps}
					\begin{aligned}
						&f_{\Dpik,\ast}:\overline{\ext}_{\ast}^{1}(\Dpik,\Dpik)\xrightarrow{\sim}\ext^{1,\flat}_{\gr}({\Dpik},{\Dpik})\oplus\mathrm{Im}(\nu_{\ast}),\\
					\end{aligned}
				\end{equation}
			\end{pro}			
In the sequel,\;we identify $\homo_{E}(D,D)$ (resp.,\;$\homo_{\fil}(D,D)$) with $\fg$ (resp.\;$\mathrm{Ad}_{g}(\fb)$) under the basis $e_{1},\cdots,e_{n}$ of $D$.\;Then for each $u\in \sW_n^{S_0}$,\;
			\begin{equation}
				\begin{aligned}
					&\homo_{\bF_u}(D,D)\cong\mathrm{Ad}_{u^{-1}}(\fb),\; &\homo_{\fil,\bF_u}(D,D)=\mathrm{Ad}_{u^{-1}}(\fb)\cap \mathrm{Ad}_{g}(\fb).\;
				\end{aligned}
			\end{equation}
			\begin{rmk}\label{explainHomasLiealg} The non-critical assumption implies that $g\bB=u^{-1}b_uw_{0}\bB$ for some $b_u\in \bB$.\;Note that $\overline{\fb}=\mathrm{Ad}_{w_0}(\fb)$ is the Borel algebra of lower triangular matrices.\;Then $\mathrm{Ad}_{u^{-1}}(\fb)\cap \mathrm{Ad}_{g}(\fb)\cong \mathrm{Ad}_{u^{-1}b_u}(\fb\cap \overline{\fb})=\mathrm{Ad}_{u^{-1}b_u}(\ft)$.\;
			\end{rmk}
			\begin{lem}\label{desofimage}$\mathrm{Im}(\nu_{u})\cap \mathrm{Ad}_{u^{-1}}(\fn_{\emptyset})=\mathrm{Ad}_{u^{-1}}(\fn_{\emptyset})$.\;
			\end{lem}
			\begin{proof}
				The Lemma follows from the definition of the Fontaine-Mazur simple $\sL$-invariants $\{\sL^u_{i,j}\}_{(i,j)\in R_u^+}$ and the Colmez--Greenberg--Stevens formula (see \cite[Proposition 3.5]{2019DINGSimple}),\;which implies that $\mathrm{coker}(\nu_{u})$ only comes from $\mathrm{Ad}_{u}(\ft)$.\;
			\end{proof}
			
			Let ${\ext}^{1}_{\mathrm{tri}}(\Dpik,\Dpik)$ be the image of the following map:
			\begin{equation}\label{natmorphismforGDpikfern}
				g_{\Dpik}:\bigoplus_{u\in \sW_n^{S_0}}{\ext}^{1}_{u}(\Dpik,\Dpik)\rightarrow {\ext}^1(\Dpik,\Dpik).\;
			\end{equation}
			Recall in \cite[Section,\;Lemma 2.1]{DensityHMS} that $\sum_{u\in \sW_{n}}\mathrm{Ad}_{u^{-1}}(\fb)\cap \mathrm{Ad}_{g}(\fb)=\mathrm{Ad}_{g}(\fb)$.\;Put
			\begin{equation}
				\mathrm{Ad}_{g}(\fb)_{/S_0}:=\sum_{u\in \sW_n^{S_0}}\mathrm{Ad}_{u^{-1}}(\fb)\cap \mathrm{Ad}_{g}(\fb)\subseteq \mathrm{Ad}_{g}(\fb).\;
			\end{equation}
			
			\begin{thm}[Infinite-fern for the semistable case]\label{infinitePoten}
				\begin{equation}
					\begin{aligned}
						&{\ext}_{\mathrm{tri}}^{1}(\Dpik,\Dpik)/\ext^{1}_{g}({\Dpik},{\Dpik})\cong \mathrm{Ad}_{g}(\fb)_{/S_0}\cap \mathrm{Im}(\nu),\;
					\end{aligned}
				\end{equation}
				where the latter intersection is taken in $\mathrm{Ad}_{g}(\fb)$.\;Moreover,\;$\dim_E\mathrm{Im}(\nu)=\dim_E\mathrm{Ad}_{g}(\fb)_{/S_0}-|I_0\backslash S_0|$.\;
			\end{thm}
			
			For $u\in \sW_n^{S_0}$,\;the $\bP_{S_0(u)}$-parabolic filtration $\cF_{S_0(u)}$ on $\cM_{\Dpik}$ induces a partial flag $\fil_{\bullet}^{\bF_{S_0(u)}}(D)$ on $D$:
			\begin{equation}
				\begin{aligned}
					\fil_{\bullet}^{\bF_{S_0(u)}}(D)&: \ 0 =\fil_0^{\bF_{S_0(u)}}(D) \subsetneq \fil_1^{\bF_{S_0(u)}}(D) \subsetneq \cdots \subsetneq \fil_{s}^{\bF_{S_0(u)}}(D)=D\\
				\end{aligned}
			\end{equation}
			Let $\homo_{\bF_{S_0(u)}}(D,D)$ be the subspace of endomorphisms of the $E$-vector space $D$ that preserve the filtration $\fil^{\bF_{S_0(u)}}_{\bullet}(D)$.\;Put $\homo_{\fil,\bF_{S_0(u)}}(D,D):=\homo_{\fil}(D,D)\cap \homo_{\bF_{S_0(u)}}(D,D)$.\;By (\ref{maptofilDsigma}),\;we obtain a natural map
			\begin{equation}\label{sendtofilDSzero}
				\begin{aligned}
					\nu_{S_0(u)}:&\;\ext^{1}_{S_0(u)}(\Dpik,\Dpik)/\ext^{1}_{S_0(u)}(\Dpik,\Dpik)\cap \ext^{1}_{g}({\Dpik},{\Dpik})\rightarrow\homo_{\fil,\bF_{S_0(u)}}(D,D)
				\end{aligned}
			\end{equation}
			In this case,\;$\homo_{\fil,\bF_{S_0(u)}}(D,D)\cong \mathrm{Ad}_{u^{-1}}(\fp_{S_0(u)})\cap \mathrm{Ad}_{g}(\fb)$.\;We have 
			\begin{pro}\label{infinitefernforsemi} 
				$\sum_{u\in\sW_n^{S_0}}\mathrm{Ad}_{u^{-1}}(\fp_{S_0(u)})\cap \mathrm{Ad}_{g}(\fb)=\mathrm{Ad}_{g}(\fb)$.\;
			\end{pro}
			\begin{proof}For $1\leq i\leq s$ and $w\in \sW_s$,\;put $t_0^w=0$ and $t_i^w=\sum_{j=1}^il_{w^{-1}(j)}$ (so that $t^w_s=n$).\;We have an inclusion $j_{S_0}:\sW_s\rightarrow \sW_n^{S_0}$ sending $w\in \sW_s$ to the unique element $j_{S_0}(w)$ in $\sW_n^{S_0}$ that satisfies $j_{S_0}(w)(t_r+l)=t^w_r+l$ for $1\leq l\leq l_r$ and $1\leq r\leq s$.\;In this case,\;$\bL_{S_0(j_{S_0}(w))}:=\prod_{i=1}^s\GLN_{l_{w^{-1}(i)}}$.\;The subspace $\sum_{u\in j_{S_0}(\sW_s)}\mathrm{Ad}_{u}(\fp_{S_0(u)})$ of the left-hand side is already equal to $\mathrm{Ad}_{g}(\fb)$ by \cite[Theorem 6.1]{HEparaforsemitable}.\;
			\end{proof}
			\begin{cor}\label{surforgD}
				The natural map $\overline{g}^+_{\Dpik}:\bigoplus_{u\in\sW_n^{S_0}}\overline{\ext}^{1}_{S_0(u)}(\Dpik,\Dpik)\rightarrow \overline{\ext}^{1}(\Dpik,\Dpik)$  is surjective.\;
			\end{cor}
			\begin{proof}It suffices to show the following assertion for Lie-algebras:
				\[\sum_{u\in \sW_n^{S_0}}\mathrm{Im}(\nu_{u})\cap \mathrm{Ad}_{g}(\fb)=\mathrm{Im}(\nu).\]
				When $\Dpik$ is generic,\;we have $\mathrm{Im}(\nu_{u})=\mathrm{Ad}_{u^{-1}}(\fp_{S_0(u)})$ and $\mathrm{Im}(\nu)=\mathrm{Ad}_{g}(\fb)$,\;this is Proposition \ref{infinitefernforsemi}.\;It remains to treat the general case.\;By Lemma \ref{desofimage},\;we have 
				\[\mathrm{Im}(\nu_{u})\cap \mathrm{Ad}_{g}(\fb)=(\mathrm{Im}(\nu_{u})\cap \mathrm{Ad}_{u^{-1}}(\ft)+\mathrm{Ad}_{u^{-1}}(\fn_{\emptyset}))\cap \mathrm{Ad}_{g}(\fb)\supseteq \mathrm{Im}(\nu_{u})\cap \mathrm{Ad}_{u^{-1}}(\ft)\cap \mathrm{Ad}_{g}(\fb).\]
				On the other hand,\;Proposition \ref{profordimofdefor} gives
				\[\sum_{u\in\sW_n^{S_0}}\mathrm{Im}(\nu_{u})\cap \mathrm{Ad}_{u^{-1}}(\ft)\cap \mathrm{Ad}_{g}(\fb)=\sum_{u\in\sW_n^{S_0}} \mathrm{Ad}_{u^{-1}}(\ft)\cap \mathrm{Ad}_{g}(\fb)\]
				by choosing $u$ carefully.\;This show that $\sum_{u\in\sW_n^{S_0}}\mathrm{Im}(\nu_{u})\cap \mathrm{Ad}_{g}(\fb)$ contains $\mathrm{Im}(\nu)\cap \mathrm{Ad}_{g}(\fb)=\mathrm{Im}(\nu)$.\;We complete the proof.\;
			\end{proof}	
			
			Moreover,\;for $i\in \Delta$,\;by replacing $\cF_{S_0(u)}$ with the $\bL_i$-parabolic filtration $\cF_{u,[i]}$,\;we get a partial flag $\bF_{u,[i]}$ on $D$.\;In this case,\;$\homo_{\bF_{u,[i]}}(D,D)\cong \mathrm{Ad}_{u}(\fl_{i})$ and  $\homo_{\fil,\bF_{u,[i]}}(D,D)\cong \mathrm{Ad}_{u}(\fl_{i})\cap \mathrm{Ad}_{g}(\fb)$.\;Put
			\begin{equation}\label{dfnforDeltaS0}
				\mathrm{Ad}_{g}(\fb)^{\Delta}_{/S_0}:=\sum_{u\in \sW_n^{S_0},i\in S_0(u)}\mathrm{Ad}_{u^{-1}}(\fl_i)\cap \mathrm{Ad}_{g}(\fb)\subseteq \mathrm{Ad}_{g}(\fb).\;
			\end{equation}

			\section{Higher extension groups and Breuil-Schraen $\sL$-invariants}
			
			In \cite{wholeLINV},\;we see that the higher extension groups between locally analytic generalized Steinberg representations play a crucial role in identifying the automorphic counterparts of $p$-adic Hodge parameters for Steinberg $(\varphi,\Gamma)$-modules.\;In this section,\;we first use them to study the $\ext^{\bullet}$ between parabolic inductions of locally analytic generalized Steinberg representations,\;and them capture the $p$-adic Hodge parameters of Steinberg subquotients of $\Dpik$ through Breuil-Schraen $\sL$-invariants.\;
			
			Let $D(G,E)$ be the locally $\bQ_p$-analytic distribution algebra.\;For admissible locally $\bQ_p$-analytic representations $V_1,V_2$,\;the continuous strong duals $V_1^{\vee}$ and $V_2^{\vee}$ are coadmissible modules over $D(G,E)$.\;Let $\ext^1_{G}(V_1,V_2):=\ext^1_{D(G,E)}(V_2^{\vee},V_1^{\vee})$,\;where the latter extension group is defined in the abelian category of abstract $D(G,E)$-modules.\;In this case,\;$\ext^1_{G}(V_1,V_2)$ coincides with the extension group in the category of admissible locally $\bQ_p$-analytic representations.\;Suppose that $V_1,V_2$ have the same central character $\chi$.\;Let $\ext^1_{G,Z}(V_1,V_2)$ be the subspace of $\ext^1_{G}(V_1,V_2)$ consisting of locally $\bQ_p$-analytic extensions with central character $\chi$.\;
			
			For any $I\subseteq \Delta$,\;write $\bL_I=\GLN_{h_1}\times\cdots\times \GLN_{h_r}$.\;For $1\leq j\leq r$,\;write $\bL_{I,j}:=\GLN_{h_j}$.\;Let $\Delta_{I,j}\subseteq \Delta$ be the simple roots of $\bL_{I,j}$.\;Then $I=\sqcup_{1\leq j\leq r}\Delta_{I,j}$.\;Let 
			\[\delta_{\op_I}=\prod_{i=1}^r|\mathrm{det}_{\GLN_{h_i}}|^{-\sum_{j=i+1}^rh_j+\sum_{j=1}^{i-1}h_j}=\prod_{i=1}^r|\mathrm{det}_{\GLN_{h_i}}|^{-n-h_i+2\sum_{j=1}^{i}h_j}\]
			be the  modulus character of $\op_I$.\;Let $\eta_I=\delta_{\op_I}^{1/2}$ and $\eta:=\eta_{\emptyset}$.\;
			
			
			Let $d$ be an integer.\;Recall that $\bT_d$ (resp.,\;$\bB_{d}$,\;resp.,\;$\bP_{d,I}$) is the standard diagonal torus (resp.,\;Borel subgroup of upper triangular matrices,\;resp.,\;parabolic subgroup associate to $I\subseteq \Delta_d$) of $\GLN_d$ and $\overline{\bB}_{d}$ (resp.,\;$\op_{d,I}$) is the Borel subgroup opposite to $\bB_{d}$ (resp.,\;$\bP_{d,I}$).\;For $I\subseteq \Delta_d$ and $?\in\{\infty,\ana\}$,\;set (\cite[Section 2.1.2]{2019DINGSimple})
			\[v^{?}_{I,\Delta_{d}}:=i^{?}_{I,\Delta_{d}}/\sum_{I\subsetneq J}i^{?}_{J,\Delta_{d}},\;i^{?}_{I,\Delta_{d}}:=\left(\ind^{\GLN_d(\bQ_p)}_{\op_{d,I}(\bQ_p)}1_{\op_{d,I}(\bQ_p)}\right)^{?}.\]
			$\{v^{\ana}_{I,\Delta_{d}}\}_{I\subseteq \Delta_d}$ (resp.,\;$\{v^{\infty}_{I,\Delta_{d}}\}_{I\subseteq \Delta_d}$) are the so-called locally analytic (resp.,\;smooth) generalized Steinberg representations of $\GLN_d(\bQ_p)$.\;We write $\st^{\ana}_{d}:=v^{\ana}_{\emptyset,\Delta_{d}}$ (resp.,\;$\st^{\infty}_{d}:=v^{\infty}_{\emptyset,\Delta_{d}}$), which we call the locally analytic (resp.,\;smooth) Steinberg representation.\;

			\subsection{On $\ext^{\bullet}$ between parabolic inductions of locally analytic generalized Steinberg representations}\label{TYPE1higherext}

			Fix $J\subseteq \Delta$.\;Write $\bL_J\cong \bL_{J,1}\times \bL_{J,2}\times\cdots\times \bL_{J,d}$ and $J=\sqcup_{1\leq j\leq d}\Delta_{J,j}$ (so that $d=|\Delta\backslash J|+1$).\;Fix $\ba:=(a_1,a_2,\cdots,a_d)\in \BZ^{d}$.\;For $I\subseteq J$,\;put
				\begin{equation}
				i^{\ana}_{I, J}(\ba):=\widehat{\boxtimes}_{j=1}^d|\mathrm{det}_{\bL_{J,j}}|^{a_j}i^{\ana}_{\Delta_{J,j}\cap I,\Delta_{J,j}},\;v^{\ana}_{I, J}(\ba):=\widehat{\boxtimes}_{j=1}^d|\mathrm{det}_{\bL_{J,j}}|^{a_j}v^{\ana}_{\Delta_{J,j}\cap I,\Delta_{J,j}}
			\end{equation}
			Then the locally analytic representation (parabolic inductions of locally analytic generalized Steinberg representations)
			\begin{equation}
				v^{\ana}_{I, J,\Delta}(\ba):=\left(\ind^G_{\op_J}v^{\ana}_{I, J}(\ba)\right)^{\ana}
			\end{equation}
			admits the so-called Bruhat-Tits resolution $\bC_{I, J}(\ba)$ (see \cite[(463)]{wholeLINV}):
			\begin{equation}
				\begin{aligned}
					\left(\ind^G_{\op_J(\bQ_p)}i^{\ana}_{J, J}(\ba)\right)^{\ana}\rightarrow\cdots\rightarrow\bigoplus_{I\subseteq{K}\subseteq  J,|K|=l}&\left(\ind^G_{\op_K(\bQ_p)}i^{\ana}_{K, J}(\ba)\right)^{\ana}\rightarrow\\ &\cdots\rightarrow\left(\ind^G_{\op_I(\bQ_p)}i^{\ana}_{I, J}(\ba)\right)^{\ana}\rightarrow v^{\ana}_{I, J,\Delta}(\ba),
				\end{aligned}
			\end{equation}
			where the middle term lies in $-l$ degree.\;Thus,\;for any $I_2\subseteq I_1\subseteq  J$,\;we have
			\begin{equation}
				M^{\bullet}_{I_1,I_2}(\ba):=\ext^\bullet_{G,Z}(\bC_{I_1, J}(\ba),\bC_{I_2, J}(\ba))=\ext^{\bullet-|I_1\backslash I_2| }_{G,Z}(v^{\ana}_{I_1, J,\Delta}(\ba),v^{\ana}_{I_2, J,\Delta}(\ba)).\;
			\end{equation}
			By the same argument as in \cite[Lemma 3.3.2 $\&$ Proposition 3.3.3]{wholeLINV},\;$M^{\bullet}_{I_1,I_2}(\ba)$ can be identified with (see \cite[(593)]{wholeLINV} and keep its notation)
			\begin{equation}
				M^{\bullet}_{I_1,I_2}:=\ext^\bullet_{G,Z}(\bC_{I_1,J},\bC_{I_2, J}).\;
			\end{equation}

			Applying \cite[Theorem 4.5.2]{wholeLINV} to the case $(I_0,I_1)=(I_1,J)$ and $(I_2,I_3)=(I_2,J)$,\;we deduce (keep the notation in  \cite[Theorem 4.5.2]{wholeLINV})
			\begin{thm}
				Set $J'=I_2\cup (J\backslash I_1)\subseteq J$.\;We have the following results.\;
				\begin{itemize}
					\item[(i)] If $\ext^{h}_{G,Z}\big(v^{\ana}_{I_1, J}(\ba),v^{\ana}_{I_2, J}(\ba)\big)\neq 0$,\;then $|I_1\backslash I_2|\leq h\leq n^2-n+|J|-| I_1\backslash I_2|$.\;
					\item[(ii)] The bottom extension group $\BE_{I_1,I_2}(\ba):=\ext_{G,Z}
					^{|I_1\backslash I_2|}\big(v^{\ana}_{I_1, J,\Delta}(\ba),v^{\ana}_{I_2, J,\Delta}(\ba)\big)$ admits a canonical decreasing filtration
					\begin{equation}
						\begin{aligned}
							0=\fil^{-|J'|+1}(\BE_{I_1,I_2}(\ba))\subseteq \fil^{-|J'|}(\BE_{I_1,I_2}(\ba))\subseteq
							\cdots\subseteq \fil^{-|J|}(\BE_{I_1,I_2}(\ba))=\BE_{I_1,I_2}(\ba)
						\end{aligned}
					\end{equation}
					with graded piece $\fil^{-l}(\BE_{I_1,I_2}(\ba))/\fil^{-l+1}(\BE_{I_1,I_2}(\ba))\cong E_{2,J',J}^{-l,l+|J|-2|J'|}$ admitting a basis induced from $\Psi_{J',J,l}$ for $|J'|\leq l\leq |J|$.\;
				\end{itemize}
			\end{thm}
			For $I_3\subseteq I_2\subseteq I_1$,\;we have a natural cup product map:
			\begin{equation}
				\kappa_{I_1,I_3}^{I_2}:	\BE_{I_1,I_2}(\ba)\otimes_E \BE_{I_2,I_3}(\ba)\xrightarrow{\cup} \BE_{I_1,I_3}(\ba).
			\end{equation}
			Using \cite[Theorem 5.3.6]{wholeLINV},\;we get that:
			\begin{thm}\label{decompforextgps}
			\begin{itemize}
				\item[(1)]	The cup product map $\kappa_{I_1,I_3}^{I_2}$ is injective.\;If $I_1\backslash I_2$ and $I_2\backslash I_3$ are not connected (i.e.,\;for any $i\in I_1\backslash I_2$ and $i'\in I_2\backslash I_3$,\;we have $|i-i'|\geq 2$),\;then $\kappa_{I_1,I_3}^{I_2}$ is an isomorphism.\;
				\item[(2)]For $1\leq j\leq d$ and $I_2'\subseteq I_1'\subseteq \Delta_{J,j}$,\;we put
				\[\BE_{I_1',I_2'}(\ba)_j:=\ext_{\bL_{J,j}(\bQ_p),Z}
				^{|I_2'\backslash I_1'|}(v^{\ana}_{I'_1,\Delta_{J,j}},v^{\ana}_{I_2',\Delta_{J,j}}).\]
			\end{itemize}
			There is an isomorphism
				\[\BE_{I_1,I_2}(\ba)=\bigotimes_{j=1}^d \BE_{I_1\cap \Delta_{J,j},I_2\cap \Delta_{J,j}}(\ba)_j.\]
			\end{thm}
			\begin{proof}
				The first statement is \cite[Theorem 5.3.6]{wholeLINV}.\;For the second statement,\;choose a sequence
				\[I_1=T_0\subseteq T_1\subseteq \cdots \subseteq T_d=I_2\]
				such that $T_l\backslash T_{l-1}=(I_2\backslash I_1)\cap \Delta_{J,l}$ for $1\leq l\leq d$.\;Then for any $1\leq l\leq d$,\;the map
				\[\kappa_{T_0,T_l}^{T_{l-1}}:\BE_{T_0,T_{l-1}}(\ba)\otimes_E \BE_{T_{l-1},T_l}(\ba)\xrightarrow{\cup} \BE_{T_0,T_l}(\ba)\]
				is an isomorphism.\;By \cite[Theorem 1.3.5 (i)]{wholeLINV},\;we have $\BE_{T_{l-1},T_l}\cong \BE_{I_1\cap \Delta_{J,j},I_2\cap \Delta_{J,j}}(\ba)_j$.\;By applying this argument step by step,\;we get the desired tensor product decomposition.\;
			\end{proof}
			
			This theorem reduces $\BE_{I_1,I_2}(\ba)$ to the higher extension groups of its Steinberg Levi  blocks.\;In particular,\;we can describe the
			cup product structure of $\BE_{I_1,I_2}(\ba)$.\;We collect (and review) them in the following proposition,\;see \cite[Lemmas  5.4.1 $\&$ 5.4.4]{wholeLINV}.\;
			
			\begin{pro}
				\begin{itemize}
					\item[(1)]$\BE_{I_1,I_2}(\ba)\cong \BE_{I_1\backslash I_2,\emptyset}(\ba)$.\;
					\item[(2)] For each $I_2\subseteq I_1\subseteq J$ satisfying $\#I_1\backslash I_2=1$,\;there is a canonical isomorphism
					\[\BE_{I_1,I_2}(\ba)\cong \homo(\bQ_p^{\times},E).\;\]
					\item[(3)] For each $I_2\subseteq I_1\subseteq J$ satisfying $\#I_1\backslash I_2>1$,\;we set
					\[\BE^{<}_{I_1,I_2}(\ba):=\sum_{I_3\subsetneq I_2\subsetneq I_1}\im(\kappa_{I_1,I_3}^{I_2})\subseteq \BE_{I_1,I_2}(\ba).\]
					Then $\dim_E\BE_{I_1,I_2}(\ba)/\BE^{<}_{I_1,I_2}(\ba)=t$,\;where $t:=|\{1\leq l\leq d:(I_1\backslash I_2)\cap \Delta_{J,l} \text{\;is an interval}\}|$.\;
				\end{itemize}
			\end{pro}
			
			Let $\Phi_J^+$ be the set of positive roots of $\bL_J$ with respect to $(\bL_J\cap \bB,\bL_J\cap \bT)$.\;For $\alpha=x_i-x_j\in \Phi^+_J$,\;put $I_{\alpha}=\{i,\cdots,j-1\}\subseteq J$.\;If $\alpha\in \Phi^+_J\backslash J$ (i.e.,\;$|I_{\alpha}|\geq 2$),\;we fix a choice of $x_{\alpha}\in \BE_{I_{\alpha},\emptyset}(\ba)$ such that its image in $\BE_{I_{\alpha},\emptyset}(\ba)/\BE^{<}_{I_{\alpha},\emptyset}(\ba)$ is non-zero,\;and we set $\overline{X}_{\alpha}:=\{x_{\alpha}\}$.\;If $\alpha\in J$ (i.e.,\;$|I_{\alpha}|=1$),\;we write $x^{\infty}_{\alpha}$ (resp.,\;$x_{\alpha}$) for the elements in $\BE_{I_{\alpha},\emptyset}(\ba)\cong \homo(\bQ_p^{\times},E)$ corresponding to $\val_p$ (resp.,\;$\log_p$).\;In this case,\;we put
			$\overline{X}_{\alpha}:=\{x^{\infty}_{\alpha},x_{\alpha}\}$.\;Moreover,\;for $\alpha=x_i-x_j$,\;we set $x^{\infty}_{\alpha}:=x^{\infty}_{\alpha_i}\cup\cdots\cup x^{\infty}_{\alpha_{j-1}}\in \BE_{I_{\alpha},\emptyset}(\ba)$.\;
			
			For $I\subseteq J$,\;we write $\cS_I$ for the set of subsets $S\subseteq \Phi^+_J$ such that $\sum_{\alpha\in S}\alpha=\alpha_I$.\;It is clear that there exists a bijection between $\cS_I$ and the set of partitions of $I$ into non-empty subintervals (i.e.,\;sending $S$ to the partition $I=\sqcup_{\alpha\in S}I_{\alpha}$).\;For $S\subseteq \cS_I$,\;we set $\overline{X}_{S}:=\prod_{\alpha\in S}\overline{X}_{\alpha}$.\;Then we have (similar to \cite[Proposition 5.4.3]{wholeLINV}):
			\begin{pro}For $I_2\subseteq I_1\subseteq  J$,\;$\BE_{I_1,I_2}(\ba)$ admits a basis of the form $X_{I_1,I_2}:=\sqcup_{S\in \cS_{I_1\backslash I_2}}\overline{X}_{S}$.\;
			\end{pro}

			For each $I_2\subseteq I_1\subseteq J$,\;if $|I_1\backslash I_2|=1$,\;we set ${\BE}^{\infty}_{I_1,I_2}(\ba)=E\val_p\subseteq \homo(\bQ_p^{\times},E)$.\;For general $I_2\subseteq I_1$,\;we choose a sequence $I_2=J_0\subseteq {J_1}\subseteq \cdots \subseteq {J_t}=I_1$ such that $t:=|I_1\backslash I_2|$ and $|{J_j}\backslash {J_{j-1}}|=1$ for all $1\leq j\leq t$.\;We define ${\BE}^{\infty}_{I_1,I_2}(\ba)$ as the image of the composition:
			\[{\BE}^{\infty}_{{J_t},{J_{t-1}}}(\ba)\otimes_E\cdots\otimes_E{\BE}^{\infty}_{{J_1},{J_{0}}}(\ba)\hookrightarrow {\BE}_{{J_t},{J_{t-1}}}(\ba)\otimes_E\cdots\otimes_E{\BE}_{{J_1},{J_{0}}}(\ba)\xrightarrow{\cup }{\BE}_{I_1,I_2}(\ba),\]
			which gives a canonical $E$-line in ${\BE}_{I_1,I_2}(\ba)$ and is independent of the choice of the sequence $J_0\subseteq {J_1}\subseteq \cdots \subseteq {J_t}$.\;We define $\widetilde{\BE}_{I_1,I_2}(\ba)$ as the image of the injection
			\[{\BE}^{\infty}_{J,I_1}(\ba)\otimes_E{\BE}_{I_1,I_2}(\ba)\otimes_E{\BE}^{\infty}_{I_2,\emptyset}(\ba)\xrightarrow{\cup}{\BE}_{J,\emptyset}(\ba):=\BE_{J}(\ba).\]
			By choosing a non-zero element in ${\BE}^{\infty}_{J,I_1}(\ba)$ (resp.,\;${\BE}^{\infty}_{I_2,\emptyset}(\ba)$),\;we obtain a non-canonical isomorphism 
			\[\iota_{I_1,I_2}(\ba):{\BE}_{I_1,I_2}(\ba)\xrightarrow{\sim }\widetilde{\BE}_{I_1,I_2}(\ba).\]
	
		Put $\BE_J(\ba):={\BE}_{J,\emptyset}(\ba)$ for simplicity.\;Similar to \cite[Definition 8.1.1]{wholeLINV},\;we define Breuil-Schraen $\sL$-invariants in $\BE_J(\ba)$.\;
			
			\begin{dfn}\label{BDlinvgeneral}
				A Breuil-Schraen $\sL$-invariant $\BW(\ba)$ in $\BE_J(\ba)$ is a subspace of $\BE_J(\ba)$ such that:
				\begin{itemize}
					\item[(1)] $\BW(\ba)\subseteq \BE_J(\ba)$ is co-dimensional $1$ in   $\BE_J(\ba)$.\;
					\item[(2)] For each $I_2\subseteq I_1\subseteq J$,\;$\BW(\ba)\cap \widetilde{\BE}_{I_1,I_2}(\ba)\subsetneq \widetilde{\BE}_{I_1,I_2}(\ba)$.\;Let $\BW_{I_1,I_2}(\ba):=\iota^{-1}_{I_1,I_2}(\ba)(\BW(\ba)\cap \widetilde{\BE}_{I_1,I_2}(\ba))$,\;then this condition implies $\dim_E\BE_{I_1,I_2}/\BW_{I_1,I_2}(\ba)=1$.\;
					\item[(3)] For each $I_3\subseteq I_2\subseteq I_1\subseteq J$,\;the composition $\BE_{I_1,I_2}\otimes_E \BE_{I_2,I_3} \xrightarrow{\cup} \BE_{I_1,I_3}\twoheadrightarrow \BE_{I_1,I_3}/\BW_{I_1,I_3}(\ba)$ factors through an isomorphism between $E$-lines
					\[(\BE_{I_1,I_2}/\BW_{I_1,I_2}(\ba))\otimes_E(\BE_{I_2,I_3}/\BW_{I_2,I_3}(\ba))\xrightarrow{\sim}(\BE_{I_1,I_3}/\BW_{I_1,I_3}(\ba)).\]
				\end{itemize}
				
			\end{dfn}
			
			Let $\mathcal{BS}_J(\ba)$ be the moduli space of Breuil-Schraen $\sL$-invariants inside $\BE_J(\ba)$,\;which is a closed subvariety of a Zariski-open subvariety of the projective space $\BP(\BE_J(\ba))$.\;Similar to the argument in \cite[Theorem 8.1.4]{wholeLINV},\;for $\BW(\ba)\subseteq \BE_J(\ba)$ and $\alpha\in \Psi_{J}^+$,\;there exists a unique $\sL_{\alpha}(\ba)\in E$ such that $x_{\alpha}-\sL_{\alpha}(\ba)x^{\infty}_{\alpha}\in \BW(\ba)_{I_{\alpha},\emptyset}$.\;This gives an element $(\sL_{\alpha}(\ba))_{\alpha\in \Psi_{J}^+}\in \fl_{J}\cap \fn_{\emptyset}$.\;Indeed,\;the map
			\begin{equation}
				\begin{aligned}
					\mathcal{BS}_J(\ba)\xrightarrow{\sim} \fl_{J}\cap \fn_{\emptyset},\;\BW(\ba)\mapsto  (\sL_{\alpha}(\ba))_{\alpha\in \Psi_{J}^+}.\;
				\end{aligned}
			\end{equation}
			is an isomorphism.\;The symbol $(\sL_{\alpha}(\ba))_{\alpha\in \Psi_{J}^+}$  indicatea that there exists a way to match Breuil-Schraen $\sL$-invariants inside $\BE_J(\ba)$ with the classical Fontaine-Mazur $\sL$-invariants defined for Steinberg $(\varphi,\Gamma)$-modules.\;

			\subsection{Capture $p$-adic Hodge parameters through Breuil-Schraen $\sL$-invariants}\label{BSINVtype1}
			
			Keep the notation in Section \ref{Omegafil}.\;Recall $\underline{\phi}=(\phi_1,\cdots,\phi_n)=(\alpha_1,\alpha_1p,\cdots,\alpha_1p^{l_1-1},\cdots,\alpha_s,\alpha_sp,\cdots,\alpha_sp^{l_s-1})$ in Section \ref{Omegafil}.\;For $u\in \sW_n^{S_0}$,\;put
			\[\unr(\underline{\phi}^u):=\boxtimes_{j=1}^n\unr(\phi_{u^{-1}(j)}).\]
			For any $u\in \sW_n^{S_0}$,\;there is  a $\bP_{S_0(u)}$-parabolic filtration $\cF_{S_0(u)}$ of $\Dpik$ with graded Steinberg pieces $\{E_{u,i}\}_{1\leq i\leq f_u}$.\;Note that $\bL_{S_0(u)}=\prod_{1\leq i\leq f_u}\bL_{S_0(u),i}:=\prod_{1\leq i\leq f_u}\GLN_{r_{u,i}}$.\;For $1\leq i\leq f_u$,\;we see that
			\[\unr(\underline{\phi}^u)\eta|_{\bL_{S_0(u),i}(\bQ_p)}=\unr(a_i(u))\otimes \mathrm{1}_{\bT(\bQ_p)\cap \bL_{S_0(u),i}(\bQ_p)}.\]
			We apply the results in  Section \ref{TYPE1higherext} to $\ba=\ba(u):=(a_1(u),\cdots,a_{f_u}(u))$ and $J=S_0(u)$.\;For any  $I_2\subseteq I_1\subseteq \Delta_{S_0(u)}$,\;we rewrite $\BE_{I_1,I_2}(\ba(u))$ (resp.,\;$\BE(\ba(u))$)  with $\BE_{I_1,I_2}(u)$ (resp.,\;$\BE(u)$) for simplicity.\;

			\begin{dfn}A total Breuil-Schraen $\sL$-invariant of monodromy type $S_0$ is a collection $\{\BW(u)\subseteq \BE(u)\}_{u\in \sW_n^{S_0}}$ with $\BW(u)$ a Breuil-Schraen $\sL$-invariant in $\BE(u)$ (in the sense of Definition \ref{BDlinvgeneral}).\;
			\end{dfn}
				
				Let $\mathcal{BS}(u)$ (resp.,\;$\mathcal{BS}_{S_0}$) be the moduli space of Breuil-Schraen $\sL$-invariants inside $\BE(u)$ (resp.,\;total Breuil-Schraen $\sL$-invariant of monodromy type $S_0$).\;Then we have isomorphisms
				\begin{equation}
					\begin{aligned}
						\mathcal{BS}(u)\xrightarrow{\sim} \fl_{S_0(u)}\cap \fn_{\emptyset},\;\BW(u)\mapsto  (\sL_{\alpha}(u))_{\alpha\in \Psi_{S_0(u)}^+},
					\end{aligned}
				\end{equation}
				and 
				\begin{equation}
					\begin{aligned}
					\mathcal{BS}_{S_0}\xrightarrow{\sim} \prod_{u\in \sW_n^{S_0}}\fl_{S_0(u)}\cap \fn_{\emptyset}.\;
					\end{aligned}
				\end{equation}
			
			\begin{pro}Let $\Dpik$ be a non-critical  semistable $(\varphi,\Gamma)$-module $\Dpik$ and  keep the notation in Sections \ref{Omegafil} and \ref{philoforParameter}.\;For any $u\in \sW_n^{S_0}$,\;
				$\{\underline{\sL}(E_{u,i})\}_{1\leq i\leq f_u}$ are captured by certain Breuil-Schraen $\sL$-invariants $\BW_{\Dpik}(u)\subseteq \BE(u)$.\;In particular,\;the collection $\{\BW_{\Dpik}(u)\}_{u\in \sW_n^{S_0}}$ determines $p_{\mathrm{ST}}(\underline{\sL}(\Dpik))$.\;
			\end{pro}
			\begin{rmk}
				Via the $p$-adic local Langlands correspondence,\;we expect that the conjectural locally analytic representation $\pi_{\ana}(\Dpik)$ has an ``extension structure'' that 
				comes from locally analytic parabolic inductions of locally analytic generalized Steinberg representations;\;see the two branchs of (\ref{struGL3ST2}) and its constructions for a $\GLN_{3}(\bQ_p)$-example.\;This is the key reason to consider $\ext^{\bullet}$-gruops between such representations.\;
			\end{rmk}

			\section{Locally analytic representations for semistable case}

				Throughout this section,\;fix a non-critical  semistable $(\varphi,\Gamma)$-module $\Dpik$ with Hodge-Tate weights $\bh=(\hpi_{1}>\hpi_{2}>\cdots>\hpi_{n})$ and  keep the notation in Section \ref{Omegafil}.\;Recall the two subsets $S_0\subseteq I_0\subseteq \Delta$.\;Write $\bL_{I_0}:=\bL_{I_0,1}\times \bL_{I_0,2} \times\cdots\times\bL_{I_0,m}$.\;For $1\leq j\leq m$,\;recall that $\Delta_{I_0,j}$ is the set of simple roots of $\bL_{I_0,j}$,\;so that  $I_0:=\sqcup_{1\leq j\leq m}\Delta_{I_0,j}$.\;For $1\leq j\leq m$,\;put $S_{0,j}:=\Delta_{I_0,j}\cap S_0$.\;


			For $I\subseteq \Delta$ and $u\in \sW_n^{S_0}$,\;consider the smooth principal series of $\bL_{I}(\bQ_p)$:
			\begin{equation}
				\begin{aligned}
					\mathrm{PS}^{\infty}_{I,u}(\underline{\phi})&:=\left(\ind^{\bL_I(\bQ_p)}_{\ob(\bQ_p)\cap \bL_I(\bQ_p)}\unr(\underline{\phi}^u)\eta\right)^{\infty}.
				\end{aligned}
			\end{equation}
			In particular,\;if $I=\{i\}$ for $i\in\Delta$,\;we rewrite the subscript $\{i\}$ with $i$ for simplicity.\;If $I=\Delta$,\;we drop the symbol $\Delta$.\;The irreducible constituents of $\mathrm{PS}^{\infty}_{1}(\underline{\phi})$ (equivalently,\;$\mathrm{PS}^{\infty}_{u}(\underline{\phi})$) are described as follows.\;Put $I=I_0$ and $u=1$,\;then
			\[\mathrm{PS}^{\infty}_{I_0,1}(\underline{\phi})\cong \boxtimes_{j=1}^m\Big(\ind^{\bL_{I_0,j}(\bQ_p)}_{\ob(\bQ_p)\cap \bL_{I_0,j}(\bQ_p)}\big(\unr(\underline{\phi})\eta\big)|_{\bT(\bQ_p)\cap \bL_{I_0,j}(\bQ_p)}\Big)^{\infty}.\]
			Note that $\big(\unr(\underline{\phi})\eta\big)|_{\bT(\bQ_p)\cap \bL_{I_0,j}(\bQ_p)}=\unr(\alpha_j)\otimes \mathrm{1}_{\bT(\bQ_p)\cap \bL_{I_0,j}(\bQ_p)}$ for some $\alpha_j\in E^{\times}$ by definition of $I_0$.\;For $J\subseteq I_0\subseteq I$,\;put 
			\[v^{\infty}_{J,I_0}(\underline{\phi}):=\boxtimes_{j=1}^m\unr(\alpha_j)\otimes_Ev^{\infty}_{J\cap \Delta_{I_0,j},\Delta_{I_0,j}},\;v^{\infty}_{J,I}(\underline{\phi}):=\left(\ind^{\bL_{I}(\bQ_p)}_{\op_{I_0}(\bQ_p)\cap \bL_{I}(\bQ_p) }v^{\infty}_{J,I_0}(\underline{\phi})\right)^{\infty}.\]
			\begin{pro}
				\begin{itemize}
					\item[(1)] For $J\subseteq I_0$,\;$v^{\infty}_{J,\Delta}(\underline{\phi})$ and $v^{\infty}_{J,I_0}(\underline{\phi})$ are  irreducible.\;The Jordan--Hölder factors of $\mathrm{PS}^{\infty}_{1}(\underline{\phi})$ (resp.,\;$\mathrm{PS}^{\infty}_{I_0,1}(\underline{\phi})$) are $\{v^{\infty}_{J,\Delta}(\underline{\phi})\}_{J\subseteq I_0}$ (resp.,\;$\{v^{\infty}_{J,I_0}(\underline{\phi})\}_{J\subseteq I_0}$).\;
					\item[(2)] For any $I_1,I_2\subseteq I_0 \subseteq I$,\;there exists a unique multiplicity-free finite-length representation $Q^{\diamond}_{I}(I_1,I_2)$ which has simple socle $v^{\infty}_{I_1,I}(\underline{\phi})$ and simple cosocle $v^{\infty}_{I_2,I}(\underline{\phi})$.\;
				\end{itemize} 
			\end{pro}
			\begin{proof}$(1)$ follows from \cite[Theorem 9.7]{av1980induced2}.\;It suffices to prove $(2)$.\;For $1\leq j\leq m$ and $I'_j,J'_j\subseteq
				\Delta_{I_0,j}$,\;by \cite[Lemma 2.2.1]{BQ24},\;there exists a unique multiplicity free finite length representation $Q_{\Delta_{I_0,j}}(I'_j,J'_j)$ which has simple socle $v^{\infty}_{I'_j,\Delta_{I_0,j}}$ and simple cosocle $v^{\infty}_{J'_j,\Delta_{I_0,j}}$.\;We put 
				\begin{equation*}
					\begin{aligned}
						&Q^{\diamond}_{I_0}(I_1,I_2):=\boxtimes_{j=1}^m\unr(\alpha_j)\otimes_EQ_{\Delta_{I_0,j}}(I_1\cap \Delta_{I_0,j},I_2\cap \Delta_{I_0,j}),\;\\
						&Q^{\diamond}_{I}(I_1,I_2):=\Big(\ind^{\bL_I(\bQ_p)}_{\overline{\bP}_{I_0}(\bQ_p)\cap \bL_I(\bQ_p)}Q^{\diamond}_{I_0}(I_1,I_2)\Big)^{\infty}.\;
					\end{aligned}
				\end{equation*}
				The uniqueness argument is similar to the proof of \cite[Lemma 2.2.1]{BQ24}.\;
			\end{proof}
			In particular,\;if $I_1=I_2=I'$,\;we write $Q^{\diamond}_{I}(I'):=Q^{\diamond}_{I}(I',I')$ for simplicity.\;Note that $\mathrm{PS}^{\infty}_{1}(\underline{\phi})=Q^{\diamond}_{\Delta}(I_0,\emptyset)$ and $\mathrm{PS}^{\infty}_{w_0^{(S_0)}}(\underline{\phi})=Q^{\diamond}_{\Delta}(\emptyset,I_0)$,\;where $w_0^{(S_0)}=j_{S_0}(w_{0,s})$,\;see the notation in the proof of Proposition  \ref{infinitefernforsemi}.\;The unique irreducible generic constituent of 
			$\mathrm{PS}^{\infty}_{1}(\underline{\phi})$ is equal to $Q^{\diamond}_{\Delta}(\emptyset)$.\;

			Let $\theta=(0,-1,\cdots,1-n)\in X_{\Delta}^+$.\;Put $\lambda:=\lambda_{\Dpik}=\bh_{\Dpik}-\theta\in X_{\Delta}^+$.\;For any $I_1,I_2\subseteq I_0$,\;put $Q_{\Delta}^{\diamond}(I_1,I_2,\lambda):=Q_{\Delta}^{\diamond}(I_1,I_2)\otimes_EL(\lambda)$,\;which is a locally $\bQ_p$-algebraic representation of $G$.\;In the sequel,\;write
			\[\pi_{\natural}^{\lalg}(\underline{\phi},\bh)=Q_{\Delta}^{\diamond}(I_0,\emptyset,\lambda),\quad\pi_1^{\lalg}(\underline{\phi},\bh)=Q_{\Delta}^{\diamond}(\emptyset,I_0,\lambda).\]
			For any $u\in \sW_n^{S_0}$ and $j\in \Delta$,\;put $\widehat{j}=\Delta\backslash\{j\}$.\;Let $\pi_{\widehat{j},u}(\underline{\phi})$ be the unique generic irreducible constituent of $\mathrm{PS}^{\infty}_{\widehat{j},u}(\underline{\phi})$.\;Consider the locally $\bQ_p$-analytic representation of $G$ (see \cite[The main theorem]{orlik2015jordan}): 
			\[C_{j,u}:=\cF^{G}_{\op_{\widehat{j}}}\Big(\overline{L}(-s_{j}\cdot{\lambda}),\pi_{\widehat{j},u}(\underline{\phi})\Big).\]
			\begin{lem}For $u,u'\in \sW_n^{S_0}$,\;$C_{j,u}\cong C_{j,u'}$ if and only if $\{u^{-1}(1),\cdots,u^{-1}(j)\}=\{(u')^{-1}(1),\cdots,(u')^{-1}(j)\}$. In particular,\;$C_{j,u}$ only depends on the coset $\sW_{\widehat{j}}u^{-1}$.\;
			\end{lem}

			For $j\in \Delta$ and $u\in \sW_n^{S_0}$, let $[u]_j$ be the minimal length representative in $\sW_{\widehat{j}}u^{-1}$. We put
			\[\cI_j^{\sharp}:=\big\{[u]_j: u\in \sW_n^{S_0}\big\}.\]
		By \cite[Lemma 2.1.31]{BQ24},\;we have
			\begin{lem}For $u\in \cI_j^{\sharp}$ and $j\in \Delta$,\;$\big(\ind^G_{\op_{\widehat{j}}(\bQ_p)}\pi_{\widehat{j},u}(\underline{\phi})\big)^{\infty}=Q^{\diamond}_{\Delta}(I_u^+,I_u^-)$ for $I_u^+:=D_R(w_{\widehat{j}}u^{-1})\cap I_0=D_L(uw_{\widehat{j}})\cap I_0$ and $I_u^-:=D_R(u^{-1})\cap I_0=D_L(u)\cap I_0$.\;
			\end{lem}
			
			\subsection{Preliminaries on extension groups}
			
			\begin{lem}\label{parameterline} There is an isomorphism of $(n+1)$-dimensional $E$-vector spaces:
				\begin{equation}
					\begin{aligned}
						&\homo_{\sm}(\bT(\bQ_p),E)\oplus_{\homo_{\sm}(\bQ_p^{\times},E)}\homo(\bQ_p^{\times},E)\xrightarrow{\sim} \ext^1_{G}\left(\pi_{\natural}^{\lalg}(\underline{\phi},\bh),\pi_1^{\lalg}(\underline{\phi},\bh)\right).
					\end{aligned}
				\end{equation}
			\end{lem}
			
			\begin{proof}Write $\pi_{\sharp}^{\lalg}(\underline{\phi},\bh):=\pi_{\natural}^{\lalg}(\underline{x},\bh)$ and $\pi^{\lalg}_1:=\pi_1^{\lalg}(\underline{x},\bh)$ for simplicity. We define the map and prove that $\dim_E\ext^1_{G}\left(\pi_{\natural}^{\lalg},\pi_1^{\lalg}\right)=n+1$. We define a canonical injection:
				\begin{equation}\label{lalginj1}
					\begin{aligned}
						\homo_{\sm}(\bT(\bQ_p),E)&\rightarrow\ext^1_{G}\big(\pi_1^{\lalg},\pi_1^{\lalg}\big), \\
						\psi&\mapsto \Big(\ind^G_{\ob(\bQ_p)}\unr(\underline{\phi}^{w_0^{(S_0)}})\eta\otimes_E\big(1+\psi\epsilon\big)\Big)^{\mathrm{sm}}\otimes_EL(\lambda).
					\end{aligned}
				\end{equation}
				Composed with the pull-back map for the natural maps $\pi_{\natural}^{\lalg}\twoheadrightarrow Q_{\Delta}^{\diamond}(\emptyset,\lambda)\hookrightarrow \pi_1^{\lalg}$, we get a map
				\[\homo_{\sm}(\bT(\bQ_p),E)\rightarrow \ext^1_{G}\big(Q_{\Delta}^{\diamond}(\emptyset,\lambda),\pi_1^{\lalg}\big)\rightarrow \ext^1_{G}\left(\pi_{\natural}^{\lalg},\pi_1^{\lalg}\right).\]
				Let $Q'=\ker (\pi_{\natural}^{\lalg}\twoheadrightarrow Q_{\Delta}^{\diamond}(\emptyset,\lambda))$. We have the following two short exact sequences:
				\begin{equation}\label{compuextseqforalg}
					\begin{aligned}
						&0\rightarrow \ext^1_{G}\left(Q_{\Delta}^{\diamond}(\emptyset,\lambda),\pi_1^{\lalg}\right) \rightarrow \ext^1_{G}\left(\pi_{\natural}^{\lalg},\pi_1^{\lalg}\right) \rightarrow  \ext^1_{G}\left(Q',\pi_1^{\lalg}\right), \\
						&0\rightarrow \ext^1_{G,Z}\left(Q_{\Delta}^{\diamond}(\emptyset,\lambda),\pi_1^{\lalg}\right) \rightarrow \ext^1_{G,Z}\left(\pi_{\natural}^{\lalg},\pi_1^{\lalg}\right) \rightarrow \ext^1_{G,Z}\left(Q',\pi_1^{\lalg}\right).
					\end{aligned}
				\end{equation}
				Note that $\ext^1_{G,Z}\left(Q',\pi_1^{\lalg}\right)\cong \ext^1_{G}\left(Q',\pi_1^{\lalg}\right)$ follows from \cite[Lemma 3.2]{Dilogarithm}. By \cite[Proposition 4.7]{schraen2011GL3} and the adjunction formula between the (smooth) Jacquet functor and parabolic induction, we see that $\ext^1_{G,Z}\left(Q',\pi_1^{\lalg}\right)\cong \ext^{1,\lalg}_{G,Z}\left(Q',\pi_1^{\lalg}\right)=0$. Thus, $\ext^1_{G}\left(Q_{\Delta}^{\diamond}(\emptyset,\lambda),\pi_1^{\lalg}\right) \xrightarrow{\sim} \ext^1_{G}\left(\pi_{\natural}^{\lalg},\pi_1^{\lalg}\right)$. On the other hand, we have \[\ext^1_{G,Z}\left(Q_{\Delta}^{\diamond}(\emptyset,\lambda),\pi_1^{\lalg}\right)\cong \ext^{1,\lalg}_{G,Z}\left(Q_{\Delta}^{\diamond}(\emptyset,\lambda),\pi_1^{\lalg}\right)\] which is $s-1$-dimensional by the (smooth) adjunction formula. Therefore, the restriction of (\ref{lalginj1}) to the subspace $\homo_{\sm}(\overline{\bT}(\bQ_p),E)$ is an injection onto $\ext^1_{G,Z}\left(\pi_{\natural}^{\lalg},\pi_1^{\lalg}\right)$, where $\overline{\bT}:=\bT/\bZ$. We define a second canonical injection:
				\begin{equation}
					\homo(\bQ_p^{\times},E)\rightarrow\ext^1_{G}\big(\pi_1^{\lalg},\pi_1^{\lalg}\big),
					\psi\mapsto \pi_1^{\lalg}\otimes_E\big(1+\big(\psi\circ\mathrm{det}_{\GLN_n}\big)\epsilon\big).
				\end{equation}
				Composed with the pull-back map for the natural maps $\pi_{\natural}^{\lalg}\twoheadrightarrow Q_{\Delta}^{\diamond}(\emptyset,\lambda)\hookrightarrow \pi_1^{\lalg}$, we get a map (and this composition is also injective):
				\begin{equation}\label{lalginj2}
					\homo(\bQ_p^{\times},E)\rightarrow \ext^1_{G}\big(Q_{\Delta}^{\diamond}(\emptyset,\lambda),\pi_1^{\lalg}\big)\rightarrow \ext^1_{G}\left(\pi_{\natural}^{\lalg},\pi_1^{\lalg}\right).
				\end{equation}
				The two injections (\ref{lalginj1}) and (\ref{lalginj2}) coincide on $\homo_{\sm}(L^{\times},E)$. It remains to prove that
				\[\dim_E\ext^1_{G}\left(\pi_{\natural}^{\lalg},\pi_1^{\lalg}\right)=n+1.\]
				If these identities hold, then we get the isomorphism in the lemma. Write $\pi_1^{\lalg}=[Q_{\Delta}^{\diamond}(\emptyset,\lambda)-Q'']$. We have the following two exact sequences:
				\begin{equation}
					\begin{aligned}
						0\rightarrow \ext^1_{G}\left(Q_{\Delta}^{\diamond}(\emptyset,\lambda),Q_{\Delta}^{\diamond}(\emptyset,\lambda)\right) \rightarrow \ext^1_{G}&\big(Q_{\Delta}^{\diamond}(\emptyset,\lambda),\pi_1^{\lalg}\big)\rightarrow \ext^1_{G,Z}\left(Q_{\Delta}^{\diamond}(\emptyset,\lambda),Q''\right),\\
						0\rightarrow \ext^1_{G,Z}\left(Q_{\Delta}^{\diamond}(\emptyset,\lambda),Q_{\Delta}^{\diamond}(\emptyset,\lambda)\right) \rightarrow \ext^1_{G,Z}&\big(Q_{\Delta}^{\diamond}(\emptyset,\lambda),\pi_1^{\lalg}\big)\rightarrow \ext^1_{G,Z}\left(Q_{\Delta}^{\diamond}(\emptyset,\lambda),Q''\right)\rightarrow 0,
					\end{aligned}
				\end{equation}
				since $\ext^i_{G,Z}\left(Q_{\Delta}^{\diamond}(\emptyset,\lambda),Q_{\Delta}^{\diamond}(\emptyset,\lambda)\right)=0$ for $i=1,2$. Thus, we get from the second exact sequence that $\ext^1_{G,Z}\left(Q_{\Delta}^{\diamond}(\emptyset,\lambda),Q''\right)\cong \ext^{1,\lalg}_{G,Z}\left(Q_{\Delta}^{\diamond}(\emptyset,\lambda),Q''\right)\cong \ext^1_{G,Z}\left(Q_{\Delta}^{\diamond}(\emptyset,\lambda),\pi_1^{\lalg}\right)$ and that the last map in the first exact sequence is surjective. By the same argument as in the proof of $(2)$ \cite[Lemma 3.17]{HigherLinvariantsGL3(Qp)}, we see that $\ext^1_{G}\left(Q_{\Delta}^{\diamond}(\emptyset,\lambda),Q_{\Delta}^{\diamond}(\emptyset,\lambda)\right)\cong \homo(\bQ_p^{\times},E)$. We get the desired dimension.
			\end{proof}

			\begin{lem}\label{parameterlineextgp}
				Let $j\in \Delta$ and $u\in \cI^{\sharp}_{j}$. For $I\in [I_u^+,I_u^-]:=\{I\subseteq \Delta\;|\;I_u^+\cap I_u^-\subseteq I\subseteq I_u^+\cup I_u^-\}$, we have
				\begin{equation}
					\begin{aligned}
						\dim_E\ext^1_G\left(C_{j,u},Q_{\Delta}^{\diamond}(I,I_u^{-},\lambda)\right)=1,
						\dim_E\ext^1_G\left(Q_{\Delta}^{\diamond}(I_u^+,I,\lambda),C_{j,u}\right)=1,
					\end{aligned}
				\end{equation}
				which leads to a unique representation of the form $[Q_{\Delta}^{\diamond}(I,I_u^{-},\lambda)-C_{j,u}]$ (resp., $[C_{j,u}-Q_{\Delta}^{\diamond}(I_u^{+},I,\lambda)]$). The same holds when $Q_{\Delta}^{\diamond}(I,I_u^{-},\lambda)$ (resp., $Q_{\Delta}^{\diamond}(I_u^+,I,\lambda)$) is replaced by $\pi_1^{\lalg}(\underline{\phi},\bh)$ (resp., $\pi_{\natural}^{\lalg}(\underline{\phi},\bh)$).
			\end{lem}
			\begin{proof}The first assertion follows from \cite[Proposition 5.1.14]{BQ24}.\;We prove the existence of unique representations of the forms $[\pi_1^{\lalg}-C_{j,u}]$ and $[C_{j,u}-\pi_{\natural}^{\lalg}]$.\;Indeed, for any irreducible constituent $V$ appearing in $\pi_1^{\lalg}/Q_{\Delta}^{\diamond}(\emptyset,I_u^{-},\lambda)$, we also get from \cite[Proposition 5.1.14]{BQ24} that $\ext^1_G\left(C_{j,u},V\right)=0$ since $\homo(Q_{\Delta}^{\diamond}(I_u^+,I_{u}^-,\lambda),V)=0$, so that the injection $Q_{\Delta}^{\diamond}(\emptyset,I_u^{-},\lambda)\hookrightarrow\pi_1^{\lalg}$ induces an isomorphism
				$\ext^1_G\big(C_{j,u},Q_{\Delta}^{\diamond}(\emptyset,I_u^{-},\lambda)\big)\xrightarrow{\sim}\ext^1_G\big(C_{j,u},\pi_1^{\lalg}\big)$.\;Similarly, the map $\ext^1_G\big(Q_{\Delta}^{\diamond}(I_u^{+},\emptyset,\lambda),C_{j,u}\big)\rightarrow\ext^1_G\big(\pi_{\natural}^{\lalg},C_{j,u}\big)$ induced by the surjection $\pi_{\natural}^{\lalg}\twoheadrightarrow Q_{\Delta}^{\diamond}(I_u^{+},\emptyset,\lambda)$ is also an isomorphism.
			\end{proof}
			
			For $u\in \sW_n^{S_0}$ and $i\in \Delta$, fix a locally analytic representation $V_{i,u}$ of $G$, which is isomorphic to a non-split extension of $C_{i,u}$ by $\pi_1^{\lalg}(\underline{\phi},\bh)$ (see Lemma \ref{parameterlineextgp}; $V_{i,u}$ is unique up to a scalar), i.e.,
			\[V_{i,u}:=[\pi_1^{\lalg}(\underline{\phi},\bh)-C_{i,u}],\]
			and fix an injection $\iota_{i,u}:\pi_1^{\lalg}(\underline{\phi},\bh)\hookrightarrow V_{i,u}$. We have $C_{i,u}\cong V_{i,u}/\pi_1^{\lalg}(\underline{\phi},\bh)$ and $\bL_{\widehat{i}}\cong \GLN_i\times\GLN_{n-i}$.

			\begin{pro}\label{reinterforextgp} For $i\in \Delta$, there is a canonical isomorphism of $(n+2)$-dimensional $E$-vector spaces associated to $(V_{i,u},\iota_{i,u})$:
			\begin{equation}\label{isoforfullextgp}
				\begin{aligned}
					\homo_{\sm}(\bT(\bQ_p),E)\oplus_{\homo_{\sm}(\bL_{\widehat{i}}(\bQ_p),E)}\homo(\bL_{\widehat{i}}(\bQ_p),E)\xrightarrow{\sim}\ext^1_{G}\left(\pi_{\natural}^{\lalg}(\underline{\phi},\bh),V_{i,u}\right).
				\end{aligned}
			\end{equation}
			The restriction of (\ref{isoforfullextgp}) to $\homo(\bL_{\widehat{i}}(\bQ_p),E)$ depends only on $(V_{i,u},\iota_{i,u})$. Moreover, there is a short exact sequence:
			\begin{equation}
				\begin{aligned}
					0\rightarrow \ext^1_{G}\Big(\pi_{\natural}^{\lalg},\pi_{1}^{\lalg}\Big) \xrightarrow{\iota_{i,u}} \ext^1_{G}\Big(\pi_{\natural}^{\lalg},V_{i,u}\Big) \rightarrow \ext^1_{G}\Big(\pi_{\natural}^{\lalg},V_{i,u}/\pi_1^{\lalg}(\underline{\phi},\bh)\Big)\rightarrow 0.
				\end{aligned}
			\end{equation}
			\end{pro}
			\begin{proof}It follows from Lemma \ref{parameterline} and  Lemma \ref{parameterlineextgp} and the exact sequence
			$0\rightarrow \ext^1_{G}\Big(\pi_{\natural}^{\lalg},\pi_{1}^{\lalg}\Big) \rightarrow \ext^1_{G}\Big(\pi_{\natural}^{\lalg},V_{i,u}\Big) \rightarrow \ext^1_{G}\Big(\pi_{\natural}^{\lalg},V_{i,u}/\pi_1^{\lalg}\Big)$ that $\dim_E\ext^1_{G}\Big(\pi_{\natural}^{\lalg}(\underline{x},\bh),V_{i,u}\Big)\leq n+2$. We construct the desired map and show that the last map in this exact sequence is surjective. We first define the injection (by Lemma \ref{parameterline}):
			\begin{equation}
				\homo_{\sm}(\bT(\bQ_p),E)\hookrightarrow \ext^1_{G}\Big(\pi_{\natural}^{\lalg},\pi_{1}^{\lalg}\Big)\hookrightarrow \ext^1_{G}\Big(\pi_{\natural}^{\lalg},V_{i,u}\Big).
			\end{equation}
			We next construct a canonical injection $\homo(\bL_{\widehat{i}}(\bQ_p),E)\hookrightarrow \ext^1_{G}\left(\pi_{\natural}^{\lalg}(\underline{x},\bh),V_{i,u}\right)$. Put
			\begin{equation}
				R_{u}:=\cF^{G}_{\op_{\widehat{i}}}\Big(\overline{M}_{\widehat{i}}(-\lambda),\pi^{\infty}_{\widehat{i},u}(\underline{\phi})\Big).
			\end{equation}
			Let $R_u'$ be the unique quotient of $R_{u}$ with socle $Q_{\Delta}^{\diamond}(\emptyset,I_u^-,\lambda)$. Consider the following injection:
			\begin{equation}
				\begin{aligned}
					\homo(\bL_{\widehat{i}}(\bQ_p),E)&\hookrightarrow \ext^1_{G}\left(R_{u},R_{u}\right)\\
					\psi&\mapsto \left(\ind^G_{\op_{\widehat{i}}(\bQ_p)}\pi^{\infty}_{\widehat{i},u}(\underline{\phi})\otimes L_{\widehat{i}}(\lambda_{\sigma})\otimes(1+\psi\epsilon)\right)^{\ana}.
				\end{aligned}
			\end{equation}
			Composed with the pull-back map for the natural map $\pi_{\natural}^{\lalg}\twoheadrightarrow Q_{\Delta}^{\diamond}(I_u^+,\emptyset,\lambda)\hookrightarrow Q_{\Delta}^{\diamond}(I_u^+,I_u^-,\lambda)\hookrightarrow R_{u}$ and the push-forward map via $R_{u}\twoheadrightarrow R_u'$, we get a map
			\begin{equation}
				\begin{aligned}
					f_u:&\;\homo(\bL_{\widehat{i}}(\bQ_p),E)\hookrightarrow \ext^1_{G}\big(\pi_{\natural}^{\lalg},R_{u}\big)\\
					&\rightarrow \ext^1_{G}\big(\pi_{\natural}^{\lalg},R'_{u}\big)\cong \ext^1_{G}\Big(\pi_{\natural}^{\lalg},[Q_{\Delta}^{\diamond}(\emptyset,I_u^-,\lambda)-C_{i,u}]\Big)\rightarrow \ext^1_{G}\big(\pi_{\natural}^{\lalg},V_{i,u}\big),
				\end{aligned}
			\end{equation}
			which is still an injection, and we can prove that it depends only on $(V_{i,u},\iota_{i,u})$. The map $f_u$ induces an injection
			\[\homo_{\sm}(\bL_{\widehat{i}}(\bQ_p),E)\rightarrow \ext^{1}_{G}\big(\pi_{\natural}^{\lalg},\pi_1^{\lalg}\big),\]
			which is compatible with the injection $\homo_{\sm}(\bT(\bQ_p),E)\rightarrow\ext^1_{G}\big(\pi_{\natural}^{\lalg},\pi_1^{\lalg}\big)$. Thus, we get the desired map. On the other hand, by the second isomorphism in Lemma \ref{reinterforextgp}, we see that the image of $\homo(\bL_{\widehat{i}}(\bQ_p),E)$ in $\ext^1_{G}\big(\pi_{\natural}^{\lalg},V_{i,u}\big)$ is not contained in $\ext^1_{G}\left(\pi_{\natural}^{\lalg}(\underline{x},\bh),\pi_1^{\lalg}(\underline{x},\bh)\right)$, so that the second map in (\ref{isoforfullextgp}) is a surjection and thus an isomorphism by comparing the dimensions.
			\end{proof}
			
			\begin{pro}
				We have a perfect pairing of $1$-dimensional $E$-vector spaces:
				\begin{equation}\label{dualpairing}
					\ext^1_{G}\left(\pi_{\natural}^{\lalg}(\underline{x},\bh),C_{i,u}\right)\times\ext^1_{G}\left(C_{i,u},\pi_1^{\lalg}(\underline{x},\bh)\right)\xrightarrow{\sim}\homo(\GLN_{n-i}(\bZ_p),E)\xleftarrow{\sim }E,
				\end{equation}
				where the second canonical isomorphism is given by $E\xrightarrow{\sim }\homo(\GLN_{n-i}(\bZ_p),E)$, $1\mapsto \log_{p}\circ \det$.
			\end{pro}
			\begin{proof}
				Similar to the proof of \cite[Proposition 2.1.10]{BDcritical25}, we have the following (formal) isomorphism of $1$-dimensional $E$-vector spaces:
				\[\ext^1_{G}\left(C_{i,u},\pi_1^{\lalg}\right)\otimes_E\ext^1_{G}\left(\pi_{\sharp}^{\lalg}(\underline{\phi},\bh),C_{i,u}\right)^{\vee}\xrightarrow{\sim}\ext^1_{G}\left(C_{i,u}\otimes\ext^1_{G}\left(\pi_{\sharp}^{\lalg}(\underline{\phi},\bh),C_{i,u}\right),\pi_1^{\lalg}\right),\]
				where the $G$-representation $C_{i,u}\otimes\ext^1_{G}\left(\pi_{\sharp}^{\lalg}(\underline{\phi},\bh),C_{i,u}\right)$ has the trivial action of $G$ on the term $\ext^1_{G}\left(\pi_{\sharp}^{\lalg}(\underline{\phi},\bh),C_{i,u}\right)$. Any representative $V'_{i,u}$ of the image of the canonical vector of the left-hand side is isomorphic to a non-split extension of $C_{i,u}$ by $\pi_1^{\lalg}$. Then $V'_{i,u}$ corresponds to an injection $\iota:\pi_1^{\lalg}\hookrightarrow V'_{i,u}$ and an isomorphism $\kappa:V'_{i,u}/\pi_1^{\lalg}\xrightarrow{\sim}C_{i,u}\otimes_E\ext^1_{G}\left(\pi_{\sharp}^{\lalg}(\underline{\phi},\bh),C_{i,u}\right)$. By the same strategy as in \cite[(37)]{BDcritical25}, we have a canonical isomorphism of $1$-dimensional $E$-vector spaces (where the embedding $\homo(\bZ_p^{\times},E)\hookrightarrow \homo(\bL_{\widehat{i}}(\bZ_p^{\times}),E)$ is induced by $\det:\bL_{\widehat{i}}(\bZ_p^{\times})\rightarrow \bZ_p^{\times}$):
				\[\homo(\GLN_{n-i}(\bZ_p),E)\xrightarrow{\sim}\homo(\bL_{\widehat{i}}(\bZ_p^{\times}),E)/\homo(\bZ_p^{\times},E).\]
				By the proof of Proposition \ref{reinterforextgp} (similar to \cite[(34)]{BDcritical25}), the latter space is isomorphic to $\ext^1_{G}\left(\pi_{\natural}^{\lalg},V'_{i,u}/\pi_1^{\lalg}\right)$ and thus isomorphic to (via $\kappa$) $\ext^1_{G}\left(\pi_{\natural}^{\lalg},C_{i,u}\otimes_E\ext^1_{G}\left(\pi_{\sharp}^{\lalg}(\underline{\phi},\bh),C_{i,u}\right)\right)\xleftarrow{\sim} \ext^1_{G}\left(\pi_{\natural}^{\lalg},C_{i,u}\right)\otimes_E\ext^1_{G}\left(\pi_{\sharp}^{\lalg}(\underline{\phi},\bh),C_{i,u}\right)$, the result follows.
			\end{proof}

			\subsection{Constructions of locally analytic representations}\label{secforconsoflocana}
			
			For $u\in\sW_n^{S_0}$,\;consider the locally $\bQ_p$-analytic parabolic induction
			\begin{equation}\label{locaparabolicind}
				\mathrm{PS}_{u}(\underline{\phi},\bh):=\Big(\ind^G_{\ob(\bQ_p)}\unr(\underline{\phi}^u)\eta\chi_{\lambda}\Big)^{\bQ_p-\ana}.
			\end{equation}
			Note that the locally algebraic vectors in $\mathrm{PS}_{u}(\underline{\phi},\bh)$ are equal to
d			\[\mathrm{PS}^{\lalg}_{u}(\underline{\phi},\bh)=Q_{\Delta}^{\diamond}(I_0\backslash I_{u}^-,I_{u}^-,\lambda).\]
				Let $\mathrm{ST}_{u}(\underline{\phi},\bh)$  be the unique maximal quotient of $\mathrm{PS}_{u}(\underline{\phi},\bh)$ (see (\ref{locaparabolicind})) with socle $Q_{\Delta}^{\diamond}(\emptyset,\lambda)$.\;Note that the locally algebraic vectors $\mathrm{ST}^{\lalg}_{u}(\underline{\phi},\bh)$   in $\mathrm{ST}_{u}(\underline{\phi},\bh)$ are equal to $Q_{\Delta}^{\diamond}(\emptyset,I_{u}^-,\lambda)$.\;Let $\mathrm{ST}_{u,1}(\underline{\phi},\bh)$ be  the ``first two layers" of the socle filtration of 
				$\mathrm{ST}_u(\underline{\phi},\bh)$,\;i.e.,\;the unique subrepresentation of $\mathrm{ST}_u(\underline{\phi},\bh)$ satisfying the form:
				\[\big[\mathrm{ST}^{\lalg}_u(\underline{\phi},\bh)-\soc_G\mathrm{ST}_u(\underline{\phi},\bh)/\mathrm{ST}^{\lalg}_u(\underline{\phi},\bh)\big].\]
				For $u\in\sW_{n}$,\;recall that $[u]_j\in \cI^{\sharp}_j$ is the minimal length representative in $\sW_{\widehat{j}}u$.\;
				
%
%
%

			\begin{lem}\label{stexttogenericfactor}For $u\in \sW_s$ and  $j\in \Delta$,\;there is a short exact sequence of locally $\bQ_p$-analytic representations
				\begin{equation}\label{firstwholeforpatialu}
					0\rightarrow Q_{\Delta}^{\diamond}(\emptyset,I_{u}^-,\lambda) \rightarrow \mathrm{ST}_{u,1}(\underline{\phi},\bh)\rightarrow \oplus_{j\in \Delta}C_{j,[u]_j}\rightarrow 0.
				\end{equation}
			\end{lem}
			\begin{proof}We have
				$\ext^1_G\left(C_{j,[u]_j},Q_{\Delta}^{\diamond}(\emptyset,I_{u}^-,\lambda)\right)\neq 0$ by Lemma \ref{parameterlineextgp}.\;For any other non-generic irreducible constituent $V$ of $\mathrm{PS}^{\infty}_{\widehat{j},u}(\underline{\phi})$,\;since
				\[\homo\big(\big(\ind^G_{\op_{\widehat{i}}(\bQ_p)}V\big)^{\infty},Q_{\Delta}^{\diamond}(\emptyset,I_{u}^-,\lambda)\big)=0,\]
				we obtain that $\ext^1_G\big(\cF^{G}_{\op_{\widehat{i}}}\big(\overline{L}(-s_{i}\cdot{\lambda}),V\big),Q_{\Delta}^{\diamond}(\emptyset,I_{u}^-,\lambda)\big)=0$ by \cite[Proposition 5.1.14]{BQ24}.\;Therefore,\;only the non-split extension $[Q_{\Delta}^{\diamond}(\emptyset,I_{u}^-,\lambda)-C_{j,[u]_j}]$ lies in $\mathrm{ST}_{u,1}(\underline{\phi},\bh)$.\;The result follows.\;
			\end{proof}

			Consider the amalgamated sums:
			\[\pi^{\flat}_{1}(\underline{\phi},\bh)_u:=\bigoplus_{\pi_1^{\lalg}(\underline{\phi},\bh)}^{j\in \Delta\backslash S_0(u)}V_{j,[u]_j}\hookrightarrow \pi^{\sharp}_{1}(\underline{\phi},\bh)_u:=\bigoplus_{\pi_1^{\lalg}(\underline{\phi},\bh)}^{j\in \Delta}V_{j,[u]_j}.\;\]
			We have the following short exact sequences of locally analytic representations:
			\begin{equation}\label{firstwholeforu}
				\begin{aligned}
					&0\rightarrow \pi^{\lalg}_{1}(\underline{\phi},\bh) \rightarrow \pi^{\flat}_{1}(\underline{\phi},\bh)_u\rightarrow \oplus_{j\in \Delta\backslash S_0(u)}C_{j,[u]_j}
					\rightarrow 0,\\
					&0\rightarrow \pi^{\lalg}_{1}(\underline{\phi},\bh) \rightarrow \pi^{\sharp}_{1}(\underline{\phi},\bh)_u\rightarrow \oplus_{j\in \Delta}C_{j,[u]_j}
					\rightarrow 0.
				\end{aligned}
			\end{equation}
			For $i\in \Delta$,\;put 
			\[\cI^{\flat}_{i}:=\{u\in \cI^{\sharp}_{i}:i\in \Delta\backslash S_0(u)\}\subseteq \cI^{\sharp}_{i}.\;\]
			For $\star\in\{\sharp,\flat\}$,\;consider the amalgamated sums:
			\[\pi^{\star}_{1}(\underline{\phi},\bh):=\bigoplus_{\pi_1^{\lalg}(\underline{\phi},\bh)}^{j\in \Delta,u\in\cI^{\star}_{j}}V_{j,u}.\]
			We have the following short exact sequence of locally analytic representations:
			\begin{equation}\label{firstwhole}
				\begin{aligned}
					&0\rightarrow \pi^{\lalg}_{1}(\underline{\phi},\bh) \rightarrow \pi^{\star}_{1}(\underline{\phi},\bh)\rightarrow \oplus_{j\in \Delta,u\in\cI^{\star}_{j}}C_{j,u}\rightarrow 0.
				\end{aligned}
			\end{equation}
			Roughly speaking,\;$\pi^{\sharp}_{1}(\underline{\phi},\bh)$ is the unique quotient of  amalgamated sum $\bigoplus_{\pi_1^{\lalg}(\underline{\phi},\bh)}^{u\in\sW_n^{S_0}}\pi^{\star}_{1}(\underline{\phi},\bh)_u$ with socle $Q_{\Delta}^{\diamond}(\emptyset,\lambda)$,\;and is also the unique quotient of
			certain "amalgamated sum" of $\{\mathrm{ST}_{u,1}(\underline{\phi},\bh)\}_{u\in \sW_n^{S_0}}$ with socle $Q_{\Delta}^{\diamond}(\emptyset,\lambda)$.\;In particular,\;we have injections $\mathrm{ST}_{u,1}(\underline{\phi},\bh)\hookrightarrow \pi_{1}^{\sharp}(\underline{\phi},\bh)_u\hookrightarrow \pi_{1}^{\sharp}(\underline{\phi},\bh)$ and $\pi^{\flat}_{1}(\underline{\phi},\bh)\hookrightarrow \pi^{\sharp}_{1}(\underline{\phi},\bh)$ of $G$-representations.\;

			For $u\in \sW_n^{S_0}$,\;we consider the natural map
			\begin{equation}
				\begin{aligned}
					\homo(\bT(\bQ_p),E)&\rightarrow\ext^1_{G}\big(\mathrm{PS}_{u}(\underline{\phi},\bh),\mathrm{PS}_{u}(\underline{\phi},\bh)\big),\;\\
					\psi&\mapsto \Big(\ind^G_{\ob(\bQ_p)}\unr(\underline{\phi}^u)\eta\chi_{\lambda}\otimes_E(1+\psi\epsilon)\Big)^{\bQ_p-\ana}.
				\end{aligned}
			\end{equation}
			In particular,\;by Schraen's spectral sequence \cite[(4.37)]{schraen2011GL3},\;this map induces  an isomorphism
			\begin{equation}
				\begin{aligned}
					\homo(\bT(\bQ_p),E)&\xrightarrow{\sim}\ext^1_{G}\Big(\mathrm{PS}^{\lalg}_{u}(\underline{\phi},\bh),\mathrm{PS}_{u}(\underline{\phi},\bh)\Big).
				\end{aligned}
			\end{equation}
			Then the natural quotient map $\mathrm{PS}_{u}(\underline{\phi},\bh)\twoheadrightarrow \mathrm{ST}_{u}(\underline{\phi},\bh)$ induces a map 
			\[\homo(\bT(\bQ_p),E)\rightarrow\ext^1_{G}\left(\mathrm{PS}^{\lalg}_{u}(\underline{\phi},\bh),\mathrm{ST}_{u}(\underline{\phi},\bh)\right).\;\]
			By \cite[Lemma 2.26]{2019DINGSimple},\;this map factors through (induced by the injection $\mathrm{ST}_{u,1}(\underline{\phi},\bh)\hookrightarrow \mathrm{ST}_{u}(\underline{\phi},\bh)$)
			\[\ext^1_{G}\left(\pi_{\natural}^{\lalg}(\underline{\phi},\bh),\mathrm{ST}_{u,1}(\underline{\phi},\bh)\right)\rightarrow \ext^1_{G}\left(\pi_{\natural}^{\lalg}(\underline{\phi},\bh),\mathrm{ST}_{u}(\underline{\phi},\bh)\right),\]
			then we obtain a map $\homo(\bT(\bQ_p),E)\rightarrow\ext^1_{G}\left( \mathrm{PS}^{\lalg}_{u}(\underline{\phi},\bh),\mathrm{ST}_{u,1}(\underline{\phi},\bh)\right)$.\;Composed with the pull-back map for the natural map
			\[p_u:\pi_{\natural}^{\lalg}(\underline{\phi},\bh)=Q_{\Delta}^{\diamond}(I_0,\emptyset,\lambda)\twoheadrightarrow Q_{\Delta}^{\diamond}(I_0\backslash I_u^-,\emptyset,\lambda)\hookrightarrow Q_{\Delta}^{\diamond}(I_0\backslash I_u^-,I_u^-,\lambda)=\mathrm{PS}_{u}^{\lalg}(\underline{\phi},\bh),\] 
			we get a map:
			\begin{equation}\label{stextalg}
				\homo(\bT(\bQ_p),E)\rightarrow\ext^1_{G}\left( \pi_{\natural}^{\lalg}(\underline{\phi},\bh),\mathrm{ST}_{u,1}(\underline{\phi},\bh)\right).\;
			\end{equation}
			Composing with the push-forward map via the injection $\mathrm{ST}_{u,1}(\underline{\phi},\bh)\hookrightarrow \pi^{\sharp}_1(\underline{\phi},\bh)$,\;we actually obtain a map:
			\[\zeta_{u}:\homo(\bT(\bQ_p),E)\rightarrow\ext^1_{G}\left(\pi_{\natural}^{\lalg}(\underline{\phi},\bh),\pi^{\sharp}_1(\underline{\phi},\bh)\right).\;\]
			
			\begin{pro}\label{extpi1repn}
				\begin{itemize}
					\item[(1)] We have $\dim_E\ext^1_{G}\left(\pi_{\natural}^{\lalg}(\underline{\phi},\bh),\pi_1^{\sharp}(\underline{\phi},\bh)_u\right)=2n$.\;There is a short exact sequence:
					\begin{equation}\label{extenforPSw}
						\begin{aligned}
							0\rightarrow \ext^1_{G}\left(\pi_{\natural}^{\lalg}(\underline{\phi},\bh),\pi^{\lalg}_{1}(\underline{\phi},\bh)\right)\rightarrow \;\ext^1_{G}&\left(\pi_{\natural}^{\lalg}(\underline{\phi},\bh),\pi^{\sharp}_{1}(\underline{\phi},\bh)_u\right)\\
							&\rightarrow \oplus_{j\in \Delta}  \ext^1_{G}\left(\pi_{\natural}^{\lalg}(\underline{\phi},\bh),C_{j,[u]_j}\right)\rightarrow 0 .
						\end{aligned}
					\end{equation}
					Thus,\;we get an isomorphism
					$\zeta_u:\homo(\bT(\bQ_p),E)\xrightarrow{\sim}\ext^1_G\left(\pi_{\natural}^{\lalg}(\underline{\phi},\bh),\pi^{\sharp}_1(\underline{\phi},\bh)_u\right)$.\;
					\item[(2)] For $?\in \{\flat,\sharp\}$,\;we have the following exact sequence
					\begin{equation*}
						\begin{aligned}
							0\rightarrow \ext^1_{G}\left(\pi_{\natural}^{\lalg}(\underline{\phi},\bh),\pi^{\lalg}_{1}(\underline{\phi},\bh)\right)\rightarrow \;\ext^1_{G}&\left(\pi_{\natural}^{\lalg}(\underline{\phi},\bh),\pi^{?}_1(\underline{\phi},\bh)\right)\\
							&\rightarrow \oplus_{j\in \Delta,u\in \cI^{?}_{j}}  \ext^1_{G}\left(\pi_{\natural}^{\lalg}(\underline{\phi},\bh),C_{j,u}\right)\rightarrow 0 .
						\end{aligned}
					\end{equation*}
					In particular,\; $\dim_E\ext^1_{G}\left(\pi_{\natural}^{\lalg}(\underline{\phi},\bh),\pi^{\sharp}_1(\underline{\phi},\bh)\right)=n+1+\sum_{j=1}^{n-1}|\cI^{?}_{j}|$.\;
				\end{itemize}
			\end{pro}
			
			For each Steinberg $(\varphi,\Gamma)$-module $E_{u,i}$,\;by \cite{2019DINGSimple},\;we associate to $E_{u,i}$ the Fontaine-Mazur simple $\sL$-invariants $\sL^{\Delta}(E_{u,i}):=\{\sL^{\Delta}(E_{u,i})_j\}_{j\in \Delta_{S_0(u),i}}$ and a locally analytic representation $\pi^{\Delta}_{\ana}(E_{u,i})$ that determines $\sL^{\Delta}(E_{u,i})$,\;where $\sL^{\Delta}(E_{u,i})_j=E(\val_p+\cL^u_{i,j}\log_p)\subseteq\homo(\bQ_p^{\times},E)$ is an $E$-line (recall (\ref{dfnforimagehomou})).\;More precisely,\;	\[\pi^{\Delta}_{\ana}(E_{u,i})=\Big[\pi^{\sharp}_1(\underline{\phi}_{E_{u,i}},\bh_{E_{u,i}})-\oplus_{j\in \Delta_{S_0(u),i}}Q_{\Delta_{S_0(u),i}}(j,\lambda_{E_{u,i}})\Big],\]
			such that the non-split extension $\pi^{\sharp}_1(\underline{\phi}_{E_{u,i}},\bh_{E_{u,i}})-Q_{\Delta_{S_0(u),i}}(j,\lambda_{E_{u,i}})$ encodes exactly the information of $\sL^{\Delta}(E_{u,i})_j$,\;where $\lambda_{E_{u,i}}=\bh_{E_{u,i}}-\theta|_{\bT(\bQ_p)\cap \bL_{S_0(u),i}(\bQ_p)}=\lambda|_{\bT(\bQ_p)\cap \bL_{S_0(u),i}(\bQ_p)}$.\;Consider the locally $\bQ_p$-analytic parabolic induction 
			\begin{equation*}\label{locaparabolicindblock}
				\mathrm{PS}^{\Delta}_{\cF_{u}
				}(\underline{\phi},\bh):=\left(\ind^G_{\op_{S_0(u)}(\bQ_p)}\pi^{\Delta}_{\ana}(\underline{\phi}^u)\eta_{S_0(u)}\right)^{\ana},\;\pi^{\Delta}_{\ana}(\underline{\phi}^u):=\boxtimes_{j=1}^{f_u}\pi^{\Delta}_{\ana}(E_{u,j}).\;
			\end{equation*}
			Let  $\mathrm{ST}^{\Delta}_{\cF_{u}}(\underline{\phi},\bh)$  be the unique maximal quotient of $\mathrm{PS}^{\Delta}_{\cF_{u}}(\underline{\phi},\bh)$ with socle $Q_{\Delta}^{\diamond}(\emptyset,\lambda)$.\;We have $\mathrm{PS}_{u}(\underline{\phi},\bh)\hookrightarrow\mathrm{PS}^{\Delta}_{\cF_{u}}(\underline{\phi},\bh)$ and $\mathrm{ST}_{u}(\underline{\phi},\bh)\hookrightarrow\mathrm{ST}^{\Delta}_{\cF_{u}}(\underline{\phi},\bh)$.\;By taking  ``amalgamated sum" of $\{\mathrm{ST}^{\Delta}_{\cF_{u}}(\underline{\phi},\bh)\}_{u\in \sW_n^{S_0}}$, we obtain a locally analytic representation $\pi_{1}^{\Delta}(\underline{\phi},\bh)$ of socle $Q_{\Delta}^{\diamond}(\emptyset,\lambda)$ such that
			$\pi^{\sharp}_{1}(\underline{\phi},\bh)\hookrightarrow \pi^{\Delta}_{1}(\underline{\phi},\bh)$.\;We do not study the interior structure of  $\mathrm{ST}^{\Delta}_{\cF_{u}}(\underline{\phi},\bh)$ (resp.,\;$\pi^{\Delta}_{1}(\underline{\phi},\bh)$) here.\;Note that $\pi^{\Delta}_{1}(\underline{\phi},\bh)$ admits a unique  subrepresentation of the following form :
			\begin{equation}
				\begin{aligned}
					&\pi^{\Delta,-}_{1}(\underline{\phi},\bh):=\pi^{\flat}_{1}(\underline{\phi},\bh)\otimes_{\pi_1^{\lalg}(\underline{\phi},\bh)}\bigoplus^{i\in \Delta,u\in\cI_i^{\sharp}\backslash\cI_i^{\flat}}_{\pi_1^{\lalg}(\underline{\phi},\bh)}[\pi_1^{\lalg}(\underline{\phi},\bh)-C_{i,u}-Q(I_u^-,\lambda)].\;
				\end{aligned}
			\end{equation}
			Note that the existence of each representation $[\pi_1^{\lalg}(\underline{\phi},\bh)-C_{i,u}-Q(I_u^-,\lambda)]$ is a direct consequence of Proposition \ref{reinterforextgp} and the proof of  Lemma  \ref{parameterlineextgp}),\;and the corresponding extension classes in $\Big\{\ext^1(Q(I_u^-,\lambda),[\pi_1^{\lalg}(\underline{\phi},\bh)-C_{i,u}])\Big\}_{i\in \Delta,u\in\cI_i^{\sharp}\backslash\cI_i^{\flat}}$
		already carry exactly the information of 
			\[\Big\{\sL^{\Delta}(E_{u,i})\Big\}_{u\in \sW_n^{S_0},1\leq i\leq f_u}\]
			 or equivalently $p^{\Delta}_{\mathrm{ST}}([\underline{\sL}(\Dpik)])$.\;
			
			\begin{rmk}\label{rmkforextensionsquare}
				Using the discussion in  \cite{HigherLinvariantsGL3(Qp)} and the $\ext^{\bullet}$ groups in \cite{BQ24},\;for any $u\in \cI_i^{\sharp}\backslash\cI^{\flat}_i$,\;we can show that $\mathrm{ST}^{\Delta}_{\cF_{u}}(\underline{\phi},\bh)/(\mathrm{ST}^{\Delta}_{\cF_{u}}(\underline{\phi},\bh))^{\lalg}$ admits a subrepresentation  $C^{\Delta}_{i,u}\supseteq[C_{i,u}-Q(I_u^-,\lambda)]$ (the so-called extension square of locally analytic representations) of the form:
				\begin{equation}
					\xymatrix{
						& C_{s_{i-1}s_i,\ast} \ar@{-}[dl] \ar@{-}[dr]  &    \\
						C_{i,u} \ar@{-}[dr] \ar@{-}[r]  & Q^{\diamond}_{\Delta}(I_u^-,I_{[s_iu]_i}^+,\lambda) \ar@{-}[r] & C_{i,[s_iu]_i} \ar@{-}[dl]\\ 
						&  C_{s_{i+1}s_i,\ast'}  & }
				\end{equation}
				where $C_{s_{i-1}s_i,\ast}=\cF^{G}_{\op_{\widehat{i-1}}}\Big(\overline{L}(-s_{i-1}s_i\cdot{\lambda}),\ast\Big)$ (resp.,\;$C_{s_{i+1}s_i,\ast'}=\cF^{G}_{\op_{\widehat{i+1}}}\Big(\overline{L}(-s_{i+1}s_i\cdot{\lambda}),\ast'\Big)$)  for some irreducible non-generic smooth representation $\ast$ (resp.,\;$\ast'$) of $\bL_{\widehat{i-1}}(\bQ_p)$ (resp.,\;$\bL_{\widehat{i+1}}(\bQ_p)$).\;More precisely,\;let $\pi^{\infty}_{i,u}$ be the unique generic irreducible constituent of $\mathrm{PS}^{\infty}_{i,u}(\underline{\phi})$.\;Recall the locally analytic representations in (\ref{exampleforGl2}) (apply to the  Steinberg  $(\varphi,\Gamma)$-module $R_{u,i}^{i+1}$ of rank-$2$):
				\[\pi_{\min}^-(\underline{\phi}_{R_{u,i}^{i+1}},\bh_{R_{u,i}^{i+1}})\hookrightarrow \pi_{\min}(\underline{\phi}_{R_{u,i}^{i+1}},\bh_{R_{u,i}^{i+1}}).\]
				For $\ast\in\{-,\emptyset\}$,\;let $\pi^{\ast}_{i,u}(\underline{\phi},\bh):=\Big(\pi^{\infty}_{i,u}\otimes_EL_{i}(\lambda)\Big)\times_{\pi_1^{\lalg}(\underline{\phi}_{R_{u,i}^{i+1}},\bh_{R_{u,i}^{i+1}})}\pi^{\ast}_{\min}(\underline{\phi}_{R_{u,i}^{i+1}},\bh_{R_{u,i}^{i+1}})$.\;Put
				\begin{equation}
					R^{\ast}_{i,u}:=\left(\ind^G_{\op_{i}(\bQ_p)}\pi^{\ast}_{i,u}(\underline{\phi},\bh)\eta_{i}\right)^{\ana}.\;
				\end{equation}
                Note that $R_{i,u}/R^{-}_{i,u}\cong \cF^G_{\ob}\big(\overline{M}(-s_i\cdot\lambda),\unr(\underline{\phi}^{s_iu})\eta\big)$,\;and $C_{i,[s_iu]_i}$ appears as an irreducible consisitent of the latter term.\;We see that $\dim_E\ext^1_G\big(C^{\Delta}_{i,u},\pi^{\lalg}_1(\underline{\phi},\bh)\big)=2$.\;
                \end{rmk}


%

%

%
%

			\subsection{Main results}\label{mainthmI}

			This section follows the route in the recent work \cite[Section 2.4]{BDcritical25} of Breuil-Ding.\;See Section \ref{reinterfor33} for the spaces $\ext^{1,\flat}_{\gr}({\Dpik},{\Dpik})$ and $\homo_{\fil}(D,D)$.\;Recall that we have fixed a basis $\{e_{1,j}\}_{1\leq j\leq r_1},\cdots,\{e_{s,j}\}_{1\leq j\leq r_s}$ (and we relabel them with $e_{1},\cdots,e_{n}$ by using the lexicographical ordering) of $D$.\;
			

			For $J\subseteq \{1,\cdots,n\}$,\;put
			\[e_J:=\wedge_{j\in J}e_{j}\in \bigwedge_E\nolimits^{\!|J|}\!D.\]
			For $i\in\Delta$,\;recall $\cI^{\flat}_{i}\subseteq \cI^{\sharp}_{i}$.\;In the sequel,\;$\star\in\{\sharp,\flat\}$.\;For $u\in \cI^{\star}_{i}$,\;let $I_i(u)=\{u^{-1}(1),\cdots,u^{-1}(i)\}$  and $I_i(u)^c=\{1,\cdots,n\}\backslash I_i(u)$.\;For each $u\in \cI^{\star}_{i}$,\;we fix an
			isomorphism of $1$-dimensional $E$-vector spaces (as in \cite[(43)]{BDcritical25}):
			\begin{equation}
				\epsilon_{i,u}:\ext^1_{G}(C_{i,u},\pi_1^{\lalg}(\underline{\phi},\bh))\xrightarrow{\sim}Ee_{I_i(u)^c}\in \bigwedge\nolimits_E^{\!n-i}\!D.
			\end{equation}
			Then we get for each $i\in\Delta$ a map:
			\begin{equation}\label{dfnforeplisonistar}
				\begin{aligned}
					&\epsilon^{\star}_{i}:=\bigoplus_{u\in \cI^{\star}_{i}}\epsilon_{i,u}:\ext^1_{G}\Big(\bigoplus_{u\in \cI^{\star}_{i}}C_{i,u},\pi_1^{\lalg}(\underline{\phi},\bh)\Big)\rightarrow\bigwedge\nolimits_E^{\!n-i}\!D.
				\end{aligned}
			\end{equation}
			Let $\big(\bigwedge_{E}^{\!n-i}\!D_{\sigma}\big)^{\star}$ be the image of $\epsilon_{i}$.\;Let
			\[\big(\bigwedge\nolimits_{E}^{\!n-i}\!D_{\sigma}\big)^{\star,c}=E\big\langle e_J:\;J\subseteq \{1,\cdots,n\},|J|=n-i,e_J\neq \be_{I_i(u)^c},\;\forall u\in \cI_{i}^{\star}\big\rangle.\]
			Then $\bigwedge_E^{n-i}D_{\sigma}=\big(\bigwedge\nolimits_{E}^{\!n-i}\!D_{\sigma}\big)^{\star}\oplus\big(\bigwedge\nolimits_{E}^{\!n-i}\!D_{\sigma}\big)^{\star,c}$ and  $\dim_E\big(\bigwedge_{E}^{\!n-i}\!D_{\sigma}\big)^{\star}=|\cI_{i}^{\star}|$.\;
			\begin{rmk}
				Since $\bigoplus_{u\in \cI^{\star}_{i}}C_{i,u}$ is a subrepresentation of $\bigoplus_{u\in \sW_n^{\widehat{i}}}C_{i,u}$,\;thus the left hand of (\ref{dfnforeplisonistar}) should be viewed as a quotient of $\ext^1_{G}\big(\bigoplus_{u\in \sW_n^{\widehat{i}}}C_{i,u},\pi_1^{\lalg}(\underline{\phi},\bh)\big)$ (the left hand of \cite[(44)]{BQ24}).\;In the sequel,\;we identity $\big(\bigwedge_{E}^{\!n-i}\!D\big)^{\star}$ with $\bigwedge_{E}^{\!n-i}\!D/\big(\bigwedge_{E}^{\!n-i}\!D\big)^{\star,c}$.\;Therefore,\;$\homo_E\Big(\big(\bigwedge\nolimits_{E}^{\!n-i}\!D\big)^{\star},\fil_i^{\max}(D)^{\star}\Big)$ is a subspace of $\homo_E\Big(\bigwedge\nolimits_{E}^{\!n-i}\!D,\fil_i^{\max}(D)\Big)$ of vectors that vanish on $\big(\bigwedge\nolimits_{E}^{\!n-i}\!D\big)^{\star,c}$.\;
			\end{rmk}

			For $\star\in\{\sharp,\flat\}$,\;sending $e_J\not\in \{e_{I_i(u)^c}:u\in \cI^{\star}_i\}$ to zero induces an isomorphism
			\begin{equation}
				\pr^{\star}:\fil_i^{\max}(D)\xrightarrow{\sim}\fil_i^{\max}(D)^{\star},
			\end{equation}
			i.e.,\;$\fil_i^{\max}(D)^{\star}$ is the image of $\fil_i^{\max}(D)\hookrightarrow \bigwedge_E^{n-i}D\twoheadrightarrow \bigwedge_E^{n-i}D/\big(\bigwedge_{E}^{\!n-i}\!D\big)^{\star
				,c}\cong \big(\bigwedge_{E}^{\!n-i}\!D\big)^{\star}$.\;For $1\leq i\leq n$,\;we define the following $E$-line of $\bigwedge_E^{n-i}D$:
			\begin{equation}
				\begin{aligned}
					\fil_i^{\max}(D)&:= \fil^{-\bh_{n}}(D)\wedge \fil^{-\bh_{n-1}}(D)\wedge\cdots\wedge \fil^{-\bh_{i+1}}(D)\xrightarrow{\sim}\bigwedge\nolimits_E^{\!n-i}\!\fil^{-\bh_{i+1}}(D)\subseteq\bigwedge\nolimits_E^{\!n-i}\!D .\;
				\end{aligned}
			\end{equation}
			The non-critical assumption implies that the coefficient of $e_{I_i(u)^c}$ (for each $u\in \cI^{\star}_i $) in 	$\fil_i^{\max}(D)$ is non-zero.\;Similar to \cite[(46)]{BDcritical25},\;we consider the morphisms of $E$-vector spaces:
			\begin{equation}\label{compositiondfnpisi}
				\begin{aligned}
					&\ext^1_{G}\Big(\bigoplus_{u\in \cI_i^{\star}}C_{i,u},\pi_1^{\lalg}(\underline{\phi},\bh)\Big)\otimes \ext^1_{G}\Big(\bigoplus_{u\in \cI_i^{\star}}C_{i,u},\pi_1^{\lalg}(\underline{\phi},\bh)\Big)^{\vee}\\
					\xrightarrow{\sim}&\;\ext^1_{G}\Big(\Big(\bigoplus_{u\in \cI_i^{\star}}C_{i,u}\Big)\otimes\ext^1_{G}\Big(\bigoplus_{u\in \cI_i^{\star}}C_{i,u},\pi_1^{\lalg}(\underline{\phi},\bh)\Big) ,\pi_1^{\lalg}(\underline{\phi},\bh)\Big)\\
					\xrightarrow{\sim}&\;\ext^1_{G}\Big(\Big(\bigoplus_{u\in \cI_i^{\star}}C_{i,u}\Big)\otimes_E\fil_i^{\max}(D)^{\flat},\pi_1^{\lalg}(\underline{\phi},\bh)\Big),
				\end{aligned}
			\end{equation}
			where the second morphism is the push-forward induced by the composition
			\[\fil_i^{\max}(D)^{\star}\hookrightarrow \big(\bigwedge\nolimits_{E}^{\!n-i}\!D\big)^{\star}\xrightarrow{\epsilon_i^{-1}}\ext^1_{G}\Big(\bigoplus_{u\in \cI_i^{\star}}C_{i,u},\pi_1^{\lalg}(\underline{\phi},\bh)\Big).\] 
			Let $\pi^{\star}_{s_i}(\underline{\phi},\bh)$  be a representative of the image for the canonical vector of the first term of (\ref{compositiondfnpisi}) by the composition (\ref{compositiondfnpisi}),\;which lies in the extension group 
			\[\ext^1_{G}\Big(\Big(\bigoplus_{u\in \cI^{\star}_{i}}C_{i,u}\Big)\otimes_E\fil_i^{\max}(D)^{\star},\pi_1^{\lalg}(\underline{\phi},\bh)\Big).\]
			In a similar way,\;we also have a representative $\pi_{s_i,u}(\underline{\phi},\bh)$ that lies in $\ext^1_{G}\Big(C_{i,u}\otimes_E\fil_i^{\max}(D)^{\star},\pi_1^{\lalg}(\underline{\phi},\bh)\Big)$.\;We have injections of $G$-representations $\pi_1^{\lalg}(\underline{\phi},\bh)\hookrightarrow \pi_{s_i,u}(\underline{\phi},\bh) \hookrightarrow \pi^{\star}_{s_i}(\underline{\phi},\bh)$ and a canonical isomorphism
			\[\bigoplus^{u\in \cI^{\star}_i}_{\pi_1^{\lalg}(\underline{\phi},\bh)}\pi_{s_i,u}(\underline{\phi},\bh)\xrightarrow{\sim }\pi^{\star}_{s_i}(\underline{\phi},\bh).\;\]
			For $S\subseteq \Delta$,\;put
			\[\pi^{\star}_{S}(\underline{\phi},\bh):=\bigoplus_{i\in S}^{\pi_{1}^{\lalg}(\underline{\phi},\bh)}\pi^{\star}_{s_i}(\underline{\phi},\bh).\]
			Recall the composition of the surjections in \cite[(68) \& (69)]{BDcritical25} and the (inverse of) the isomorphism \cite[(69)]{BDcritical25}:
			\begin{equation}
				\begin{aligned}
					g_{i}:&\;\homo_E\Big(\bigwedge\nolimits_E^{\!n-i}\!D,\fil_i^{\max}(D)\Big)=\homo_E\Big(\bigwedge\nolimits_E^{\!n-i}\!D,\bigwedge\nolimits_E^{\!n-i}\!\fil_H^{-\bh_{i+1}}(D)\Big)\\&\twoheadrightarrow \homo_E\Big(\bigwedge\nolimits_E^{\!n-i-1}\!\fil_H^{-\bh_{i+1}}(D)\wedge D ,\bigwedge\nolimits_E^{\!n-i}\!\fil_H^{-\bh_{i+1}}(D)\Big)\\
					&\xrightarrow{\sim}\left\{f\in\homo_E\big(D,\fil_H^{-\bh_{i+1}}(D)\big):f|_{\fil_H^{-\bh_{i+1}}(D)} \text{\;scalar}\right\}\hookrightarrow \homo_{\fil}\big(D,D\big),
				\end{aligned}
			\end{equation}
			the third inclusion is given by \cite[Lemma 2.2.5 (ii)]{BDcritical25}.\;By \cite[Lemma 2.2.5 (ii)]{BDcritical25},\;we have a surjection
			\[\bigoplus_{i=1}^n\left\{f\in\homo_E\big(D,\fil_H^{-\bh_{i}}(D)\big):f|_{\fil_H^{-\bh_{i}}(D)} \text{\;scalar}\right\}\twoheadrightarrow \homo_{\fil}(D,D).\]

			\begin{thm}\label{construext1}
				For $\star\in\{\sharp,\flat\}$,\;put
				\begin{equation}
					\homo^{\star}_{\fil}(D,D):=\sum_{i\in \Delta}g_{i}\Big(\homo_E\Big(\big(\bigwedge\nolimits_{E}^{\!n-i}\!D\big)^{\star},\fil_i^{\max}(D)^{\star}\Big)\Big)\hookrightarrow \homo_{\fil}(D,D).
				\end{equation}
				There  is a surjection of finite dimensional $E$-vector spaces which only depends on the $(\epsilon_{i,u})_{i\in \Delta,u\in \cI^{\star}_i}$ and  on a choice of $\log_p(p)\in E$:
				\[t^{\star}_{D}:\ext^1_{G}\left(\pi_{\natural}^{\lalg}(\underline{\phi},\bh),\pi^{\star}_{1}(\underline{\phi},\bh)\right)\twoheadrightarrow \ext^{1,\flat}_{\gr}({\Dpik},{\Dpik})\oplus\homo^{\star}_{\fil}(D,D).\]
			\end{thm}

			\begin{proof}We define $t^{\star}_{D}$ in restriction to $\ext^1\Big(\pi_{\natural}^{\lalg}(\underline{\phi},\bh),\pi^{\lalg}(\underline{\phi},\bh))$ firstly.\;This is essentially the Step 2 in the proof of \cite[Proposition 2.2.4]{BDcritical25}.\;It suffices to use \cite[(63)]{BDcritical25} and replace \cite[(64) \& (65)]{BDcritical25} with 
				\begin{equation}
					\begin{aligned}
						t^{\star}_{D}|_{\homo_{\mathrm{sm}}(\bT(\bQ_p),E)}&\xrightarrow{\sim }\ext^{1,\flat}_{\gr}({\Dpik},{\Dpik})= \prod_{i=1}^s\ext^{1}_g(E_i[1/t],E_i[1/t])\\
						(\psi_1,\cdots,\psi_s)&\mapsto (E_{s-i}[1/t]\otimes_{\cR_{E,L}}\cR_{E[\epsilon]/\epsilon^2,L}(1+\psi_{s-i}\epsilon)_{1\leq i\leq s}.
					\end{aligned}
				\end{equation}
				The argument Step 3 in \cite[Proposition 2.2.4]{BDcritical25} need more argument.\;Choose $i\in \Delta$.\;We have isomorphisms 
				\begin{equation}
					\begin{aligned}
						\ext^1_{G}\Big(\pi_{\natural}^{\lalg}(\underline{\phi},\bh),\pi^{\star}_{s_i}(D)/\pi^{\lalg}(\underline{\phi},\bh)\Big)&\xrightarrow{\sim}\ext^1_{G}\Big(\pi_{\natural}^{\lalg}(\underline{\phi},\bh),\Big(\bigoplus_{u\in \cI^{\star}_{i}}C_{i,u}\Big)\otimes_E\fil_i^{\max}(D)^{\star}\Big)\\
						&\xleftarrow{\sim}\ext^1_{G}\Big(\pi_{\natural}^{\lalg}(\underline{\phi},\bh),\bigoplus_{u\in\cI^{\star}_{i}}C_{i,u}\Big)\otimes_E\fil_i^{\max}(D)^{\star}\\
						&\xrightarrow[(\ref{dualpairing})]{\sim}\ext^1_{G}\Big(\bigoplus_{u\in \cI^{\star}_{i}}C_{i,u},\pi^{\lalg}(\underline{\phi},\bh)\Big)^{\vee}\otimes_E\fil_i^{\max}(D)^{\star}\\
						&\xleftarrow{\sim}\homo_E\Big(\big(\bigwedge\nolimits_{E}^{\!n-i}\!D\big)^{\star},\fil_i^{\max}(D)^{\star}\Big).
					\end{aligned}
				\end{equation}
				By definition,\;we get the desired surjection $t^{\star}_{D}$.\;This completes the proof.\;
			\end{proof}
			
			We next describe the image $\homo^{\star}_{\fil}(D,D)$ explicity.\;For $i\in \Delta$ and $u\in \cI_i^{\star}$,\;we  define $\homo_{\fil,\bF_u}^i(D,D)=$
			\begin{equation}\label{dfnforhomofilmaxilevi}
				\begin{aligned}
					\left\{f\in \homo_{\fil,\bF_u}(D,D):\exists\;a,b\in E\; \text{such that\;} f(e_{u^{-1}(j)})=ae_{u^{-1}(j)},\right.\\ \left.\forall\;1\leq j\leq i,f(e_{u^{-1}(j)})-be_{u^{-1}(j)}\in \oplus_{l=1}^{i}Ee_{u^{-1}(l)}, \forall\;i+1\leq j\leq n\right\}.
				\end{aligned}
			\end{equation}
			Keep the notation in Section \ref{reinterfor33}.\;Then we have 
			\begin{equation}
				\begin{aligned}
					&\homo^i_{\fil,\bF_u}(D,D)\cong \mathrm{Ad}_{u^{-1}}(\tau_{\widehat{i}})\cap \mathrm{Ad}_{g}(\fb).\;
				\end{aligned}
			\end{equation}
			As in \cite[(134),(137)]{BDcritical25},\;sending $f$ to $(a\log_{p},b\log_{p})$ ($a,b$ are given in the definition (\ref{dfnforhomofilmaxilevi})) induces a canonical map 
			\[f_{i}:\homo^i_{\fil,\bF_u}(D,D)\rightarrow \homo(\bL_{\widehat{i}}(\bZ_p),E),\]
			which is an isomorphism by comparing dimensions.\;Indeed,\;by Remark \ref{explainHomasLiealg},\;we have $\mathrm{Ad}_{u^{-1}}(\tau_{\widehat{i}})\cap \mathrm{Ad}_{g}(\fb)=\mathrm{Ad}_{ub(u)}(\fz_{\widehat{i}})$ and thus thus $\dim_E\homo^i_{\fil,\bF_u}(D,D)=2$.\;
			Consider the following composition of isomorphisms (similar to \cite[(142)-(144)]{BDcritical25}):
			\begin{equation}\label{tDmapforiusigma}
				\begin{aligned}
					\ext^1_{G}\left(\pi_{\natural}^{\lalg}(\underline{\phi},\bh),V
					_{i,u}\right)\xleftarrow{\sim}&\;\homo_{\sm}(\bT(\bQ_p),E)\oplus_{\homo_{\sm}(\bL_{\widehat{i}}(\bQ_p),E)}\homo(\bL_{\widehat{i}}(\bQ_p),E)\\\xleftarrow{\sim}&\;\homo_{\sm}(\bT(\bQ_p),E)\oplus\homo(\bL_{\widehat{i}}(\bZ_p),E)\\
					\xrightarrow[f_{i}^{-1}]{\sim}&\;\ext^{1,\flat}_{\gr}({\Dpik},{\Dpik})\oplus\homo^i_{\fil,\bF_u}(D,D)\\
					\hookrightarrow&\; \ext^{1,\flat}_{\gr}({\Dpik},{\Dpik})\oplus\homo_{\fil}(D,D).\;
				\end{aligned}
			\end{equation}
			
				\begin{pro}\label{imageofhomostar}
				We have $\homo^{\flat}_{\fil}(D,D)\hookrightarrow \homo^{\sharp}_{\fil}(D,D)\hookrightarrow \homo_{\fil}(D,D)$,\;which have the forms 
				\begin{equation}
					\begin{aligned}
						&\homo^{\flat}_{\fil}(D,D)\cong\sum\nolimits_{u\in \sW_n^{S_0}}\mathrm{Ad}_{u^{-1}}(\tau_{S_0(u)})\cap \mathrm{Ad}_{g}(\fb),\\
						&\homo^{\sharp}_{\fil}(D,D)\cong\sum\nolimits_{u\in \sW_n^{S_0}}\mathrm{Ad}_{u^{-1}}(\fb)\cap \mathrm{Ad}_{g}(\fb).
					\end{aligned}
				\end{equation}
			\end{pro}
				\begin{proof}
				Let $i\in \Delta$ and $u\in \cI_i^{\ast}$.\;By \cite[Proposition 2.5.4]{BDcritical25},\;the composition (given in (\ref{tDmapforiusigma}))
				\begin{equation}
					\begin{aligned}
						\homo_{\sm}(\bT(\bQ_p),E)\;\oplus&\;\homo(\bL_{\widehat{i}}(\bZ_p),E)\xrightarrow[f_{i}^{-1}]{\sim}\ext^{1,\flat}_{\gr}({\Dpik},{\Dpik})\oplus\homo_{\fil,\bF_u}(D,D)
					\end{aligned}
				\end{equation}
				coincides with the composition
				\begin{equation}
					\begin{aligned}
						\homo_{\sm}(\bT(\bQ_p),E)\oplus\homo(\bL_{\widehat{i}}(\bZ_p),E)\xrightarrow{\sim}&\;\ext^1_{G}\left(\pi_{\natural}^{\lalg}(\underline{\phi},\bh),V_{i,u}\right)\\&\xrightarrow{t^{\star}_{D}}\ext^{1,\flat}_{\gr}({\Dpik},{\Dpik})\oplus\homo^{\star}_{\fil}(D,D).\;
					\end{aligned}
				\end{equation}
				Thus $\homo^{\star}_{\fil}(D,D)$ equals  to the envelope $\sum_{i\in\Delta,u\in \cI^{\star}_i}\mathrm{Ad}_{u^{-1}}(\tau_{\widehat{i}})\cap \mathrm{Ad}_{g}(\fb)$ in $\mathrm{Ad}_{g}(\fb)$.\;Then
				\begin{itemize}
					\item[(i)] $\star=\flat$.\;The envelope equals to 
					$\sum_{u\in \sW_n^{S_0}}\sum_{j\in \Delta\backslash S_0(u)}\mathrm{Ad}_{u^{-1}}(\tau_{\widehat{j}})\cap \mathrm{Ad}_{g}(\fb)$.\;Note that $\mathrm{Ad}_{u^{-1}}(\tau_{\widehat{j}})\cap \mathrm{Ad}_{g}(\fb)=\mathrm{Ad}_{ub(u)}(\fz_{\widehat{j}})$
					for some $b(u)\in \bB(E)$.\;Therefore,\;we deduce that 
					\begin{equation*}
						\begin{aligned}
							\sum_{j\in \Delta\backslash S_0(u)}\mathrm{Ad}_{u^{-1}}(\tau_{\widehat{j}})\cap \mathrm{Ad}_{g}(\fb)=\sum_{j\in \Delta\backslash S_0(u)}\mathrm{Ad}_{ub(u)}(\fz_{\widehat{j}})=\mathrm{Ad}_{ub(u)}(\fz_{S_0(u)})=\mathrm{Ad}_{u^{-1}}(\tau_{S_0(u)})\cap \mathrm{Ad}_{g}(\fb).\;
						\end{aligned}
					\end{equation*}
					\item[(ii)] $\star=\sharp$.\;We see that $\sum_{j\in \Delta}\mathrm{Ad}_{u^{-1}}(\tau_{\widehat{j}})\cap \mathrm{Ad}_{g}(\fb)=\mathrm{Ad}_{u^{-1}}(\fb)\cap \mathrm{Ad}_{g}(\fb)$.\;
				\end{itemize}
				This completes the proof.\;
			\end{proof}
			\begin{rmk}
			\begin{itemize}
				\item[(1)]Similar to \cite[Proposition 2.2.7]{BDcritical25},\;$t^{\star}_{D}$ does not depends on the choices of $(\epsilon_{i,u})_{i\in \Delta,u\in \cI^{\star}_i}$.\;
				\item[(2)]By Theorem \ref{construext1},\;we get the dimension of  $\ker(t^{\star}_{D})$.\;Moreover,\;$\ker(t^{\flat}_{D})$ and $\ker(t^{\sharp}_{D})$ determines  some  ``crystalline" Hodge parameters between the Steinberg blocks in $\cC_{\mathrm{ST}}(\Dpik)$.\;
			\end{itemize}
				
			\end{rmk}
			
			
			For $S\subseteq \Delta$,\;put
			\[\ker(t^{\star}_{D})_S:=\ker(t^{\star}_{D}|_{\ext^1_{G,\mathrm{inf}}(\pi_{\natural}^{\lalg}(\underline{\phi},\bh),\pi^{\star}_{S}(\underline{\phi},\bh))}),\quad \ker(t^{\star}_{D}):=\ker(t^{\star}_{D})_{\Delta}\]
			for simplicity.\;Note that $\ker(t^{\star}_{D})_{S_1}\subseteq \ker(t^{\star}_{D})_{S_2}\subseteq \ker(t^{\star}_{D})$ if $S_1\subseteq S_2$.\;The remainder this section describe what information of $p$-adic Hodge parameters is determined by $\ker(t^{\star}_{D})_S$.\;
			
			For $i\in \Delta$,\;define $D^{\star,(i)}:=E\langle e_t:t\notin  I_{i}(u)^c,\forall u\in \cI^{\star}_i\rangle$ and $D^{\star}_{(i)}=E\langle e_{t}:t\in I_{i}(u)^c\text{\;for some\;}u\in \cI^{\star}_i\rangle\subseteq D$.\;Thus $D=D^{\star,(i)}\oplus D^{\star}_{(i)}$.\;Let  $\bd^{\star}_{i}:=\dim_ED^{\star,(i)}\leq i$.\;By definition,\;we see that
			\begin{itemize}
				\item[(1)] For $1\leq i\leq n$,\;$\bd^{\sharp}_{i}\leq \bd^{\flat}_{i}$.\;
				\item[(2)] $D^{\sharp}_{(n-1)}\subseteq D^{\sharp}_{(n-2)}\subseteq\cdots \subseteq D^{\sharp}_{(1)}=D$ and  $0=\bd^{\sharp}_{1}\leq \bd^{\sharp}_{2}\leq \cdots\leq \bd^{\sharp}_{n}$.\;
			\end{itemize}
			Similar to \cite[Theorem 4.20]{HEparaforsemitable},\;we obtain			
			\begin{thm}\label{Hodgeparadeter}$\ker(t^{\star}_{D})_S$ determines (recall that $S:=\{i_{1}<i_{2}<\cdots<i_{{|S|}}\}$)
				\begin{equation}
					\begin{aligned}
						\Big\{\fil_H^{-\bh_{{\max\{i_{{j}},\bd^{\star}_{i_{t}}\}}+1}}(D)\oplus D/\fil_H^{-\bh_{\bd^{\star}_{i_{t}}}}(D)\Big\}_{1\leq j\leq |S|}.
					\end{aligned}
				\end{equation}	
				In particular,\;if $\bd^{\star}_{i_{t}}=0$,\;then $\ker(t^{\star}_{D})_S$ determines $\{\fil_H^{-\bh_{i_{{j}}+1}}(D)\}_{1\leq j\leq |S|}$.\;
			\end{thm}
			\begin{cor} \label{Hodgeparadetercor}For $1\leq t\leq n-1$,\;$\ker(t^{\star}_{D})$ determines
				\begin{equation}
					\begin{aligned}
						\Big\{\fil_H^{-\bh_{{\max\{j,\bd^{\star}_{t}\}}+1}}(D)\oplus D/\fil_H^{-\bh_{\bd^{\star}_{t}}}(D)\Big\}_{1\leq j\leq t}.
					\end{aligned}
				\end{equation}	
				In particular,\;if $\bd^{\star}_{t}=0$,\;then $\ker(t^{\star}_{D})$ determines $\{\fil_H^{-\bh_{j+1}}(D)\}_{1\leq j\leq t}$.\;
			\end{cor}

		\begin{exmp}
			A typical example is $\bL_{S_0}\cong \GL_{2}\times \GLN_2$ and $S_0=\{1,3\}$.\;As $E$-vector spaces,\;$D=D_1\oplus D_2$ with $\dim_ED_1=2$  and $\dim_ED_2=2$.\;In this case,\;
			\[\sW_4^{S_0}=\{1,s_2,s_2s_1,s_2s_3,s_2s_1s_3,s_2s_1s_3s_2\}.\]
			Then $\cI_1^{\sharp}=\{1,s_2s_1\}$,\;$\cI_3^{\sharp}=\{1,s_2s_3\}$ and $\cI_2^{\sharp}=\{1,s_2,s_2s_1s_3s_2\}$.\;We have $S_0(1)=\{1,3\}$, $S_0(s_2)=\emptyset$,\;$S_0(s_2s_1)=\{2\}$,\;$S_0(s_2s_3)=\{2\}$,\;$S_0(s_2s_1s_3)=\emptyset$ and $S_0(s_2s_1s_3s_2)=\{1,3\}$.\;Then we have $\cI_1^{\flat}=\{s_2s_1\}$,\;$\cI_3^{\flat}=\{s_2s_3\}$ and $\cI_2^{\flat}=\{1,s_2,s_2s_1s_3s_2\}$.\;Therefore,\;
			\begin{equation*}
				D^{\flat}_{(i)}=\left\{
				\begin{array}{ll}
					E\langle e_1,e_2,e_4\rangle ,\;&i=1\\
					D,\;& i=2\\
					E\langle e_2\rangle      & i=3
				\end{array}
				\right.,\;D^{\sharp}_{(i)}=\left\{
				\begin{array}{ll}
					D,\;&i=1\\
					D,\;& i=2\\
					E\langle e_2,e_4\rangle      & i=3.
				\end{array}
				\right.
			\end{equation*}
			\end{exmp}


		Consider the following composition
		\begin{equation}\label{dfnforgammaD}
			\begin{aligned}
				\gamma_{\Dpik}:\bigoplus_{u\in\sW_n^{S_0}}\overline{\ext}^{1}_{u}(\Dpik,\Dpik)
				\xrightarrow{\oplus_{u}\kappa_{u}} \bigoplus_{u\in\sW_n^{S_0}}\homo_u(\bT(\bQ_p),E)\xrightarrow{\oplus_{u}\zeta_{u}} \ext^1_{G}\Big(\pi_{\natural}^{\lalg}(\underline{\phi},\bh),\pi^{\sharp}_{1}(\underline{\phi},\bh)\Big).\;
			\end{aligned}
		\end{equation}
		The same strategy as in   \cite[Proposition 2.5.5,\;Corollary 2.5.6]{BDcritical25} and proof of  Propositions \ref{imageofhomostar} gives:
		\begin{thm}\label{descibeimagetD}
			Under the splitting in Proposition \ref{splitingforext1gps},\;for any $u\in \sW_n^{S_0}$,\;
			\[t^{\sharp}_{D}\circ\gamma_{\Dpik}|_{\overline{\ext}^{1}_{u}(\Dpik,\Dpik)}=f_{\Dpik}\circ\iota_u,\;\quad t^{\flat}_{D}\circ\gamma_{\Dpik}|_{\overline{\ext}^{1,\flat}_{S_0(u)}(\Dpik,\Dpik)}=f_{\Dpik}\circ\iota_u|_{\overline{\ext}^{1,\flat}_{S_0(u)}(\Dpik,\Dpik)},\]
			where $\iota_u:\overline{\ext}^{1}_{u}(\Dpik,\Dpik)\hookrightarrow \overline{\ext}^{1}(\Dpik,\Dpik)$ is the natural inclusion.\;Consider the map
			\begin{equation}\label{natmorphismforGDpikfernoverline}
				\overline{g}_{\Dpik}:\bigoplus_{u\in \sW_n^{S_0}}\overline{\ext}^{1}_{u}(\Dpik,\Dpik)\rightarrow \overline{\ext}^1(\Dpik,\Dpik).\;
			\end{equation}
			Then $f_{\Dpik}\circ \overline{g}_{\Dpik}=t^{\sharp}_{D}\circ\gamma_{\Dpik}$ and $f_{\Dpik}\circ \overline{g}^{\flat}_{\Dpik}=t^{\flat}_{D}\circ\gamma^{\flat}_{\Dpik}$,\;where   $\overline{g}_{\Dpik}^{\flat}:=\overline{g}_{\Dpik}|_{\bigoplus_{u\in \sW_n^{S_0}}\overline{\ext}^{1,\flat}_{S_0(u)}(\Dpik,\Dpik)}
			$.\;
			
		\end{thm}

		Let $\pi^{\star}_{\min}(\Dpik)$ be the unique quotient of the tautological extension of $\ker(t^{\star}_{D})\otimes_E\pi_{\natural}^{\lalg}(\underline{\phi},\bh)$ by $\pi^{\star}_{1}(\underline{\phi},\bh)$ with socle $Q_{\Delta}^{\diamond}(\emptyset,\lambda)$.\;Similar to \cite[Corollary 2.2.13]{BDcritical25},\;we can show that  $\pi^{\star}_{\min}(\Dpik)$ does not depends on any choices of $(\epsilon_{i,u})_{i\in \Delta,u\in \cI_i}$ and $\log_p(p)$.\;Finally,\;put
	\[\pi^{\star,\Delta,-}_{\min}(\Dpik):=\pi^{\star}_{\min}(\Dpik)\oplus_{\pi^{\star}_{1}(\underline{\phi},\bh)}\pi^{\Delta,-}_{1}(\underline{\phi},\bh).\]

		\begin{cor}\label{thmforamosttwo}
			$p^{\Delta}_{\mathrm{ST}}([\underline{\sL}(\Dpik)])$ and  $\Dpik^{\mathrm{cr}}$ are determined  by $\pi^{\star,\Delta,-}_{\min}(\Dpik)$.\;In particular,\;if  $\max_{1\leq i\leq s}l_i\leq 2$,\;$\pi^{\star,\Delta,-}_{\min}(\Dpik)$ determines $\Dpik$.\;
		\end{cor}
		\begin{proof}Write $\Dpik=[\Dpik_0-\Dpik^{\mathrm{cr}}]$,\;  $\Dpik^{\mathrm{cr}}$ are determined by $\pi^{\star}_{\min}(\Dpik)$.\;On the other hand,\;for $u\in\sW_n^{S_0}$,\; $p^{\Delta}_{\mathrm{ST},u}([\underline{\sL}(\Dpik)])$ are already determined by  $\pi^{\Delta,-}_{1}(\underline{\phi},\bh)$.\;The last assertion follows from Corollary \ref{captureparameters2}.\;
		\end{proof}

		\subsection{Further comments}\label{inductiveCons}

		This section gives a possible way to construct a locally analytic representation $\pi_{\min}(\Dpik)$ that determines $\Dpik$,\;which gives  another ``evidence" (but still a conjecture) for the main theorems in Section \ref{mainthmforcapturePara},\;and gives a complete version of the discussion in Section \ref{mainthmI}.\;For any Steinberg $(\varphi,\Gamma)$-module $\bM$,\;assume that we have associated an explicit locally analytic representation $\pi_{\min}(\bM)$ that determines $\bM$ (which are  known for the case $\rk(\bM)\leq 3$).\;
		
		For $u\in\sW_n^{S_0}$,\;recall the $\bP_{S_0(u)}$-parabolic filtration $\cF_{S_0(u)}=[E_{u,1}-E_{u,2}-\cdots-E_{u,f_u}]$ (see (\ref{parafils0u})) associated to the triangulation $\cF_{u}$.\;Consider the locally $\bQ_p$-analytic parabolic induction 
		\begin{equation*}\label{locaparabolicindblockfull}
			\mathrm{PS}_{\cF_{u}
			}(\underline{\phi},\bh):=\left(\ind^G_{\bP_{S_0(u)}(\bQ_p)}\pi_{\ana}(\cF_u)\eta_{S_0(u)}\right)^{\ana},\;\pi_{\ana}(\cF_u):=\boxtimes_{j=1}^s\pi_{\min}(E_{u,j}).\;
		\end{equation*}
		Let  $\mathrm{ST}_{\cF_{u}}(\underline{\phi},\bh)$  be the unique maximal quotient of $\mathrm{PS}_{\cF_{u}}(\underline{\phi},\bh)$ with socle $Q_{\Delta}^{\diamond}(\emptyset,\lambda)$.\;By definition,\;we have $\mathrm{PS}_{S_0(u)}(\underline{\phi},\bh)\hookrightarrow\mathrm{PS}_{\cF_{u}}(\underline{\phi},\bh)$ and $\mathrm{ST}_{S_0(u)}(\underline{\phi},\bh)\hookrightarrow\mathrm{ST}_{\cF_{u}}(\underline{\phi},\bh)$.\;It is possible to take  ``amalgamated sum" of $\{\mathrm{ST}_{\cF_{u}}(\underline{\phi},\bh)\}_{u\in \sW_n^{S_0}}$, and obtain a locally analytic representation $\pi_{1}^+(\underline{\phi},\bh)$ such that
		$\pi^{\sharp}_{1}(\underline{\phi},\bh)\hookrightarrow \pi^+_{1}(\underline{\phi},\bh)$.\;

		Let $u\in \cI^{\sharp}_i$.\;Suppose that $i$ is the $j$-th simple root of $\Delta_{S_0(u),l}$ for some $1 \leq l\leq f_u$.\;Using the conjectural $j$-th branch of $\pi_{\ana}(E_{u,i})$ in  \cite[Conjecture 1.1]{breuil2019ext1},\;the locally analytic $\ext^1$-conjecture predicates a locally analytic representation $\pi^+_{i,u}$,\;and a isomorphism of $\binom{r_{u,l}}{j}$-dimensional $E$-vector spaces (as in \cite[(43)]{BDcritical25}):
		\begin{equation}
			\begin{aligned}
				\epsilon_{i,u}:&\;\ext^1_{G}\big(\pi^+_{i,u},\pi_1^{\lalg}(\underline{\phi},\bh)\big)\xrightarrow{\sim}\Big(\bigwedge\nolimits_E^{\!r_{u,l}-j}\!D_{\dr}(E_{u,l})\Big)\wedge \Big(\bigwedge\nolimits_E^{\!n-i-(r_{u,l}-j)}\!D_{\dr}(E^{(l)}_{u})\Big)\subseteq  \bigwedge\nolimits_E^{\!n-i}\!D.
			\end{aligned}
		\end{equation}
		where $E^{(l)}_{u}:=[E_{u,l+1}-E_{u,l+2}-\cdots-E_{u,f_u}]$ is the unique quotient of $\Dpik$.\;Note that  $n-i-(r_{u,l}-j)=\rk(E^{(l)}_{u})$ and  $\bigwedge\nolimits_E^{\!n-i-(r_{u,l}-j)}\!D_{\dr}(E^{(l)}_{u})$ is an $E$-line.\;Then for each $i\in\Delta$,\;we thus  get a surjection:
		\begin{equation}
			\begin{aligned}
				\epsilon_{i}:=\bigoplus_{u\in \cI^{\sharp}_{i}}\epsilon_{i,u}:\ext^1_{G}\Big(\bigoplus_{u\in \cI^{\sharp}_{i}}\pi^+_{i,u},\pi_1^{\lalg}(\underline{\phi},\bh)\Big)\rightarrow\bigwedge\nolimits_E^{\!n-i}\!D.
			\end{aligned}
		\end{equation}
		Finally,\;we can define a surjection of finite dimensional $E$-vector spaces which only depends on the isomorphisms $(\epsilon_{i,u})_{i\in \Delta,u\in \cI^{\flat}_i}$ and  on a choice of $\log_p(p)\in E$:
		\[t^+_{D}:\ext^1_{G}\left(\pi_{\natural}^{\lalg}(\underline{\phi},\bh),\pi^+_{1}(\underline{\phi},\bh)\right)\twoheadrightarrow\ext^{1,\flat}_{\gr}({\Dpik},{\Dpik})\oplus\homo_{\fil}(D,D).\]
		Such map is also suggested by  \cite[(6)]{BDcritical25} for general de Rham $(\varphi,\Gamma)$-module $\Dpik$ over $\cR_{E}$.\;Moreover,\;we can further extend  (\ref{dfnforgammaD}) to
		\begin{equation}
			\begin{aligned}
				\gamma^+_{\Dpik}:&\bigoplus_{u\in\sW_n^{S_0}}\overline{\ext}^{1}_{\cF_{S_0(u)}}(\Dpik,\Dpik)
				\xrightarrow{\oplus_{u\in\sW_n^{S_0}}\kappa_{\cF_{S_0(u)}}} \bigoplus_{u\in\sW_n^{S_0}}\prod_{i=1}^{f_u}\overline{\ext}^1\big(E_{u,i},E_{u,i}\big)\\
				&\xrightarrow{\sim} \bigoplus_{u\in\sW_n^{S_0}}\prod_{i=1}^{f_u}\ext^1_{\GLN_{r_{u,i}}(\bQ_p)}\big(\pi_1^{\lalg}(\underline{\phi}_{E_{u,i}},\bh_{E_{u,i}}),\pi_{\ana}(E_{u,i})\big)\rightarrow \ext^1_{G}\left(\pi_{\natural}^{\lalg}(\underline{\phi},\bh),\pi^+_{1}(\underline{\phi},\bh)\right).\;
			\end{aligned}
		\end{equation}
		Let $u\in\sW_n^{S_0}$ and  $\iota_{S_0(u)}:\overline{\ext}^{1}_{S_0(u)}(\Dpik,\Dpik)\hookrightarrow \overline{\ext}^{1}(\Dpik,\Dpik)$ be the natural inclusion.\;Recall the surjection $g^+_{\Dpik}:\bigoplus_{u\in\sW_n^{S_0}}\overline{\ext}^{1}_{S_0(u)}(\Dpik,\Dpik)\rightarrow \overline{\ext}^{1}(\Dpik,\Dpik)$ in Corollary \ref{surforgD}.\;Suggested by Theorem \ref{capturetype1diamond}),\;we expect
		\begin{expect}\label{descibeimagetDGL2}  $\ker(t^+_{D})$ determines $\fil_{H}^{\bullet}(D)$.\;
		\end{expect}
	 \begin{rmk}Under the splitting in Proposition \ref{splitingforext1gps} (see Corollary \ref{surforgD} for $\overline{g}^+_{\Dpik}$),\;we expect that  $f_{\Dpik}\circ \overline{g}^{+}_{\Dpik}=t^+_{D}\circ\gamma^+_{\Dpik}$.\;
%
%
%
%
	 \end{rmk}
		Let $\pi_{\min}(\Dpik)$ be the unique quotient of the tautological extension of $\ker(t^{+}_{D})\otimes_E\pi_{\natural}^{\lalg}(\underline{\phi},\bh)$ by $\pi^{+}_{1}(\underline{\phi},\bh)$ with socle $\mathrm{ST}^{\lalg}(\underline{\phi},\bh)$.\;It gives a possible way to construct a locally analytic representation which determines $\Dpik$.

			
		Recall the discussion after Proposition \ref{extpi1repn}.\;The remiander of this section gives  a surjection of finite dimensional $E$-vector spaces:
		\[t^{\Delta}_{D}:\ext^1_{G}\left(\pi_{\natural}^{\lalg}(\underline{\phi},\bh),\pi^{\Delta}_{1}(\underline{\phi},\bh)\right)\rightarrow\ext^{1,\flat}_{\gr}({\Dpik},{\Dpik})\oplus\homo^{\Delta}_{\fil}(D,D),\]
		where $\homo^{\Delta}_{\fil}(D,D)\cong \mathrm{Ad}_{g}(\fb)^{\Delta}_{/S_0}$ (see (\ref{dfnforDeltaS0})).\;Recall Remark \ref{rmkforextensionsquare}.\;Fix $V^{\Delta}_{i,u}$ a locally analytic representation of $G$,\;which is isomorphic to a non-split extension of $C^{\Delta}_{i,u}$ by $\pi_1^{\lalg}(\underline{\phi},\bh)$,\;i.e.,\;
		\[V^{\Delta}_{i,u}:=[\pi_1^{\lalg}(\underline{\phi},\bh)-C^{\Delta}_{i,u}],\]
		By (\ref{isoforfullextgp}),\;the representation $V^{\Delta}_{i,u}$ already gives a choice of a non-split extension in $\ext^1_{G}\big(\pi_{\natural}^{\lalg}(\underline{\phi},\bh),V_{i,u}\big)$. Note that $V_{i,u}\hookrightarrow V^{\Delta}_{i,u}$,\;and thus fix an injection $\iota^{\Delta}_{i,u}:\pi_1^{\lalg}(\underline{\phi},\bh)\hookrightarrow V_{i,u}\hookrightarrow V^{\Delta}_{i,u}$.\;We have   $C^{\Delta}_{i,u}\cong V^{\Delta}_{i,u}/\pi_1^{\lalg}(\underline{\phi},\bh)$.\;Recall the discussion in Remark \ref{rmkforextensionsquare} and the discussion in \cite[Lemma 3.42]{HigherLinvariantsGL3(Qp)}.\;Let $\mathrm{ST}^{\ast}_{i,u}$ be the unique quotient of $R^{\ast}_{i,u}$  with socle $Q_{\Delta}^{\Delta}(\emptyset,\lambda)$.\;By Schraen's spectral sequence (see \cite[Corollary 4.5]{schraen2011GL3}),\;there is a bijection
		\begin{equation}\label{pi1vBsecond}
			\begin{aligned}
				\ext^1_{G}\left(\pi_{\natural}^{\lalg}(\underline{\phi},\bh),R_{i,u}\right)\cong \ext^1_{\bL_{i}(\bQ_p)}\left(\hH_0\big(\overline{\bN}_{i}(\bQ_p),\pi_{\natural}^{\lalg}(\underline{\phi},\bh)\big),\pi_{i,u}(\underline{\phi},\bh)\right).\;
			\end{aligned}
		\end{equation}
		In particular,\;we get an inclusion $\ext^1_{\bL_{i}(\bQ_p)}\left(\pi^{\infty}_{i,u}\otimes_EL_{i}(\lambda),\pi_{i,u}(\underline{\phi},\bh)\right)\hookrightarrow \ext^1_{G}\left(\pi_{\natural}^{\lalg}(\underline{\phi},\bh),R_{i,u}\right)$.\;By composing with the push-forward map via the map $R_{i,u}\twoheadrightarrow \mathrm{ST}_{i,u}\hookrightarrow \pi^{\Delta}_{1}(\underline{\phi},\bh)$,\;we actually obtain a map
		\[\zeta_{i,u}:\ext^1_{\bL_{i}(\bQ_p)}\left(\pi^{\infty}_{i,u}\otimes_EL_{i}(\lambda),\pi_{i,u}(\underline{\phi},\bh)\right)\rightarrow \ext^1_{G}\left(\pi_{\natural}^{\lalg}(\underline{\phi},\bh),\pi^{\Delta}_{1}(\underline{\phi},\bh)\right).\]The unique non-split extension in  $\ext^1_{\bL_{i}(\bQ_p)}\left(\pi^{\infty}_{i,u}\otimes_EL_{i}(\lambda),\pi_{i,u}(\underline{\phi},\bh)/\big(\pi^{\infty}_{i,u}\otimes_EL_{i}(\lambda)\big)\right)$ gives a non-split extension of $\pi_{\natural}^{\lalg}(\underline{\phi},\bh)$ by $V^{\Delta}
		_{i,u}$,\;which has the following form:
		\begin{equation}
			\xymatrix{
				& C_{s_{i-1}s_i,\ast} \ar@{-}[dl] \ar@{-}[dr]  &   & \\
				C_{i,u} \ar@{-}[dr] \ar@{-}[r]  & Q^{\Delta}_{\Delta}(I_u^-,I_{[s_iu]_i}^+,\lambda) \ar@{-}[r] & C_{i,[s_iu]_i} \ar@{-}[dl] \ar@{-}[r] & \pi_{\natural}^{\lalg}(\underline{\phi},\bh)\\ 
				&  C_{s_{i+1}s_i,\ast'}  &  &}.
		\end{equation}
		By  Lemma \ref{parameterline},\;we see that $\dim_E\ext^1_{G}\left(\pi_{\natural}^{\lalg}(\underline{\phi},\bh),V^{\Delta}_{i,u}\right)=n+2$.\;For any $u\in \cI^{\sharp}_i$,\;we put $I_i(u)_{[s_i]}^c:=\big(I_i(u)^c\backslash\{u^{-1}(j+1)\}\big)\cup\{u^{-1}(j)\}$ and  fix the following
		isomorphisms of $1$-dimensional $E$-vector spaces:
		\begin{equation}
			\begin{aligned}
				\epsilon_{i,[s_iu]_i}:&\;\ext^1_{G}(C_{i,[s_iu]_i},\pi_1^{\lalg}(\underline{\phi},\bh))\xrightarrow{\sim}Ee_{I_i(u)_{[s_i]}^c}\in \bigwedge\nolimits_E^{\!n-i}\!D,\\
				\epsilon^{\Delta}_{i,u}:&\;\ext^1_{G}(C^{\Delta}_{i,u},\pi_1^{\lalg}(\underline{\phi},\bh))\xrightarrow{\sim}Ee_{I_i(u)^c}\oplus Ee_{I_i(u)_{[s_i]}^c}\in \bigwedge\nolimits_E^{\!n-i}\!D,
			\end{aligned}
		\end{equation}
		Then we get for each $i\in\Delta$ a map:
		\begin{equation}
			\begin{aligned}
				\epsilon^{\Delta}_{i}:=\bigoplus_{u\in \cI^{\flat}_{i}}\epsilon^{\Delta}_{i,u}:\ext^1_{G}\Big(\bigoplus_{u\in \cI^{\flat}_{i}}C^{\Delta}_{i,u},\pi_1^{\lalg}(\underline{\phi},\bh)\Big)\rightarrow\bigwedge\nolimits_E^{\!n-i}\!D.
			\end{aligned}
		\end{equation}
		Let $\big(\bigwedge_{E}^{\!n-i}\!D\big)^{\Delta}$ be the image of $\epsilon^{\Delta}_{i}$ and let $\big(\bigwedge_{E}^{\!n-i}\!D\big)^{\Delta,c}$  be its complement in $\bigwedge_E^{n-i}D$.\;Let $\fil_i^{\max}(D)^{\Delta}$ be the image of $\fil_i^{\max}(D)\hookrightarrow \bigwedge_E^{n-i}D\twoheadrightarrow \bigwedge_E^{n-i}D/\big(\bigwedge_{E}^{\!n-i}\!D\big)^{\Delta,c}\cong \big(\bigwedge_{E}^{\!n-i}\!D\big)^{\Delta}$.\;Denote
		\begin{equation}
			\homo^{\Delta}_{\fil}(D,D):=\sum_{i\in \Delta}g_{i}\Big(\homo_E\Big(\big(\bigwedge\nolimits_{E}^{\!n-i}\!D\big)^{\Delta},\fil_i^{\max}(D)^{\Delta}\Big)\Big)\hookrightarrow \homo_{\fil}(D,D).
		\end{equation}
		Similar to the argument in 	Theorem \ref{construext1},\;we obtain the desired map $t^{\Delta}_{D}$,\;which only depends on the $(\epsilon^{\Delta}_{u,i})_{i\in \Delta,u\in \cI_i}$ and  on a choice of $\log_p(p)\in E$.\;On the other hand,\;there exists a natural map (recall $\overline{\ext}^{1}_{u,[i]}(\Dpik,\Dpik)$ after (\ref{dfnforkappau})):
		\begin{equation}\label{dfnforgammaDDiamond}
			\begin{aligned}
				\gamma^{\Delta}_{\Dpik}:&\bigoplus_{\substack{u\in\sW_n^{S_0}\\i\in S_0(u)}}\overline{\ext}^{1}_{u,[i]}(\Dpik,\Dpik)
				\rightarrow \bigoplus_{\substack{u\in\sW_n^{S_0}\\i\in S_0(u)}} \overline{\ext}^1(R_{u,i}^{i+1},R_{u,i}^{i+1})\times\prod_{j\neq i,i+1}\ext^1(R_{u,j},R_{u,j})\\
				&\rightarrow \bigoplus_{\substack{u\in\sW_n^{S_0}\\i\in S_0(u)}}\ext^1_{\bL_{i}(\bQ_p)}\left(\pi^{\infty}_{i,u}\otimes_EL_{i}(\lambda),\pi_{i,u}(\underline{\phi},\bh)\right)\xrightarrow{\bigoplus_{u,i}\zeta_{i,u}} \ext^1_{G}\left(\pi_{\natural}^{\lalg}(\underline{\phi},\bh),\pi^{\Delta}_{1}(\underline{\phi},\bh)\right).\;
			\end{aligned}
		\end{equation}
		For  $u\in\sW_n^{S_0}$ and $i\in \Delta$,\;let $\iota_{u,[i]}:\overline{\ext}^{1}_{u,[i]}(\Dpik,\Dpik)\hookrightarrow \overline{\ext}^{1}(\Dpik,\Dpik)$ be the natural inclusion.\;For $i\in \Delta$,\;define $D^{\Delta,(i)}:=E\langle e_t:t\notin  I_{i}(u)^c\cup\{u^{-1}(i)\},\forall\;u\in \cI^{\sharp}_i\rangle$ and $D^{\Delta}_{(i)}=E\langle e_{t}:t\in I_{i}(u)^c\cup\{u^{-1}(i)\}\text{\;for some\;}u\in \cI^{\sharp}_i\rangle\subseteq D$.\;Put 
		$\bd^{\Delta}_{i}=\dim_ED^{\Delta,(i)}$ (it is obvious that $\bd^{\Delta}_{i}\leq \bd^{\sharp}_{i}\leq \bd^{\flat}_{i}$).\;For $1\leq t\leq n-1$,\;it should be true that
		$\ker(t^{\Delta}_{D})$ determines
		\begin{equation}
			\begin{aligned}
				\Big\{\fil_H^{-\bh_{{\max\{j,\bd^{\Delta}_{t}\}}+1}}(D)\oplus D/\fil_H^{-\bh_{\bd^{\Delta}_{t}}}(D)\Big\}_{1\leq j\leq t}
			\end{aligned}
		\end{equation}	
		If $\bd^{\Delta}_{t}=0$,\;then $\ker(t^{\Delta}_{D})$ determines $\{\fil_H^{-\bh_{j+1}}(D)\}_{1\leq j\leq t}$.\;

				

		\subsection{Universal extensions}\label{sectionforunverext}

		Let $R_{\Dpik}$ be the universal deformation ring of deformations of $\Dpik$ over $\Art_E$.\;For $u\in \sW_n^{S_0}$,\;let $R_{\Dpik,u}$ be the universal deformation ring of $\cF_u$-trianguline deformations of $\Dpik$ over $\Art_E$,\;and let $R_{\Dpik,g}$ be the universal deformation ring of de Rham deformations.\;Let $I_u$ (resp.,\;$I_{g}$) be the kernel of $R_{\Dpik}\rightarrow R_{\Dpik,u}$ (resp.,\;$R_{\Dpik}\rightarrow R_g$).\;Then $R_{\Dpik,u,g}\cong R_{\Dpik}/(I_u+I_g)$ is the universal deformation ring of de Rham $\cF_u$-trianguline deformations of $\Dpik$ over $\Art_E$.\;We have natural surjections $R_{\Dpik}\twoheadrightarrow R_{\Dpik,u}\rightarrow  R_{\Dpik,u,g}$.\;For a continuous character $\delta$ of $\bT(\bQ_p)$,\;denote by the $R_{\delta}\cong E[[U,T]]$ (resp.\;$R_{\delta,g}$) the universal deformations ring (resp.,\;the universal deformations ring of locally algebraic deformations) of $\delta$ over $\Art_E$.\;We have natural surjection $R_{\delta}\twoheadrightarrow R_{\delta,g}$.\;For a complete local Noetherian $E$-algebra $R$,\;we use $\fm_R$ (or use $\fm$ for simplicity when it does not cause confusion) to denote its maximal ideal.\;Then for $\ast\in\{\emptyset,g,u\}$,\;we have $(\fm_{R_{\Dpik,\ast}}/\fm^2_{R_{\Dpik,\ast}})^{\vee}\cong \ext^{1}_{\ast}(\Dpik,\Dpik)$.\;Moreover,\;$(\fm_{R_\delta}/\fm^2_{R_\delta})^{\vee}\cong \homo(\bT(\bQ_p),E)$ and $(\fm_{R_{\delta,g}}/\fm^2_{R_{\delta,g}})^{\vee}\cong \homo_{\mathrm{sm}}(\bT(\bQ_p),E)$.\;

		Let $\cH$ be the Bernstein centre over $E$ associated to $\mathrm{PS}^{\infty}_{1}(\underline{\phi})$,\;and let $\widehat{\cH}_{\underline{\phi}}$ be the completion of $\cH$ at $\mathrm{PS}^{\infty}_{1}(\underline{\phi})$.\;We have isomorphism
		$\widehat{\cH}_{\underline{\phi}}\xrightarrow{\sim}R_{\unr(\underline{\phi}^u)z^{\bh},g}$ by sending a smooth deformation $\chi$ of $\unr(\underline{\phi}^u)z^{\bh}$ to $\big(\ind^G_{\ob(\bQ_p)}\eta\chi z^{-\bh}\big)^{\infty}$.\;Moreover,\;for $u\in \sW_n^{S_0}$,\;we have the following commutative Cartesian diagram over $\Art_E$:
		\begin{equation}
			\xymatrix{
				R_{\unr(\underline{\phi}^u)z^{\bh}}/\fm^2 \ar[r] \ar[d]^{\kappa_u} & R_{\unr(\underline{\phi}^u)z^{\bh},g}/\fm^2 \ar[d]^{\kappa_u} 
				\\
				R_{\Dpik,u}/\fm^2 \ar[r]	&  R_{\Dpik,u,g}/\fm^2,}
		\end{equation}

		Let $A_{\Dpik}:=R_{\Dpik}/\fm^2\times_{R_{\Dpik,g}} A_0$ with $A_0:=R_{\unr(\underline{\phi})z^{\bh},g}/\fm^2\hookrightarrow R_{\Dpik,g}/\fm^2$,\;and $A_{\Dpik,u}:=R_{\Dpik,u}/\fm^2\times_{R_{\Dpik,u,g}} A_{0,u}$ with $A_{0,u}:=R_{\unr(\underline{\phi}^u)z^{\bh},g}/\fm^2\hookrightarrow R_{\Dpik,u,g}/\fm^2$.\;The tangent space of $A_{\Dpik}$ (resp.,\;$A_{\Dpik,u}$) is naturally isomorphism to $\overline{\ext}^1(\Dpik,\Dpik)$ (resp.,\;$\overline{\ext}^1_u(\Dpik,\Dpik)\cong \homo_u(\bT(\bQ_p),E)$).\;We denote the kernel of $A_{\Dpik}\rightarrow A_{\Dpik,u}$ with $\cI_u$.\;Let $\cI_{\Dpik}$ be the kernel of the natural morphism $A_{\Dpik}\rightarrow \prod_{u\in \sW_n^{S_0}}A_{\Dpik,u}$.\;Note that $\cI_{\Dpik}=\cI_1$ (resp.,\; $I_{\Dpik}=0$) if $\Dpik$ is Steinberg (resp.,\;crystalline).\;
		
		Similarity,\;let $\pi_{1}(\underline{\phi},\bh)^{\univ,-}$ (resp.,\;$\pi_{1}(\underline{\phi},\bh)_u^{\univ,-}$) be the universal extension of $\mathrm{Im}(\gamma_{\Dpik})\otimes \pi_{\natural}^{\lalg}(\underline{\phi},\bh)$ (resp.,\;$\zeta_u(\homo_u(\bT(\bQ_p),E))\otimes \pi_{\natural}^{\lalg}(\underline{\phi},\bh)$) by $\pi_{1}(\underline{\phi},\bh)$ (see (\ref{dfnforgammaD}) for the map $\mathrm{Im}(\gamma_{\Dpik})$).\;Morevoer,\;let $\delta_u:=\unr(\underline{\phi}^u)\eta\chi_{\lambda}$ and let $\widetilde{\delta}_u^{\univ,-}$ be the universal extension of $\homo_u(\bT(\bQ_p),E)\otimes_E\delta_u$ by $\delta_u$,\;so that 
		the induced representation $I^G_{\ob(\bQ_p)}\widetilde{\delta}_u^{\univ,-}$ is the universal extension of $\zeta_u(\homo_u(\bT(\bQ_p),E))\otimes \mathrm{PS}^{\lalg}_{u}(\underline{\phi},\bh)$ by $\mathrm{PS}_{u,1}(\underline{\phi},\bh)$.\;Composed the above universal extensions with the pull-back map for the natural map
		$p_u:\pi_{\natural}^{\lalg}(\underline{\phi},\bh)\rightarrow \mathrm{PS}_{u}^{\lalg}(\underline{\phi},\bh)$,\;we get a locally analytic representation $(I^G_{\ob(\bQ_p)}\widetilde{\delta}_u^{\univ,-})^{\flat}$ which is the universal extension of $\zeta_u(\homo_u(\bT(\bQ_p),E))\otimes \pi_{\natural}^{\lalg}(\underline{\phi},\bh)$ by $\mathrm{PS}_{u,1}(\underline{\phi},\bh)$.\;

		Let $U_u$ be the intersection of $I^G_{\ob(\bQ_p)}\widetilde{\delta}_u^{\univ}$ with the  kernel of the projection $\mathrm{PS}_{u}(\underline{\phi},\bh)\twoheadrightarrow \mathrm{ST}_{u}(\underline{\phi},\bh)$,\;then we have hence an isomorphisms of $G$-representations
		\begin{equation}
			\begin{aligned}
				\left((I^G_{\ob(\bQ_p)}\widetilde{\delta}_u^{\univ,-})^{\flat}/U_u\right)\oplus_{\mathrm{ST}_{u,1}(\underline{\phi},\bh)}\pi_{1}(\underline{\phi},\bh)\xrightarrow{\sim }\pi_{1}(\underline{\phi},\bh)_u^{\univ,-}.\;
			\end{aligned}
		\end{equation}
		We have a natural action of $A_{\Dpik,u}\hookrightarrow R_{\unr(\underline{\phi}^u)z^{\bh}}/\fm^2$ on $\widetilde{\delta}_u^{\univ,-}$,\;where $x\in A_{\Dpik,u}\cong \homo_u(\bT(\bQ_p),E)^{\vee}$ acts via $x:\widetilde{\delta}_u^{\univ,-}\twoheadrightarrow \homo_u(\bT(\bQ_p),E)\otimes_E\delta_u\xrightarrow{x}\delta_u\hookrightarrow\widetilde{\delta}_u^{\univ,-}$.\;Similarity,\;$\pi_{1}(\underline{\phi},\bh)_u^{\univ,-}$ admits an $A_{\Dpik,u}$-action,\;which is given by 
		\begin{equation}
			\begin{aligned}
				x:\pi_{1}(\underline{\phi},\bh)_u^{\univ,-}&\rightarrow \zeta_u(\homo_u(\bT(\bQ_p),E))\otimes \pi_{\natural}^{\lalg}(\underline{\phi},\bh)\xrightarrow{x}\pi_{\natural}^{\lalg}(\underline{\phi},\bh)\twoheadrightarrow Q_{\Delta}^{\diamond}(\emptyset,\lambda)\hookrightarrow \pi_{1}(\underline{\phi},\bh)_u^{\univ,-},
			\end{aligned}
		\end{equation}
		where $x\in \fm_{A_{\Dpik,u}}/\fm^2_{A_{\Dpik,u}}\cong \homo_u(\bT(\bQ_p),E)^{\vee}\xrightarrow{\zeta_u}(\zeta_u(\homo_u(\bT(\bQ_p),E))^{\vee}$.\;Moreover,\; $I^G_{\ob(\bQ_p)}\widetilde{\delta}_u^{\univ,-}/U_u\hookrightarrow \pi_{1}(\underline{\phi},\bh)_u^{\univ,-}$ is $A_{\Dpik,u}$-equivalent.\;
		
		\begin{pro}\label{actionADonuniversal}
			There is a unique $A_{\Dpik}$-action on $\pi_{1}(\underline{\phi},\bh)^{\univ,-}$ such that for any $u\in \sW_n^{S_0}$,\;we have an $A_{\Dpik,u}\times G$-equivariant injection $\pi_{1}(\underline{\phi},\bh)_u^{\univ,-}\hookrightarrow \pi_{1}(\underline{\phi},\bh)^{\univ,-}[\cI_u]$.\;
		\end{pro}
		\begin{proof}
			By Theorem \ref{descibeimagetD},\;we define an $A_{\Dpik}$-action on $\pi_{1}(\underline{\phi},\bh)^{\univ,-}$ by letting $x\in \fm_{A_{\Dpik}}/\fm^2_{A_{\Dpik}}\cong \overline{\ext}^1(\Dpik,\Dpik)\rightarrow \mathrm{Im}(\gamma_{\Dpik})^{\vee}$ act via $x:\pi_{1}(\underline{\phi},\bh)^{\univ,-}\rightarrow \mathrm{Im}(\gamma_{\Dpik})^{\vee}\otimes_E\pi_{\natural}^{\lalg}(\underline{\phi},\bh)\rightarrow \pi_{\natural}^{\lalg}(\underline{\phi},\bh)\rightarrow \pi_{1}(\underline{\phi},\bh)^{\univ,-}$.\;This action satisfies the property in the theorem.\;The uniqueness follows from the fact that $\pi_{1}(\underline{\phi},\bh)^{\univ,-}$ is generated by $\pi_{1}(\underline{\phi},\bh)_u^{\univ,-}$.\;
		\end{proof}

		\section{Local-global compatibility}

		Under the Taylor--Wiles hypothesis,\;we show the local and global compatibility results for the patched setting.\;	
		\begin{hypothesis}[Taylor--Wiles hypothesis]\label{TWhypo}
			\begin{itemize}
				\item[(1)] $p>2$;
				\item[(2)] the field $F$ is unramified over $F^+$,\;$F$ does not contain a non-trivial root $\sqrt[p]{1}$ of $1$ and $G$ is quasi-split at all finite places of $F^+$;
				\item[(3)] $U_{v}$ is hyperspecial when the finite place $v$ of $F^+$ is inert in $F$;
				\item[(4)] $\overline{\rho}$ is absolutely irreducible and $\overline{\rho}(\gal_{F(\sqrt[p]{1})})$ is adequate.\;
			\end{itemize}
		\end{hypothesis}
		
		\subsection{Patched eigenvariety}

		We use the following patched setting and follow the notation of \cite[Section 4.1.1]{2019DINGSimple} and \cite[Section 2]{PATCHING2016}.\;We have a patched Galois deformation ring $R_{\infty}=R_{\infty}^{\fp}\widehat{\otimes} R_{\overline{r}}^{\square}$,\;where $\fp$ is a $p$-adic place over a totally real subfield $F^+$,\;$L:=F^+_{\fp}$,\;and $\overline{r}: \gal_L \rightarrow \GLN_n(k_E)$ is a continuous representation and  $R_{\overline{r}}^{\square}$ is  the maximal reduced and $p$-torsion free quotient of the universal $\co_E$-lifting ring of $\overline{r}$.\;We have a $R_{\infty}$-admissible unitary representation $\Pi_{\infty}$ (i.e.,\;the so-called patched Banach representation) of $G$ over $E$.\;We refer to \textit{loc.cit} for detail.\;Let $\Pi_\infty^{R_\infty-\ana}$ be the subrepresentation of $G$ of locally $R_\infty$-analytic vectors of $\Pi_\infty$ (see \cite[Section 3.1]{breuil2017interpretation}).\;

		We briefly recall the following version of patched eigenvariety,\;given in \cite[Section 4.1.1]{2019DINGSimple}.\;Put $\FX_{\infty}:=(\Spf R_{\infty})^{\rig}$,\;$\FX_{\infty}^{{\fp}}:=(\Spf R^{\fp}_{\infty})^{\rig}$ and  $\mathfrak{X}_{\overline{r}}=(\spf\;R_{\overline{r}}^{\square})^{\rig}$.\;We have thus
		$\FX_{\infty}:=\FX_{\infty}^{{\fp}}\times \mathfrak{X}_{\overline{r}}$.\;By \cite[Section 4.1.2]{2019DINGSimple},\;we see that $J_{\bB(\bQ_p)}(\Pi_\infty^{R_\infty-\ana})^\vee$ is a coadmissible module over $\cO(\FX_{\infty}\times\cT)$, which corresponds to a coherent sheaf $\cM_{\infty}$ over $\FX_{\infty}\times\cT$ such that
			\[\Gamma\Big(\FX_{\infty}\times\cT,\cM_{\infty}\Big)\cong J_{\bB(\bQ_p)}(\Pi_\infty^{R_\infty-\ana})^\vee.\]
			The so-called patched  eigenvariety $X_{\fp}(\overline{\rho})\hookrightarrow\FX_{\infty}\times\cT$ is 
the Zariski-closed support of $\cM_{\infty}$.\;By \cite[Theorem 4.1]{2019DINGSimple},\;we have
		\begin{pro}\label{basicpropertyeigen}
			\begin{itemize}
				\item[(1)]  For $x=(\fm_x,\chi_x)\in \FX_{\infty}\times\cT$,\;$x\in X_{\fp}(\overline{\rho})$ if and only if
				$J_{\bB(\bQ_p)}(\Pi_\infty^{R_\infty-\ana} )[\fm_x,\bT(\bQ_p)=\chi_x]\neq0$.
				\item[(2)] The rigid space $X_{\fp}(\overline{\rho})$ is reduced and equidimensional of dimension
				$g+n+n^2(|S_p|+1)+[F^+:\BQ]\frac{n(n-1)}{2}$.\;The set of very classical non-critical generic points is Zariski-dense in $X_{\fp}(\overline{\rho})$ and is an accumulation set.\;The set of very classical non-critical generic points accumulates at point $x=(\fm_x, \chi_x)$ with $\chi_x$ locally algebraic.\;
				\item[(3)] The coherent sheaf $\cM_{\infty}$ is Cohen-Macaulay over $X_{\fp}(\overline{\rho})$.\;
			\end{itemize}
		\end{pro}
		The patched eigenvariety $X_{\fp}(\overline{\rho})$ is related to the trianguline variety $X_{\mathrm{tri}}(\overline{r})$ in \cite{breuil2017interpretation}  as follows.\;Let $\iota_{\fp}:\cT\rightarrow \cT$ be the automorphism defined by 
		\[\iota_{\fp}(\delta_1,\cdots,\delta_n):=\delta_{\bB}\cdot(\delta_1,\delta_2\unr(p^{-1}),\cdots,\delta_n\unr(p^{-(n-1)})).\]
		Note that $\delta_{\bB}=\eta^{-2}$ is  the modulus character of $\bB$ and  $\iota_{\fp}(\delta_1,\cdots,\delta_n)=(\delta_1,\cdots,\delta_n)\cdot \zeta$,\;where
		\[\zeta:=\big(\unr(p^{1-n}),\cdots,\unr(p^{i-n})z^{i-1},\cdots,\;z^{n-1}\big).\;\]
		Then $\mathrm{id}\times \iota_{\fp}$ induces an isomorphism of rigid spaces $\mathrm{id}\times \iota_{\fp}:\mathfrak{X}_{\overline{r}}\times\cT
		\xrightarrow{\sim} \mathfrak{X}_{\overline{r}}\times\cT$.\;Let $\iota_{\fp}\big(X_{\mathrm{tri}}(\overline{r})\big)$ be the image of $X_{\mathrm{tri}}(\overline{r})$ via this automorphism.\;Then the natural embedding 
		\[X_{\fp}(\overline{\rho})\hookrightarrow\FX_{\infty}\times\cT\cong (\Spf R_{\infty})^{\rig}\times\cT\cong \FX_{\overline{\rho}^{\fp}}\times \BU^g\times \mathfrak{X}_{\overline{r}}\times\cT\]
		factors through $X_{\fp}(\overline{\rho})\hookrightarrow\FX_{\overline{\rho}^{\fp}}\times \BU^g\times \iota_{\fp}\big(X_{\mathrm{tri}}(\overline{r})\big)$.\;Therefore,\;$\iota_\fp$ induces morphisms
		\begin{equation}\label{injpatchtotri1}
			\iota_\fp^{-1}:Y(U^{\fp},\overline{\rho})\hookrightarrow X_{\fp}(\overline{\rho})\rightarrow X_{\mathrm{tri}}(\overline{r}).
		\end{equation}
		For each irreducible component $\FX^{\fp}$ of  $\FX_{\overline{\rho}^{\fp}}$,\;there is a (possibly empty) union $X^{\FX^{\fp}-\mathrm{aut}}_{\mathrm{tri}}(\overline{r})$ of irreducible components of $X_{\mathrm{tri}}(\overline{r})$ such that we have an isomorphism of closed analytic subsets of $\FX_{\infty}\times\cT$:
		\begin{equation}\label{comparepatchedandtrivar}
			X_{\fp}(\overline{\rho})\cong\bigcup_{\FX^{\fp}}\FX^{\fp}\times \iota_{\fp}\big(X^{\FX^{\fp}-\mathrm{aut}}_{\mathrm{tri}}(\overline{r})\big)\times \BU^g.
		\end{equation}
		For any point $y\in X_{\fp}(\overline{\rho})\hookrightarrow\FX_{\infty}\times\cT\cong \FX_{\infty}^{{\fp}}\times \mathfrak{X}_{\overline{r}}\times\cT$,\;denote by $r_{y}$ (resp.,\;$\fm_{r_y}$) its image in $\FX_{\infty}$ (resp.,\;the corresponding maximal ideal of $R_{\infty}[1/p]$).\;Denote by $r_{y,\fp}$ (resp.,\;$r_y^{\fp}$) its image in $\mathfrak{X}_{\overline{r}}$ (resp.,\;$\FX^{\fp}_{\infty}$),\;and by $\underline{\epsilon}_y$ its image in $\cT$ (so $y=(r_y^{\fp},r_{y,\fp},\underline{\epsilon}_y)=(r_y,\underline{\epsilon}_y)$).\;

		\subsection{Main theorem}
		
		This section follows the route of \cite[Section 4.1]{ParaDing2024}.\;We fix a Galois representation $\rho\in \FX_{\infty}$.\;We make the following Hypothesis.\;
		\begin{hypothesis}\label{hyongaloisrep1}
			\begin{itemize}
				\item[(a)] $\rho_p:=\rho|_{\gal_{F^+_{\fp}} }$ is a generic semistable  $p$-adic Galois representation with regular Hodge-Tate weights $\bh=(\hpi_{1}>\hpi_{2}>\cdots>\hpi_{n})$.\;
				\item[(b)]  Let $\Dpik:=D_{\rig}(\rho_p)$.\;For any $u\in \sW_n^{S_0}$,\;$\Dpik$ admits a non-critical triangulation $\cF_u$ with parameters $\delta_{\underline{\phi},u}$ (keep the notation in  Section \ref{Omegafil}).\;
				\item[(c)] Put $\lambda_{\bh}:=\bh-\theta\in X_{\Delta}^+$ and  $y_{u}:=(\rho,\delta_{\bB}\unr(\underline{\phi}^u)\chi_{{\lambda_{\bh}}})\in \FX_{\overline{\rho},U^{\fp}}\times \cT$,\;then $y_u\in X_{\fp}(\overline{\rho})$.
			\end{itemize}
		\end{hypothesis}
		Note that $x_{u}:=\iota^{-1}_{\fp}(y_u)=(\rho_p,\delta_{\underline{\phi},u})\in X_{\mathrm{tri}}(\overline{r})$.\;Similar to \cite[Proposition 4.3,\;Corollary 4.4 and Lemma 4.6]{2019DINGSimple},\;the non-critical assumption and \cite[Theorem 4.35]{PATCHING2016} show that
		\begin{lem}
			\begin{itemize}
				\item[(1)] $x_u$ is a smooth point of $X_{\mathrm{tri}}(\overline{r})$ and  $y_{u}$ is a non-critical classical smooth point of $X_{\fp}(\overline{\rho})$.\;Moreover,\;$\cM_{\infty}$ is locally free of rank $1$ at  $y_{u}$.\;
				\item[(2)] $y_{u}$ does not admit companion points of non-dominant weight.\;
			\end{itemize}
			
		\end{lem}
		Recall that the completion of $R_{\overline{r}}^{\square}[1/p]$ at $\rho_p$ is natural to $R^{\Box}_{\rho_p}\cong R^{\Box}_{\Dpik}$,\;where $R^{\Box}_{\rho_p}$ is the framed universal deformation ring of $\rho$ (over $\Art_E$).\;Let $R^{\Box}_{\Dpik}:=R_{\Dpik}\otimes_{R_{\rho_p}}R^{\Box}_{\rho_p}$.\;Put $R_{\Dpik,u}^{\Box}:=R_{\Dpik}^{\Box}\widehat{\otimes}_ER_{\Dpik,u}$ for $u\in \sW_n^{S_0}$.\;As $x_u$ in non-critical,\;the completion of $X_{\mathrm{tri}}(\overline{r})$ at $x_u$ is naturally isomorphic to $R_{\Dpik,u}^{\Box}$.\;Keep the notation in Section \ref{sectionforunverext} and the discussion below \cite[(4.3)]{ParaDing2024}.\;
		
		Let $\fa\supseteq\fm^2_{R^{\Box}_{\Dpik}}$ be an ideal of $R^{\Box}_{\Dpik}$ satisfies that $\fa/\fm^2_{R^{\Box}_{\Dpik}}\oplus \fm_{A_{\Dpik}}/\fm^2_{A_{\Dpik}}\xrightarrow{\sim } \fm_{R^{\Box}_{\Dpik}}/\fm^2_{R^{\Box}_{\Dpik}}$.\;Then the composition $A_{\Dpik}\hookrightarrow R^{\Box}_{\Dpik}/\fm^2_{R^{\Box}_{\Dpik}}\twoheadrightarrow R^{\Box}_{\Dpik}/\fa$ is an isomorphism (i.e.,\;$\fa$ delete the information comes from framing).\;In particular,\;we get $\fa/\fm^2_{R^{\Box}_{\Dpik,u}}\oplus \fm_{A_{\Dpik,u}}/\fm^2_{A_{\Dpik,u}}\xrightarrow{\sim } \fm_{R^{\Box}_{\Dpik,u}}/\fm^2_{R^{\Box}_{\Dpik,u}}$ and $\fa/\fm^2_{R^{\Box}_{\Dpik,u,g}}\oplus \fm_{A_{\Dpik,g}}/\fm^2_{A_{\Dpik,g}}\xrightarrow{\sim } \fm_{R^{\Box}_{\Dpik,g}}/\fm^2_{R^{\Box}_{\Dpik,g}}$,\;so that 
		$A_{\Dpik,u}\hookrightarrow R^{\Box}_{\Dpik,u}/\fm^2_{R^{\Box}_{\Dpik,u}}\twoheadrightarrow R^{\Box}_{\Dpik,u}/\fa$ and $A_{\Dpik,u,g}\hookrightarrow R^{\Box}_{\Dpik,g}/\fm^2_{R^{\Box}_{\Dpik,g}}\twoheadrightarrow R^{\Box}_{\Dpik,g}/\fa$ are isomorphisms too.\;We use $\fa\subseteq R_{\overline{r}}^{\square}[1/p]$ denote the preimage of $\fa\subseteq R_{\overline{r}}^{\square}$.\;Let $\fm_x:=(\fm_{\rho},\fm^{\fp})\subseteq\fa_x:=(\fa,\fm^{\fp})\subseteq R_{\infty}[1/p]$.\;
		
		By \cite[Lemma 4.2]{ParaDing2024},\;we have 
		$\Pi_\infty^{R_\infty-\ana}[\fm_x]=\Pi_\infty^{R_\infty-\ana}[\fa_x][\fm_{A_{\Dpik}}]$.\;By replacing the reference \cite{PATCHING2016} in the proof of \cite[Lemma 4.2 (2)]{ParaDing2024} with \cite[Lemma 5.2,\;Theorem 5.6 and Proposition 5.4]{BSCONJ},\;we also get that
		\[\homo_G(\pi^{\lalg}_{1}\big(\underline{\phi},\bh),\Pi_\infty^{R_\infty-\ana}[\fa_x]\big)\cong \homo_G\big(\pi^{\lalg}_{1}(\underline{\phi},\bh),\Pi_\infty^{R_\infty-\ana}[\fm_x]\big)\cong E.\]
		Let $U_u=U_{\fp,u}\times U^{\fp}_u\subseteq  \iota_{\fp}\big(X_{\mathrm{tri}}(\overline{r}))\times\FX_{\overline{\rho}^{\fp}}^{\Box}$ be a smooth affinoid neighourhood of $y_u$ such $y_{u'}\not\in U_{\fp,u}$ for $u'\neq u$.\;Let $\fm_{y_{u,\fp}}$ be the maximal ideal of $\cO(U_{\fp,w})$ at $y_{u,\fp}$ and $\fa\supset \fm^2_{y_{u,\fp}}$ be the closed ideal generated by the above ideal $\fa\subseteq R_{\overline{r}}^{\square}[1/p]$.\;Consider $\widetilde{\cM}_{y_u}:=\cM_{\infty}/(\fa+\fm^{\fp})$.\;There are natural $\bT(\bQ_p)\times R_{\infty}$-equivariant injections:
		\[{\cM}_{y_u}\hookrightarrow \widetilde{\cM}_{y_u}\hookrightarrow J_{\bB(\bQ_p)}(\Pi_\infty^{R_\infty-\ana})[\fa_x].\]
		Since $\cM_{\infty}$ is locally free of rank $1$ at all $y_{u}$,\;we obtain that $\widetilde{\cM}_{y_u}\cong R^{\Box}_{\Dpik,u}/\fa$,\;so that $\dim_E\widetilde{\cM}_{y_u}=1+d_u$,\;where $d_u:=\dim_E\homo_u(\bT(\bQ_p),E))$.\;In this case,\;we have a $\bT(\bQ_p)\times A_{\Dpik,u}$-equivariant isomorphism $\widetilde{\cM}^{\vee}_{y_u}\cong \widetilde{\delta}_u^{\univ,-}\delta_{\bB}$.\;Similar to the proof of \cite[Lemma 4.11]{2019DINGSimple},\;the maps ${\cM}_{y_u}\hookrightarrow  \Pi_\infty^{R_\infty-\ana}[\fm_x]$ and $ \widetilde{\cM}_{y_u}\hookrightarrow \Pi_\infty^{R_\infty-\ana}[\fa_x]$ are balanced by using the non-critical assumption),\;hence induces a $G\times R_{\infty}$-equivariant injection:
		\begin{equation}
			\begin{aligned}
				&\iota_u:I^G_{\ob(\bQ_p)}\widetilde{\delta}_u^{\univ,-}\hookrightarrow \Pi_\infty^{R_\infty-\ana}[\fa_x].\;
			\end{aligned}
		\end{equation}
		Recall that $U_u$ is  the intersection of $I^G_{\ob(\bQ_p)}\widetilde{\delta}_u^{\univ,-}$  with  $\ker(\mathrm{PS}_{u}(\underline{\phi},\bh)\twoheadrightarrow \mathrm{ST}_{u}(\underline{\phi},\bh))$.\;Similar to the argument in the proof of \cite[Proposition 4.7 (a)]{2019DINGSimple},\;we see that $\iota_u$  factors through the projection $I^G_{\ob(\bQ_p)}\widetilde{\delta}_u^{\univ,-}\twoheadrightarrow I^G_{\ob(\bQ_p)}\widetilde{\delta}_u^{\univ,-}/U_u$.\;Similar to \cite[Lemma 4.4]{ParaDing2024},\;we have (recall the map $p^{\Delta}_{\mathrm{ST},u}$ in (\ref{parametersmaininjtosimple}))
		\begin{pro}\label{imiotawmx}
			We have $\mathrm{ST}^{\Delta,-}_{\cF_{u}}(\underline{\phi},\bh)\hookrightarrow \mathrm{Im}(\iota_u)[\fm_x]$.\;In particular,\;$\mathrm{Im}(\iota_u)[\fm_x]$ determines $p^{\Delta}_{\mathrm{ST},u}([\underline{\sL}(\Dpik)])$.\; \end{pro}
		\begin{proof}
			The argument in the proof of \cite[Lemma 4.4]{ParaDing2024} show that $\mathrm{ST}_{u,1}(\underline{\phi},\bh)\hookrightarrow (\mathrm{Im}\iota_u)[\fm_x]$.\;We further show that $\mathrm{ST}^{\Delta,-}_{\cF_{u}}(\underline{\phi},\bh)\hookrightarrow (\mathrm{Im}\iota_u)[\fm_x]$,\;this assertion is essentially obtained by applying the proof in \cite[Theorem 4.10]{2019DINGSimple} to the Steinberg blocks $\{E_{u,i}\}_{1\leq i\leq f_u}$ of the $\bP_{S_0(u)}$-parabolic filtration $\cF_{S_0(u)}$.\;We see that the representation $\mathrm{ST}^{\Delta,-}_{\cF_{u}}(\underline{\phi},\bh)$ lives in $I^G_{\ob(\bQ_p)}\widetilde{\delta}_u^{\univ,-}$,\;by the argument around \cite[(96)]{2019DINGSimple} (and a similar strategy as in the first paragraph of \cite[Lemma 4.4]{ParaDing2024}),\;such representation is also killed by $\fm_x$.\;We complete the proof.\;
		\end{proof}
		
		Let $\widetilde{\pi}$ be the subrepresentation of $\Pi_\infty^{R_\infty-\ana}[\fa_x]$ generated by $\mathrm{Im}(\iota_u)$  for all $u\in \sW_n^{S_0}$.\;Note that $\widetilde{\pi}$ has an $A_{\Dpik}$-action via $A_{\Dpik}\xrightarrow{\sim}R^{\Box}_{\Dpik,u}/\fa \xrightarrow{\sim}R_{\infty}[1/p]/\fa_x$.\;
		\begin{thm}\label{mainthmglobal}There is an $A_{\Dpik}\times G$-equivariant isomorphism $\widetilde{\pi}\cong \pi_{1}(\underline{\phi},\bh)^{\univ,-}$.\;Moreover,\; $\pi^{\sharp,\Delta,-}_{\min}(\Dpik)\hookrightarrow \widetilde{\pi}[\fm_x]$.\;
		\end{thm}
		\begin{proof}
			Similar to the argument in the proof \cite[Theorem 4.5]{ParaDing2024},\;we see that $\widetilde{\pi}$ is a subrepresentation of $\pi_{1}(\underline{\phi},\bh)^{\univ}$ and contains all $I^G_{\ob(\bQ_p)}\widetilde{\delta}_u^{\univ,-}$ so that $\widetilde{\pi}\cong \pi_{1}(\underline{\phi},\bh)^{\univ,-}$.\;We equip $\pi_{1}(\underline{\phi},\bh)^{\univ,-}$ with an $A_{\Dpik}$-action by the $A_{\Dpik}$-action on $\widetilde{\pi}$ (induced from $R_{\infty}$) and  get a  $G\times A_{\Dpik}$-isomorphism $\widetilde{\pi}\cong \pi_{1}(\underline{\phi},\bh)^{\univ,-}$.\;We thus see that $\pi^{\sharp}_{\min}(\Dpik)= \pi_{1}(\underline{\phi},\bh)^{\univ,-}[\fm_{A_{\Dpik}}]\hookrightarrow \widetilde{\pi}[\fm_x]$.\;The last assertion follows.\;
		\end{proof}
		In particular,\;by Corollary \ref{thmforamosttwo},\;we obtain:
		\begin{cor}\label{thmforamosttwoLGC}
			Keep the situation.\;$\Pi_\infty^{R_\infty-\ana}[\fm_x]$ determines $\Dpik$ if $ \max_{1\leq l\leq s}l_i\leq 2$.\;
		\end{cor}
		
		\section[Appendix]{Appendix}\label{dfnandproofforFMLinv}
		
			\begin{lem}\label{lemdimforsimpleLinv}(The argument around \cite[(61)]{2019DINGSimple})
			Let $\delta:=\unr(p)z^{k}$ for some $k\in \BZ_{\geq 1}$. The cup product induces a perfect pairing:
			\[\hH^1_{(\varphi,\Gamma)}(\cR_E(\delta))\times \hH^1_{(\varphi,\Gamma)}(\cR_E)\xrightarrow{\cup }\hH^2_{(\varphi,\Gamma)}(\cR_E(\delta))\cong E.\]
			The local class field theory gives a natural isomorphism $\hH^1_{(\varphi,\Gamma)}(\cR_E)\cong \homo(\bQ_p^{\times},E)$, which is $2$-dimensional and generated by $\log_{p}$ and $\val_p$. We have $\dim_E\hH^1_{e}(\cR_E(\delta))=1$, $\dim_E\hH^1_{g}(\cR_E(\delta))=2$ and $\hH^1_{e}(\cR_E(\delta))^{\perp}=\homo_{\sm}(\bQ_p^{\times},E)=E\val_p$.
		\end{lem}

				\subsection[$I$: explicit examples for $\GLN_2(\bQ_p)$ and $\GLN_3(\bQ_p)$]{Appendix:\;explicit examples for $\GLN_{2}(\bQ_p)$ and $\GLN_3(\bQ_p)$}\label{exampleGL23}
				
				In this section,\;we give some explicit examples of locally analytic representations $\pi_{\min}(\Dpik)$ that determine $\Dpik$.\;Such representations $\pi_{\min}(\Dpik)$ are expected to be locally analytic subrepresentations of the conjectural locally analytic representation $\pi^{\ana}(\rho_p)$ via the $p$-adic local Langlands correspondence.
				
				\subsubsection[$\GLN_{2}(\bQ_p)$]{$\GLN_{2}(\bQ_p)$}
				
				Suppose that $\underline{\phi}=(\phi,\phi p)$ for some $\phi\in E^{\times}$.\;Let $\Dpik=[\cR_E(\unr(\phi)z^{h_1})-\cR_E(\unr(\phi p)z^{h_2})]$ (a non-split extension) be a non-critical semistable $(\varphi,\Gamma)$-module.\;Recall the following perfect cup product (see the argument around \cite[(61)]{2019DINGSimple}):
				\begin{equation}
					\begin{aligned}
						\ext^1_{(\varphi,\Gamma)}(\cR_E(\unr(\phi p)z^{h_2}),\cR_E(\unr(\phi)z^{h_1}))\times&\; \ext^1_{(\varphi,\Gamma)}(\cR_E(\unr(\phi_1)z^{h_2}),\cR_E(\unr(\phi)z^{h_1}))\\
						&\xrightarrow{\cup }\ext^2_{(\varphi,\Gamma)}(\cR_E(\unr(\phi p)z^{h_2}),\cR_E(\unr(\phi)z^{h_1}))\cong E
					\end{aligned}
				\end{equation}
				Note that $\ext^1_{(\varphi,\Gamma)}(\cR_E(\unr(\phi)z^{h_2}),\cR_E(\unr(\phi)z^{h_1}))\cong \homo(\bQ_p^{\times},E)$.\;Let $\sL(\Dpik)=(E[\Dpik])^{\perp}$ via such perfect pairing.\;We have $I_0=\Delta=\{1\}$ and $\sW_2=S_2=\langle s_1\rangle$,\;so that $Q^{\diamond}_{\{1\}}(\emptyset,\lambda)=\st_2^{\infty}(\lambda)$ and $Q^{\diamond}_{\{1\}}(\{1\},\lambda)=L(\lambda)$.\;Put 
				\begin{equation*}
					\begin{aligned}
						&I(\underline{\phi},\bh)=\cF^{\GLN_{2}}_{\overline{B}}\left(L(-s_1\cdot\lambda),\unr(\phi)\otimes\unr(\phi)\right),\\
						&\widetilde{I}(\underline{\phi},\bh)=\cF^{\GLN_{2}}_{\overline{B}}\left(L(-s_1\cdot\lambda),|\cdot|^{-1}\unr(\phi)\otimes |\cdot|\unr(\phi)\right).\;
					\end{aligned}
				\end{equation*}
				Suppose that $\Dpik$ is crystalline.\;In this case,\;$\sL(\Dpik)=E\val_p$ and $\pi_{\min}(\Dpik)$ has the form 
				\begin{equation}\label{struGL2ST}
					\begindc{\commdiag}[300]
					\obj(1,3)[b]{$\st^{\infty}_2(\lambda)$}
					\obj(3,3)[e]{$L(\lambda)$}
					\obj(3,2)[f]{$I(\underline{\phi},\bh)$}
					\obj(5,3)[f1]{$\widetilde{I}(\underline{\phi},\bh)$}
					\mor{b}{e}{}[+1,\solidline]
					\mor{f}{b}{}[+1,\solidline]
					\mor{f1}{e}{}[+1,\solidline]
					\enddc
				\end{equation}
				Suppose that $\Dpik$ is semistable non-crystalline,\;then $\sL(\Dpik)=E(\val_p+\cL(\Dpik)\log_p)$ for some $\cL(\Dpik)\in E^{\times}$.\;
				In this case,\;$\pi_1^{\lalg}(\underline{\phi},\bh):=\st_2^{\infty}(\lambda)$ and $\pi_1(\underline{\phi},\bh):=\st_2^{\infty}(\lambda)-I(\underline{\phi},\bh)$.\;Put $\pi_c(\underline{\phi},\bh)^-:=I(\underline{\phi},\bh)-L(\lambda)$ and  $\pi_c(\underline{\phi},\bh):=I(\underline{\phi},\bh)-L(\lambda)-\widetilde{I}(\underline{\phi},\bh)$.\;Then 
				\begin{eqnarray}\label{exampleforGl2}
					\pi^{-}_{\min}(\Dpik):=\pi^{-,\sharp}_{\min}(\Dpik)=[\pi^{\lalg}(\psi,\bh)-\pi^-_c(\underline{\phi},\bh)]\hookrightarrow \pi_{\min}(\Dpik)=[\pi^{\lalg}(\psi,\bh)-\pi_c(\underline{\phi},\bh)].
				\end{eqnarray}
				Moreover,\;the non-split extension $[\pi^{\lalg}(\underline{\phi},\bh)-\pi_c(\underline{\phi},\bh)]$ or $[\pi_1(\underline{\phi},\bh)-L(\lambda)]$ encodes exactly the information of $\cL(\Dpik)$.\;

				\subsubsection[$\GLN_{3}(\bQ_p)$]{$\GLN_{3}(\bQ_p)$}
				
				Suppose that $\Dpik=[\cR_E(\unr(\phi_1)z^{h_1})-\cR_E(\unr(\phi_2)z^{h_2})-\cR_E(\unr(\phi_3)z^{h_3})]$ with $\phi_i=\alpha p^{i-1}$ for $1\leq i\leq 3$ and some $\alpha\in E^{\times}$ (so that $I_0(\Dpik)=\{1,2\}$).\;Let $\st^{\infty}_3(\lambda)$,\;$v^{\infty}_{P_1}(\lambda)$,\;$v^{\infty}_{P_2}(\lambda)$ and $L(\lambda)$ be the four locally algebraic generalized Steinberg representations of $\GLN_3(\bQ_p)$,\;and we refer to \cite[Introduction]{Dilogarithm} for the same notation of Orlik-Strauch locally analytic representations $C^2_{s_2,1}$,\;$C_{s_2,s_2}$,\;$C^2_{s_2,s_2s_1}$ and $C^2_{s_1,1}$,\;$C_{s_1,s_1}$,\;$C^2_{s_1,s_1s_2}$.\;
				
				If $\Dpik$ is Steinberg (i.e.,\;$S_0(\Dpik)=\{1,2\}$),\;the locally analytic representation $\pi_{\min}(\Dpik)$ is already discussed in \cite{breuil2019ext1},\;\cite{HigherLinvariantsGL3(Qp)}.\;If $\Dpik$ is crystalline (i.e.,\;$S_0(\Dpik)=\emptyset$),\;then $\pi_{\min}(\Dpik)$ has the form (see \cite[Introduction]{HEparaforsemitable}):
				\begin{equation}\label{struGL3CRYS}
					\begindc{\commdiag}[300]
					
					\obj(4,0)[d]{$C^2_{s_2,1}$}
					\obj(5,1)[d1]{$C_{s_2,s_2}$}
					\obj(5,2)[d2]{$C^2_{s_1,s_1s_2}$}
					
					\obj(3,2)[f]{$v^{\infty}_{P_2}(\lambda)$}
					\obj(8,5)[f1]{$v^{\infty}_{P_1}(\lambda)$}
					
					\obj(1,3)[b]{$\st^{\infty}_3(\lambda)$}
					\obj(4,6)[e]{$C^2_{s_1,1}$}
					\obj(5,5)[e1]{$C_{s_1,s_1}$}
					\obj(5,4)[e2]{$C^2_{s_2,s_2s_1}$}
					
					\obj(3,4)[g]{$v^{\infty}_{P_1}(\lambda)$}
					\obj(8,1
					)[g1]{$v^{\infty}_{P_2}(\lambda)$}
					\obj(10,3)[b1]{$\st^{\infty}_3(\lambda)$}
					\mor{b}{d}{}[+1,\solidline]
					\mor{b}{e}{}[+1,\solidline]
					\mor{b}{f}{}[+1,\solidline]
					\mor{b}{g}{}[+1,\solidline]
					\mor{f}{d1}{}[+1,\solidline]
					\mor{f}{d2}{}[+1,\solidline]
					\mor{b1}{d2}{}[+1,\solidline]
					\mor{f1}{d1}{}[+1,\solidline]
					\mor{f1}{e}{}[+1,\solidline]
					\mor{g}{e1}{}[+1,\solidline]
					\mor{g1}{d}{}[+1,\solidline]
					\mor{g}{e2}{}[+1,\solidline]
					\mor{b1}{e2}{}[+1,\solidline]
					\mor{g1}{e1}{}[+1,\solidline]
					\mor{f1}{b1}{}[+1,\solidline]
					\mor{g1}{b1}{}[+1,\solidline]
					\enddc
				\end{equation}
			 It remains to consider the case $S_0(\Dpik)=\{1\}$ needs more discussion (if $S_0(\Dpik)=\{2\}$,\;we consider the dual $\Dpik^{\vee}$).\;The desired locally analytic representation $\pi_{\min}(\Dpik)$ has the following form (which is also an example of the representation $\pi_{\min}^{-,?}(\underline{\phi},\bh)$ for  $?\in \{+,\Delta\}$):
				\begin{equation}\label{struGL3ST2}
					\begindc{\commdiag}[300]
					
					\obj(3,0)[d]{$C^2_{s_2,1}$}
					\obj(5,1)[d1]{$C_{s_2,s_2}$}
					\obj(5,3)[d2]{$C^2_{s_1,s_1s_2}$}
					
					\obj(3,2)[f]{$v^{\infty}_{P_2}(\lambda)$}
					\obj(9,5)[f1]{$v^{\infty}_{P_2}(\lambda)$}
					\obj(9,3)[g2]{$L(\lambda)$}
					\obj(1,3)[b]{$\st^{\infty}_3(\lambda)$}
					\obj(3,5)[e]{$C^2_{s_1,1}$}
					\obj(7,5)[e1]{$C_{s_1,s_1}$}
					\obj(9,1)[e2]{$C^2_{s_2,s_2s_1}$}
					\obj(5,6)[e3]{$C^1_{s_2s_1,1}$}
					\obj(7,0)[e4]{$C^1_{s_1s_2,s_1s_2}$}
					\obj(5,4)[g]{$v^{\infty}_{P_1}(\lambda)$}
					\obj(7,2)[g1]{$v^{\infty}_{P_1}(\lambda)$}
					\mor{b}{d}{}[+1,\solidline]
					\mor{b}{f}{}[+1,\solidline]
					\mor{b}{e}{}[+1,\solidline]
					\mor{f}{d1}{}[+1,\solidline]
					\mor{f}{d2}{}[+1,\solidline]
					\mor{e3}{e}{}[+1,\solidline]
					\mor{e3}{e1}{}[+1,\solidline]
					\mor{e4}{e2}{}[+1,\solidline]
					\mor{e4}{d1}{}[+1,\solidline]
					\mor{g}{e}{``$\cL_{21}$"}[+1,\solidline]
					\mor{g}{e1}{}[+1,\solidline]
					\mor{g1}{g2}{}[+1,\solidline]
					\mor{f1}{e1}{}[+1,\solidline]
					\mor{g1}{d1}{``$\cL_{02}$"}[+1,\solidline]
					\mor{g1}{e2}{}[+1,\solidline]
					\mor{f1}{d}{}[+1,\solidline]
					\enddc
				\end{equation}
				In this case (recall (\ref{fil1})),\;the simple Hodge parameter $\cL_{21}$ is encoded in $\BW(1)\subseteq \BE(1)$,\;which can be seen in the first branch of $\pi_{1}(\Dpik)$.\;The higher parameter $\cL_{02}$ is encoded in $\BW(s_1s_2)\subseteq \BE(s_1s_2)$,\;which can be seen in the second branch of $\pi_{1}(\Dpik)$.\;$\pi_{\min}(\Dpik)$ is obtained by considering several parabolic inductions.\;
				
				We first use the triangulation $\cF_1$.\;Put $D_1:=[\cR_E(\unr(\phi_1)z^{h_1})-\cR_E(\unr(\phi_2)z^{h_2})]\hookrightarrow \Dpik$ and $\Dpik\twoheadrightarrow C_1':=[\cR_E(\unr(\phi_2)z^{h_2})-\cR_E(\unr(\phi_3)z^{h_3})]$.\;By \cite[Remark (1)]{HigherLinvariantsGL3(Qp)},\;we obtain two branches that correspond to the Steinberg $(\varphi,\Gamma)$-module $D_1$ and the special crystalline $(\varphi,\Gamma)$-module $C_1'$,\;i.e.,\;
				\begin{equation}
					\begindc{\commdiag}[300]
					
					\obj(3,2)[d]{$C^2_{s_2,1}$}
					\obj(5,3)[d1]{$C_{s_2,s_2}$}
					
					\obj(3,3)[f]{$v^{\infty}_{P_2}(\lambda)$}
					
					\obj(1,3)[b]{$\st^{\infty}_3(\lambda)$}
					\obj(3,5)[e]{$C^2_{s_1,1}$}
					\obj(7,5)[e1]{$C_{s_1,s_1}$}
					\obj(5,6)[e3]{$C^1_{s_2s_1,1}$}
					\obj(5,4)[g]{$v^{\infty}_{P_1}(\lambda)$}
					\mor{b}{d}{}[+1,\solidline]
					\mor{b}{f}{}[+1,\solidline]
					\mor{b}{e}{}[+1,\solidline]
					\mor{f}{d1}{}[+1,\solidline]
					\mor{e3}{e}{}[+1,\solidline]
					\mor{e3}{e1}{}[+1,\solidline]
					\mor{g}{e}{}[+1,\solidline]
					\mor{g}{e1}{}[+1,\solidline]
					\enddc
				\end{equation}
				By using the triangulation $\cF_{s_1s_2}$ and $\Dpik\twoheadrightarrow C_1:=[\cR_E(\unr(\phi_1)z^{h_2})-\cR_E(\unr(\phi_2)z^{h_3})]$,\;we obtain the second locally analytic representation
				\begin{equation}
					\begindc{\commdiag}[300]
					
					\obj(5,1)[d1]{$C_{s_2,s_2}$}
					\obj(5,3)[d2]{$C^2_{s_1,s_1s_2}$}
					
					\obj(3,2)[f]{$v^{\infty}_{P_2}(\lambda)$}
					
					\obj(1,3)[b]{$\st^{\infty}_3(\lambda)$}
					\obj(3,4)[e]{$C^2_{s_1,1}$}
					\obj(9,1)[e2]{$C^2_{s_2,s_2s_1}$}
					\obj(7,0)[e4]{$C^1_{s_1s_2,s_1s_2}$}
					\obj(7,2)[g1]{$v^{\infty}_{P_1}(\lambda)$}
					\obj(9,3)[g2]{$L(\lambda)$}
					\mor{b}{f}{}[+1,\solidline]
					\mor{b}{e}{}[+1,\solidline]
					\mor{f}{d1}{}[+1,\solidline]
					\mor{f}{d2}{}[+1,\solidline]
					\mor{e4}{e2}{}[+1,\solidline]
					\mor{e4}{d1}{}[+1,\solidline]
					\mor{g1}{g2}{}[+1,\solidline]
					\mor{g1}{d1}{}[+1,\solidline]
					\mor{g1}{e2}{}[+1,\solidline]
					\enddc
				\end{equation}
				This representation encodes the information $\cL_{02}$.\;Combining these two representations,\;we obtain a locally analytic representation $\pi_{\min}^{-}(\Dpik)$ that encodes the information of $\cL_{02}$ and $\cL_{21}$:
				\begin{equation}
					\begindc{\commdiag}[300]
					
					\obj(3,0)[d]{$C^2_{s_2,1}$}
					\obj(5,1)[d1]{$C_{s_2,s_2}$}
					\obj(5,3)[d2]{$C^2_{s_1,s_1s_2}$}
					
					\obj(3,2)[f]{$v^{\infty}_{P_2}(\lambda)$}
					
					\obj(1,3)[b]{$\st^{\infty}_3(\lambda)$}
					\obj(3,5)[e]{$C^2_{s_1,1}$}
					\obj(7,5)[e1]{$C_{s_1,s_1}$}
					\obj(9,1)[e2]{$C^2_{s_2,s_2s_1}$}
					\obj(5,6)[e3]{$C^1_{s_2s_1,1}$}
					\obj(7,0)[e4]{$C^1_{s_1s_2,s_1s_2}$}
					\obj(5,4)[g]{$v^{\infty}_{P_1}(\lambda)$}
					\obj(7,2)[g1]{$v^{\infty}_{P_1}(\lambda)$}
					\obj(9,3)[g2]{$L(\lambda)$}
					\mor{b}{d}{}[+1,\solidline]
					\mor{b}{f}{}[+1,\solidline]
					\mor{b}{e}{}[+1,\solidline]
					\mor{f}{d1}{}[+1,\solidline]
					\mor{f}{d2}{}[+1,\solidline]
					\mor{e3}{e}{}[+1,\solidline]
					\mor{e3}{e1}{}[+1,\solidline]
					\mor{e4}{e2}{}[+1,\solidline]
					\mor{e4}{d1}{}[+1,\solidline]
					\mor{g}{e}{}[+1,\solidline]
					\mor{g}{e1}{}[+1,\solidline]
					\mor{g1}{g2}{}[+1,\solidline]
					\mor{g1}{d1}{}[+1,\solidline]
					\mor{g1}{e2}{}[+1,\solidline]
					\enddc
				\end{equation}
				This representation already determines $\Dpik$.\;We are able to add some extra locally algebraic constituents by using the methods in \cite{HigherLinvariantsGL3(Qp)}.\;We first apply \cite[Proposition 3.49]{HigherLinvariantsGL3(Qp)} to our case.\;Keep the notation in the proof of \cite[Proposition 3.49]{HigherLinvariantsGL3(Qp)}.\;In this case,\;we obtain that $\cL_{\mathrm{aut}}(\Dpik:D_1)\cap \ker(\mathrm{pr})$ is $1$-dimensional (as the extension $\cR_E(\unr(\psi_2)z^{h_2})-\cR_E(\unr(\psi_3)z^{h_3})$ is crystalline).\;Note that $\ker(\mathrm{pr}_1)+\ker(\mathrm{pr}_2)$ is equal to the whole extension group,\;and $\ker(\mathrm{pr}_1)\cap \ker(\mathrm{pr}_2)=\ker(\mathrm{pr})\subseteq \cL_{\mathrm{aut}}(\Dpik:D_1)$.\;We have to study the intersection of $\cL_{\mathrm{aut}}(\Dpik:D_1)\cap \ker(\mathrm{pr}_i)$ for $i=1,2$.\;If $\cL_{\mathrm{aut}}(\Dpik:D_1)\cap \ker(\mathrm{pr}_i)=\ker(\mathrm{pr})$ for $i=1$ or $2$,\;then $\cL_{\mathrm{aut}}(\Dpik:D_1)$ is canonical (i.e.,\;does not depend on any parameter),\;this is impossible.\;Therefore,\;we get $\big(\cL_{\mathrm{aut}}(\Dpik:D_1)/\ker(\mathrm{pr})\big)\cap \big(\ker(\mathrm{pr}_i)/\ker(\mathrm{pr})\big)=0$ for $i=1,2$.\;Then we obtain the following locally analytic representation:
				\begin{equation}
					\begindc{\commdiag}[300]
					
					\obj(3,2)[d]{$C^2_{s_2,1}$}
					
					\obj(3,3)[f]{$v^{\infty}_{P_2}(\lambda)$}
					\obj(9,4)[f1]{$v^{\infty}_{P_2}(\lambda)$}
					
					\obj(1,3)[b]{$\st^{\infty}_3(\lambda)$}
					\obj(3,5)[e]{$C^2_{s_1,1}$}
					\obj(7,5)[e1]{$C_{s_1,s_1}$}
					\obj(5,6)[e3]{$C^1_{s_2s_1,1}$}
					\obj(5,4)[g]{$v^{\infty}_{P_1}(\lambda)$}
					\mor{b}{d}{}[+1,\solidline]
					\mor{b}{f}{}[+1,\solidline]
					\mor{b}{e}{}[+1,\solidline]
					\mor{e3}{e}{}[+1,\solidline]
					\mor{e3}{e1}{}[+1,\solidline]
					\mor{g}{e}{}[+1,\solidline]
					\mor{g}{e1}{}[+1,\solidline]
					\mor{f1}{e1}{}[+1,\solidline]
					\mor{f1}{d}{}[+1,\solidline]
					\enddc
				\end{equation}
				On the other hand,\;we can also apply the proof of \cite[Proposition 3.49]{HigherLinvariantsGL3(Qp)} to $C_1$,\;and consider the higher $\cL$-invariant $\cL_{\mathrm{aut}}(\Dpik:C_1)$,\;we obtain the following representation:
				\begin{equation}
					\begindc{\commdiag}[300]
					
					\obj(3,2)[d]{$C^2_{s_2,1}$}
					\obj(5,3)[d1]{$C_{s_2,s_2}$}
					
					\obj(3,3)[f]{$v^{\infty}_{P_2}(\lambda)$}
					
					\obj(1,3)[b]{$\st^{\infty}_3(\lambda)$}
					\obj(3,4)[e]{$C^2_{s_1,1}$}
					\obj(5,4)[g]{$v^{\infty}_{P_1}(\lambda)$}
					\obj(7,3)[g1]{$v^{\infty}_{P_1}(\lambda)$}
					
					\mor{b}{d}{}[+1,\solidline]
					\mor{b}{f}{}[+1,\solidline]
					\mor{b}{e}{}[+1,\solidline]
					\mor{f}{d1}{}[+1,\solidline]
					
					\mor{g}{e}{}[+1,\solidline]
					
					\mor{g1}{d1}{}[+1,\solidline]
					
					\enddc
				\end{equation}
				We actually obtain the locally analytic representation $\pi_{\min}(\Dpik)$ that determines $\Dpik$.

				\section{$\mathrm{II}$: \texorpdfstring{$\sL$}{L}-invariants and Colmez-Greenberg-Stevens formula}

				\begin{pro}[Proof of the second assertion in Proposition \ref{profordimofdefor}]\label{imageofkappaproof} 
					$\mathrm{Im}(\kappa^{\circ}_{u})=\mathrm{Im}(\kappa^{0}_{u})=\homo_u(\bT(\bQ_p),E)$.
				\end{pro}
				
				\begin{proof}
					We prove this proposition by induction on $n$. Assume this proposition holds for $N_1:=[R_{u,1}-\cdots-E_{u,n-1}]$. Suppose that $(q,n)\in R^+_u$ for some unique $1\leq q\leq s-1$. We need to define $\cL_{q,n}^u$ and prove this proposition for $\Dpik$. Put
					\begin{equation}
						\begin{aligned}
							&N_2:=[R_{u,1}-\cdots-R_{u,q}]\hookrightarrow N_1,\quad &N_2':=[R_{u,q}-R_{u,q+1}-\cdots-R_{u,n-1}],\\
							&N_3':=[R_{u,q+1}-\cdots-R_{u,n-1}],\quad &D''=[R_{u,q}-R_{u,q+1}-\cdots-R_{u,n}]=[N_2'-R_{u,n}].
						\end{aligned}
					\end{equation}
					Note that $N_3'=N_2'/R_{u,q}$ and $N_1 \twoheadrightarrow N_2'$.\;These maps induce the following commutative diagram of cup products:
					\begin{equation}
						\xymatrix{
							\ext^1(R_{u,n},N'_2) \ar@{=}[d] &\times & \ext^1(N'_2,N'_3)\ar[r]^{\cup} & \ext^2(R_{u,n},N'_3)=0\\
							\ext^1(R_{u,n},N_2') \ar@{=}[d] &\times & \ext^1(N_2',N_2')\ar[r]^{\cup} \ar[u]^{\kappa_2} & \ext^2(R_{u,n},N_2') \ar@{=}[d] \ar[u]\\
							\ext^1(R_{u,n},N'_2) \ar[d]^{g_1}  &\times & \ext^1(N'_2,R_{u,q})\ar[r]^{\cup} \ar[u]^{\kappa_1'} & \ext^2(R_{u,n},R_{u,q}) \ar@{=}[d] \\
							\ext^1(R_{u,n},N'_3)  &\times & \ext^1(N'_3,R_{u,q})\ar[r]^{\cup} \ar[u]^{\kappa_1} & \ext^2(R_{u,n},R_{u,q}). }
					\end{equation}
					 Let $\cL_{\mathrm{FM}}(D'':N_2'):=(E[D''])^{\perp}$ (via the cup-product in the second row), and let $l_{\mathrm{FM}}(D'':N_2'):=(E[D''])^{\perp}\subseteq \ext^1(N_2',R_{u,q})$ (via the cup-product in the third row).
					
					We describe these subspaces explicitly. It is easy to see that $\ext^2(N'_2,R_{u,q})=0$ so that $\kappa_2$ is surjective. On the other hand, $\kappa_1'$ is injective since $\homo(N_2',N_2')\xrightarrow{\sim }\homo(N_2',N_3')$ is an isomorphism. Therefore, we have a short exact sequence of $E$-vector spaces:
					\[0\rightarrow l_{\mathrm{FM}}(D'':N_2')\rightarrow \cL_{\mathrm{FM}}(D'':N_2')\rightarrow \ext^1(N'_2,N'_3)\rightarrow 0.\]
					It remains to study $l_{\mathrm{FM}}(D'':N_2')$. Consider the following commutative diagram:
					\begin{equation}\label{N2N3ANDRANK1}
						\xymatrix{
							\ext^1(R_{u,n},R_{u,q}) \ar[d]^{g_1'} &\times & \ext^1(R_{u,q},R_{u,q})\ar[r]^{\cup}  & \ext^2(R_{u,n},R_{u,q}) \ar@{=}[d] \\
							\ext^1(R_{u,n},N'_2) \ar[d]^{g_1}  &\times & \ext^1(N'_2,R_{u,q})\ar[r]^{\cup} \ar[u]^{\kappa_1''} & \ext^2(R_{u,n},R_{u,q}) \ar@{=}[d] \\
							\ext^1(R_{u,n},N'_3)  &\times & \ext^1(N'_3,R_{u,q})\ar[r]^{\cup} \ar[u]^{\kappa_1} & \ext^2(R_{u,n},R_{u,q}). }
					\end{equation}
					We define
					\begin{equation}\label{dfoforimageLinvariants}
						\cL_{q,n}^u:=\kappa_1''(l_{\mathrm{FM}}(D'':N_2'))\hookrightarrow \ext^1(R_{u,q},R_{u,q})\cong \homo(\bQ_p^{\times},E)
					\end{equation}
			Since $\homo(N'_2,R_{u,q})=0$, we get that $\ker(\kappa_1)\cong \homo(R_{u,q},R_{u,q})\cong E$, so that $\dim_E\mathrm{Im}(\kappa_1)=r_q\rk(N_3')-1$. Since $\ext^2(N'_3,R_{u,q})=0$, $\kappa_1''$ is surjective. Since $\homo(R_{u,n},N'_3)=0$, $g_1'$ is injective. Observe the following commutative diagram:
					\begin{equation}
						\xymatrix{
							\ext^1(R_{u,n},N'_3) \ar@{=}[d]  &\times & \ext^1(N'_3,R_{u,n}(\epsilon))\ar[r]^{\cup} & \ext^2(R_{u,n},R_{u,n}(\epsilon)) \\
 							\ext^1(R_{u,n},N'_3)  &\times & \ext^1(N'_3,R_{u,q})\ar[r]^{\cup} \ar[u]^{\eta} & \ext^2(R_{u,n},R_{u,q})\ar[u]^{\eta'}_{\sim}, }
					\end{equation}
					where $\eta$ and $\eta'$ are induced by the map $R_{u,q}\hookrightarrow R_{u,n}(\epsilon)$ of $(\varphi,\Gamma)$-modules. The top cup-product is perfect by the Tate duality. On the other hand, by \cite[Lemma 5.1.1]{breuil2020probleme}, we have
					\begin{equation}
						\begin{aligned}
							& \homo((N'_3)^{\vee}\otimes_{\cR_E}R_{u,q},(N'_3)^{\vee}\otimes_{\cR_E}\cR_E(\delta\epsilon))\\
							\cong\;& \hH^0(\gal_{\bQ_p},W_{\dr}^+((N'_3)^{\vee}\otimes_{\cR_E}\cR_E(\delta\epsilon))/W_{\dr}^+((N'_3)^{\vee}\otimes_{\cR_E}R_{u,q})).
						\end{aligned}
					\end{equation}
					Then we obtain $\dim_E\ker(\eta)=\rk(N'_3)$ and thus $\eta=0$. Therefore,
					\[\ext^1(N'_3,R_{u,q})\subseteq (Eg_1([D'']))^{\perp}, \kappa_1(\ext^1(N'_3,R_{u,q}))\subseteq l_{\mathrm{FM}}(D'':N_2').\]
					By definition and compring dimensions, we have
					\[\cL_{q,n}^u= l_{\mathrm{FM}}(D'':N_2')/\kappa_1(\ext^1(N'_3,R_{u,q})). \]
					Since the bottom cup-product in
					(\ref{N2N3ANDRANK1}) is non-degenerate when replacing the middle term with the subspace $\ext^{1}(R_{u,q},R_{u,q})$, we see that $\ext^{1}(R_{u,q},R_{u,q})\subseteq \ext^1(R_{u,n},N'_3)^{\perp}$. On the other hand, $E[D'']=E(e_1+e_2)\in \ext^1(R_{u,n},N'_2)$ for some vectors
					$e_1\in \ext^1(R_{u,n},R_{u,q})$ and $e_2\in \ext^1(R_{u,n},N'_3)$.\;Recall that the Colmez--Greenberg--Stevens formula in \cite[Theorem 2.7 ]{HigherLinvariantsGL3(Qp)} implies: for the tuple $(\widetilde{N_2'},\psi_n)\in \ext^{1}(N'_2,N_2')\times \homo(\bQ_p^{\times},E)$, there exists a deformation $[\widetilde{N'_2}-R_{u,n}\otimes_{\cR_{E}}\cR_{E[\epsilon]/\epsilon^2}(1+\psi_n\epsilon))$ of $D''$ over $E[\epsilon]/\epsilon^2$ if and only if
					$\widetilde{N_2'}\otimes_{\cR_{E[\epsilon]/\epsilon^2}}(\cR_{E[\epsilon]/\epsilon^2}(1-\psi_n\epsilon))\in\cL_{\mathrm{FM}}(D'':N_2')$. Since all elements in $\ext^{1}(R_{u,q},R_{u,q})$ are orthogonal to $e_2$, we obtain that $\psi_q-\psi_{n}\in \cL_{q,n}^u$ by the definition $\cL_{q,n}^u:=\kappa_1''(l_{\mathrm{FM}}(D'':N_2'))\cap \ext^{1}(R_{u,q},R_{u,q})$.\;This completes the proof.
				\end{proof}


				\bibliographystyle{plain}
				
				\printindex
			\end{document}